\pdfoutput=1
\documentclass{amsart}

\usepackage{mathtools}
\usepackage{amsmath}
\usepackage{amsfonts}
\usepackage{amssymb}
\usepackage{graphicx}
\usepackage{graphicx, xcolor}              % to include figures
\usepackage{amsmath}               % great math stuff
\usepackage{amsfonts}              % for blackboard bold, etc
\usepackage{amsmath}
\usepackage{amsthm}  
 \usepackage{float} 
 \usepackage{MnSymbol}
 \usepackage{tikz}
 \usepackage{enumerate}
\usetikzlibrary{shapes.geometric}
\usetikzlibrary{arrows.meta,arrows}

\usepackage{caption}
\usepackage{wrapfig}
\usepackage{subcaption}
\usepackage{color}

\usepackage{mathrsfs}
\usepackage{url}
\usepackage{cite}

\newtheorem{Example*}{Examples}
\newtheorem{note}{Note}
\newtheorem{theorem}{Theorem}

\theoremstyle{plain}

\newtheorem{conjecture}{Conjecture}
\newtheorem{corollary}{Corollary}

\newtheorem{definition}{Definition}
\newtheorem{example}{Example}
\newtheorem*{examples}{Examples}

\newtheorem{lemma}{Lemma}

\newtheorem{proposition}{Proposition}
\newtheorem{remark}{Remark}

\numberwithin{equation}{section}

\newtheorem*{Remark}{Remarks}
\begin{document}

\title{ Mock Alexander Polynomials}
\author{Neslihan G{\"u}g{\"u}mc{\"u}}
\author{Louis H.Kauffman}

\address{Neslihan G{\"u}g{\"u}mc{\"u}:  Department of Mathematics, Izmir Institute of Technology, Urla Izmir 35430, Turkey}
\address{Louis H.Kauffman:Department of Mathematics, Statistics and Computer
Science, University of Illinois at Chicago, 851 South Morgan St., Chicago
IL 60607-7045, U.S.A.
}
\email{neslihangugumcu@iyte.edu.tr} \email{kauffman@math.uic.edu; loukau@gmail.com}

\begin{abstract}
In this paper, we construct mock Alexander polynomials for starred links and linkoids in surfaces. These polynomials are defined as specific sums over states of link or linkoid diagrams that satisfy $f=n$, where $f$ denotes the number of regions and $n$ denotes the number of crossings of diagrams.

%with an equal number of crossings and regions.
 %the property $f=n$ where $f$ denotes the number of faces and $n$ denotes the number of four-valent vertices in the universe. %We also generalize the \textit{Clock Theorem} of the Formal Knot Theory to knotoids that guarantees the existence of a lattice of `clock states' for a  knotoid diagram. 
\end{abstract}

\maketitle

\section{Introduction} \label{sec:introduction}
\subsection{Background and Motivation}

The Alexander polynomial was defined by J. W. Alexander in
his 1928 paper \cite{Alex}, as the determinant of a matrix associated with 
an oriented link diagram. Later in the paper, Alexander describes how the matrix whose determinant yields the polynomial is related to the fundamental group of the complement of the knot or link.  It is sufficient here to say that the Alexander matrix is a presentation of the abelianization of the commutator subgroup of the fundamental group as a module over the group ring of the integers. Since Alexander's time, this relationship with the fundamental group has been understood very well.
For more information about this point of view, see  \cite{Fox1,Fox2, Rolfsen, OnKnots}. 

Alexander's notation for generating the matrix associated to an oriented link diagram is as follows. At each crossing two dots are placed just to the left of the undercrossing arc, one before and one after the overcrossing arc at the crossing. Here one views the crossing so that the under-crossing arc is vertical and the over-crossing arc is horizontal. See Figure~\ref{Figure 1}. 

Four regions meet locally at a given crossing. Letting the regions be labeled by A, B, C, D, as shown in Figure \ref{Figure 1}.,Alexander associates the 
equation $$xA - xB + C - D = 0$$ to that crossing. Here $A,B,C,D$ proceed cyclically around the crossing, starting at the top dot. In this way the two
regions containing the dots give rise to the two occurrences of $x$ in the equation. If some of the regions are the same at the crossing, then the equation is
simplified by that equality. For example, if $A=D$ then the equation becomes $xA - xB + C -A = 0.$ Each crossing in a diagram $K$ gives an equation 
involving the regions of the diagram. Alexander associates a matrix $M_{K}$ whose rows correspond to  crossings of the diagram, and whose columns correspond
to regions of the diagram. Each nodal equation gives rise to one row of the matrix where the entry for a given column is the coefficient of that 
column (understood as designating a region in the diagram) in the given equation. If $R$ and $R'$ are adjacent regions, let $M_{K}[R,R']$ denote the
matrix obtained by deleting the corresponding columns of $M_{K}.$ Finally, define the {\em Alexander polynomial} $\Delta_{K}(x)$ by the formula
$$\Delta_{K}(x) \dot{=} Det(M_{K}[R,R']).$$

 \begin{figure}
     \begin{center}
     \begin{tabular}{c}
     \includegraphics[width=5cm]{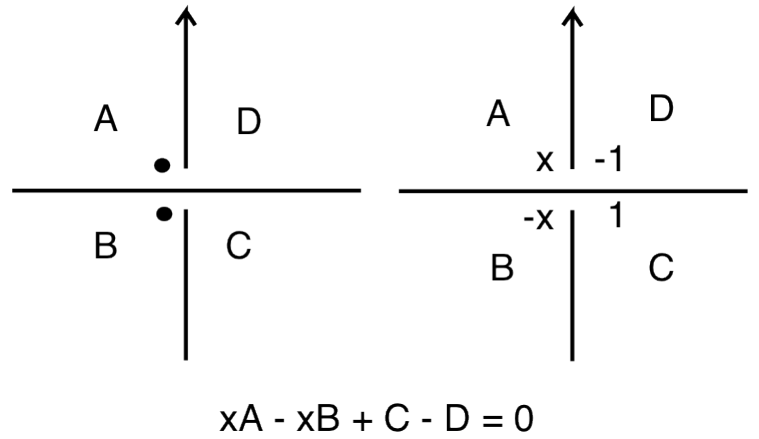}
     \end{tabular}
     \caption{Alexander Labeling.}
     \label{Figure 1}
\end{center}
\end{figure}

\noindent The notation $ A \dot{=} B$ means that $A = \pm x^{n}B$ for some integer $n.$ Alexander proves that his polynomial is well-defined, independent of
the choice of adjacent regions and invariant under the Reidemeister moves  up to $\dot{=}$. In Figure~\ref{Figure 2} we show  the calculation of the Alexander polynomial of the trefoil knot using this method.

 \begin{figure}
     \begin{center}
     \begin{tabular}{c}
     \includegraphics[width=5cm]{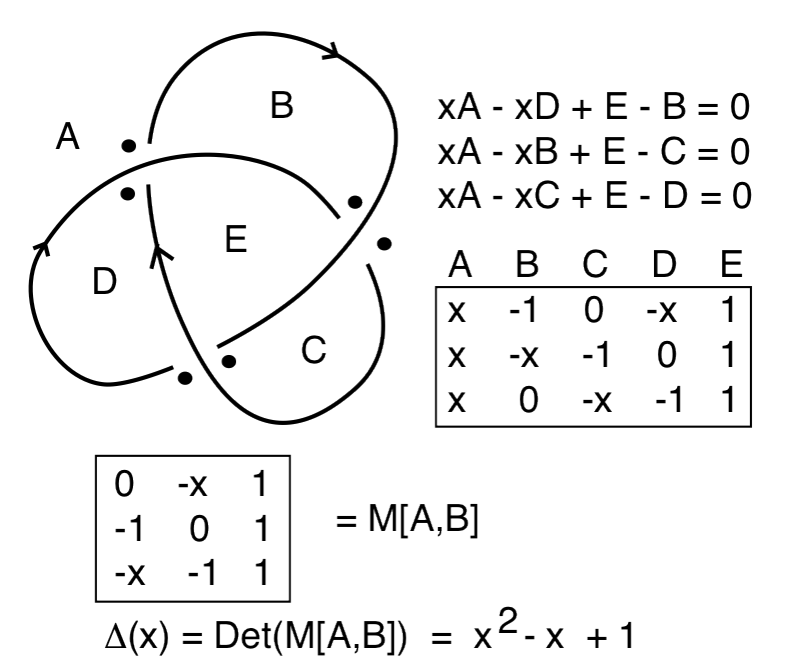}
     \end{tabular}
     \caption{The Alexander Polynomial. }
     \label{Figure 2}
\end{center}
\end{figure}

\noindent In this figure we show the diagram of the knot, the labelings and the resulting full matrix and the square matrix resulting from deleting
two columns corresponding to a choice of adjacent regions. Computing the determinant, we find that the the Alexander polynomial of the trefoil knot is given
by the equation $\Delta \dot{=} x^2 - x + 1.$ 

In 1983, the second author gave a reformulation of the Alexander polynomial as a state sum in his book titled Formal Knot Theory \cite{FKT}. The state sum reformulation  is based on combinatorial configurations on a knot or link diagram that are directly related to the expansion of the determinant that defines the Alexander 
polynomial. Here is a brief description of this reformulation.

Given an $n \times n$  square matrix $M =[M_{ij}]$, we consider the expansion formula for the determinant of $M:$
\\

$$Det(M) = \sum_{\sigma \in S_{n}}sgn(\sigma) \prod_{i=1}^{n} M_{i\sigma(i)}.$$
\\

Here the sum runs over all permutations of the index set $\{ 1,2, \ldots ,n \}$ and $sgn(\sigma)$ denotes the sign of a given permutation $\sigma.$
In terms of the matrix, each product corresponds to a {\em choice} by each column of a single row such that each row is chosen exactly once. The order of 
rows chosen by the columns (taken in standard order) gives the permutation whose sign is calculated.

Consider our description of Alexander's determinant as given above. Each crossing is labeled with Alexander's dots so that we know that 
the four local quadrants at a crossing are each labeled with $x$, $-x$, $1$ or $-1.$ The matrix has one row for each crossing and one column for each region.
Two columns corresponding to adjacent regions $A$ and $B$ are deleted from the full matrix to form the matrix $M[A,B]$, and we have the Alexander polynomial
$\Delta_{K}(X) \dot{=} Det(M[A,B]).$

In the Alexander determinant expansion the {\em choice} of a row by a column corresponds to {\em a region choosing a crossing} in the link diagram.
The only crossings that a region can choose giving a non-zero term in the determinant are the crossings in the boundary of the given region. Thus the terms in the
expansion of $Det(M[A,B])$ are in one-to-one correspondence with decorations of the flattened link diagram (i.e. we ignore the over and under crossing
structure) where each region (other than the two deleted regions corresponding to the two deleted columns in the matrix) labels one of its crossings. We call
these  labeled flat diagrams the {\em states} of the original link diagram. See Figure~\ref{Figure 3} for a list of the states of the trefoil knot.
In this figure we show the states and the corresponding matrix forms with columns choosing rows that correspond to each state.
 \begin{figure}[H]
     \begin{center}\includegraphics[width=.5\textwidth]{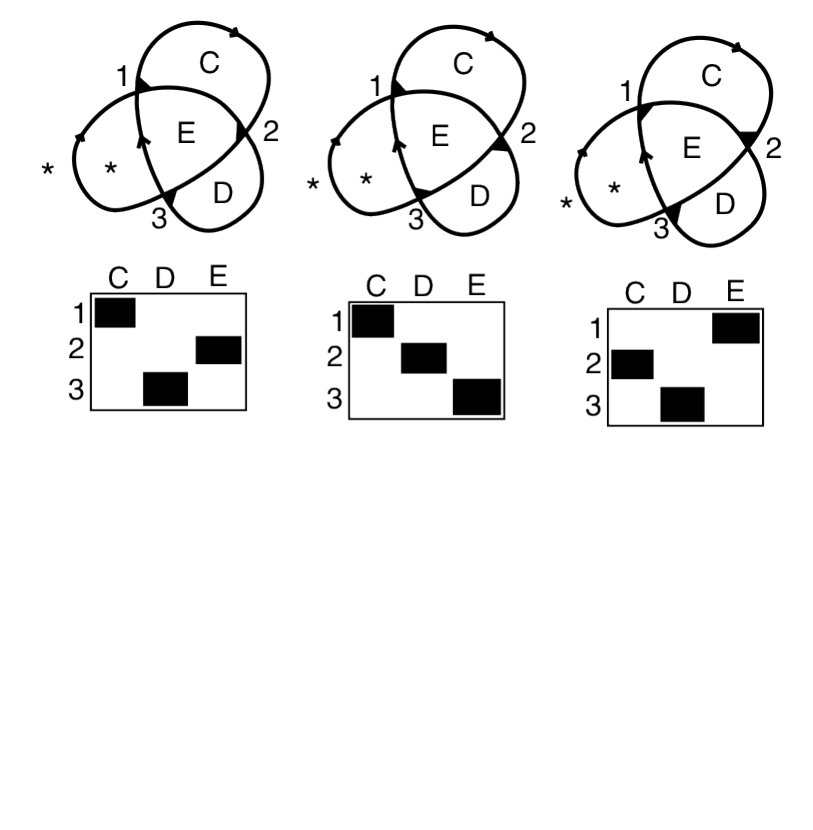}
     \vspace{-3cm}
  \caption{States with Markers}
     \label{Figure 3}
\end{center}
\end{figure}
 %We call such a flattened diagram a {\em link universe} or just {\em universe} for short. Such a diagram is a graph in the plane with four edges
%incident to each vertex. The vertex carries no information about under or over crossings of curves as in a knot or link diagram.  Thus a universe is a $4$-regular planar graph. We say that a universe is {\em connected} if
%it is connected as a planar graph. On the other hand, we say that a link diagram has $k$ components if it represents an embedding of $k$ circles in three
%dimensional space. One counts the number of components of a link diagram by walking along the diagram and crossing at each vertex, counting the number of
%cycles needed to use all the edges in the diagram. In the corresponding flattened diagram the operation of crossing at a vertex means that, at a vertex, one
%chooses to continue the walk along the unique edge that is {\em not} adjacent to the edge one is traversing. The planar embedding of the graph defines this
%adjacency. Thus we can speak of the number of {\em link components of a universe}. This number can be greater than one %even when the universe is
%connected. Two circles intersecting transversely in two points form a connected universe with two link components. Note that given a universe $U$ with $n$
%nodes there are $2^{n}$ possible link diagrams that can be made from $U$ by choosing a crossing at each node.
%\bigbreak

At this point we have almost a full combinatorial description of Alexander's determinant. The only thing missing is the permutation signs.  Call a state marker a {\em black hole} if it labels a quadrant where both oriented segments point toward the vertex of the flattened link diagram. See Figure~\ref{Figure 4} for an illustration of  a black hole.
\begin{figure}[H]
     \begin{center}
     \begin{tabular}{c}
     \includegraphics[width=1.5cm]{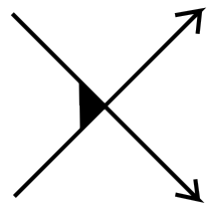}
     \end{tabular}
     \caption{A black hole.}
     \label{Figure 4}
\end{center}
\end{figure}
\bigbreak

\noindent Let $S$ be a state of the diagram K. Consider the parity $$(-1)^{b(S)}$$ where $b(S)$ is the number of black holes in
the state
$S$. Then it turns out that  up to one global sign $\epsilon$ depending on the ordering of nodes and regions, we have $$(-1)^{b(S)} = \epsilon \,\,
sgn(\sigma(S))$$ where 
$\sigma(S)$ is the permutation of crossings induced by the choice of ordering of the regions of the state. This gives a purely diagrammatic access to 
the sign of a state and allows us to write
$$\Delta_{K}(x) \dot{=} \sum_{S}<K|S>(-1)^{b(S)}$$ where $S$ runs over all states of the diagram for a given choice of deleted
adjacent regions, and $<K|S>$ denotes the product of the Alexander's crossing labels at the quadrants indicated by the state labels in the state $S.$ Here we show the contributions of
each state to a product of terms and in the polynomial we have followed the state summation by taking into account the number of black holes in each state. 
The most mysterious thing about
this state sum is the agreement of the permutation signs for the determinant with the black hole parity signs. The proof of this correspondence follows from the Clock Theorem
proved in Formal Knot Theory \cite{FKT}.
%We call $<K|S>$ the {\em product of the vertex weights}.
Thus we have a precise reformulation of the Alexander polynomial as a state summation.
\bigbreak

In Figure~\ref{Figure 5} we illustrate the calculation of the Alexander polynomial of the trefoil knot using this state summation.

 \begin{figure}[H]
     \begin{center}
     \begin{tabular}{c}
     \includegraphics[width=5cm]{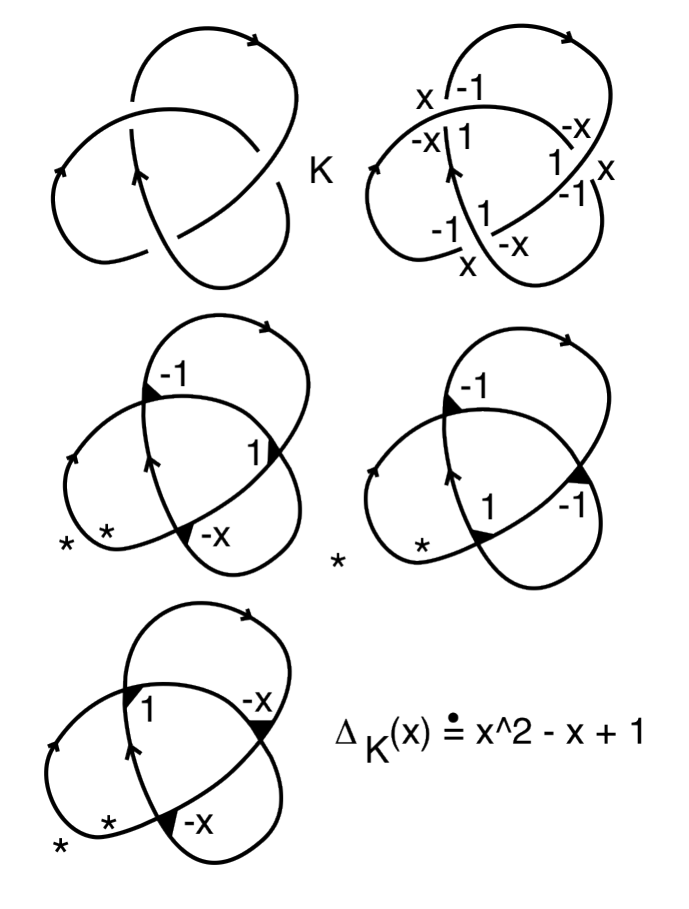}
     \end{tabular}
     \caption{State-sum calculation of Alexander polynomial}
     \label{Figure 5}
\end{center}
\end{figure}

% By the Clock Theorem, the state sum description of the Alexander polynomial of a knot corresponds to the permanent of the matrix $\tilde{M}_{K}[A,B]$ whose entries are the region labels in $(\pm s)^i$, $i=0, -1,1$.

\subsection{Alexander-Conway Polynomial}

We can change the vertex weights of the state sum to the form shown in Figure \ref{fig:WWintro}, letting $z = W - W^{-1}.$ This choice incorporates the sign for the black hole count directly into the weights and gives a direct state summation model for the Alexander-Conway polynomial of an oriented link diagram satisfying the Conway Skein relation
$\nabla_{K_{+}} - \nabla_{K_{-}} = z \nabla_{K_{0}}$ for any triple of oriented link diagrams $K_{+}, K_{-}, K_{0}$ that differ from each other at a unique crossing as shown in Figure \ref{fig:skeinintro}.

\begin{figure}[H]
\centering
\includegraphics[width=.5\textwidth]{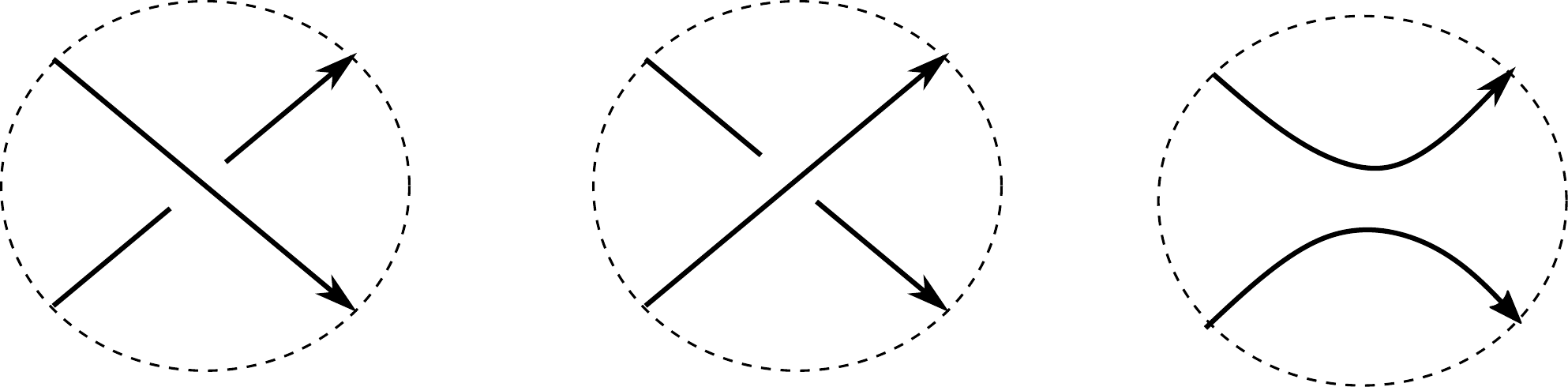}
\caption{ $K_{+}, K_{-}$   and $K_{0}$, respectively from left to right}
\label{fig:skeinintro}
\end{figure} 

 \begin{figure}[H]
     \begin{center}
     \begin{tabular}{c}
     \includegraphics[width=5cm]{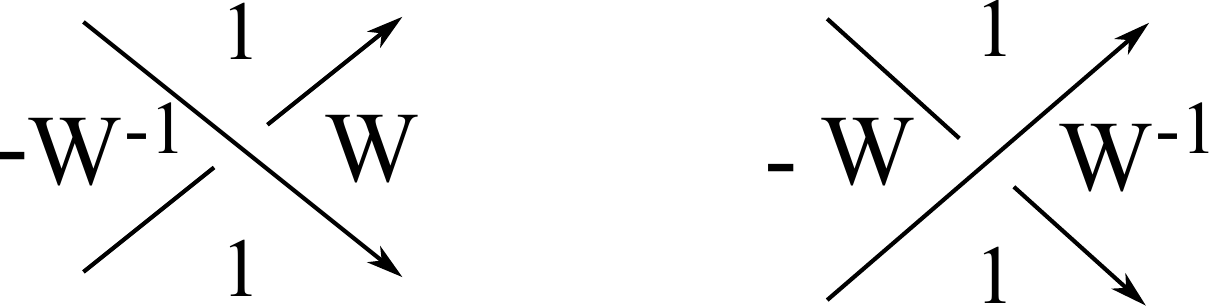}
     \end{tabular}
     \caption{New labeling at crossings}
     \label{fig:WWintro}
\end{center}
\end{figure}
 In fact, the Alexander-Conway polynomial of an oriented knot diagram is given as the permanent of the matrix associated to the knot diagram with this new labeling \cite{FKT}. The state-sum formula for the Alexander-Conway polynomial is
 
$$\nabla_{K}(W) = \sum_{S}<K|S>$$ where $S$ runs over all states of $K$.

Throughout this paper we shall (for invariant polynomials) use the vertex weights in Figure \ref{fig:WWintro} and we will discuss the many invariants that arise using the principle that there should be an equal number of regions  and crossings available for state markers so that each state is a one-to-one correspondence between available regions and crossings.
The key discovery of this paper is that this principle yields many new invariants beyond the original Alexander-Conway polynomial.

\subsection{Organization of the paper}
Now we briefly present the organization of the paper. Section \ref{sec:prem} covers preliminary notions required throughout the paper. In Section \ref{sec:decorated} we study admissable diagrams that have equal number of regions (faces in the underlying universe) and crossings (four-valent vertices in the underlying universe),  after a discussion in Section \ref{sec:Euler} on  the relationship between the number of regions and crossings of a link or a linkoid diagram that lies in a closed, connected and orientable (c.c.o.) surface.

 In Section \ref{sec:starred} we introduce starred link and linkoid diagrams that are admissable diagrams endowed with a star decoration either on their regions or crossings. Starred link and linkoid diagrams are considered up to star equivalence that is a restricted isotopy relation in the surface they lie that is defined far away from starred regions or crossings. 

In Section \ref{sec:state-sum} we generalize the ideas of Formal Knot Theory. We introduce a state-sum polynomial in variables $W, B$ for starred links and linkoids lying in a surface in Section \ref{sec:map}. In Section \ref{sec:permanent} we show that the state-sum polynomial admits a matrix permanent formulation. In Section \ref{sec:invariance} we exhibit the invariance condition on the variables  of the state-sum polynomial up to star equivalence, and obtain invariants of starred links or linkoids that we call  \textit{mock Alexander polynomials}. We also observe that the induced polynomial is an invariant of knotoids in $S^2$ when any knotoid diagram is considered to be endowed with a star at its region adjacent to its tail.  In Section \ref{sec:symmetry} we study the behavior of mock Alexander polynomials of knotoids in $S^2$ under certain symmetries such as reversion, mirror symmetry and starred region exchange. Section \ref{sec:skein} is devoted to te discussion of skein relations that mock Alexander polynomials satisfy. We show that the mock Alexander polynomial of a starred link in $S^2$ satisfies the skein relation at a non-separating crossing. We also show certain variations of the skein relation are satisfied by the mock Alexander polynomial of a starred linkoid in $S^2$.

 In Section \ref{sec:more} we  introduce  handle attaching techniques to obtain link diagrams in genus two surfaces that satisfy the equality $f= n$. We define the mock Alexander polynomial for admissable link diagrams that lie in a c c.o. surface of genus two that is obtained by the handle attaching techniques introduced. In this section, we also study the mock Alexander polynomial of  the virtual closure of a knotoid in $S^2$. 

Section \ref{sec:discussion} is a brief discussion of the states in terms of paths on the diagram, and sets the stage for a next paper, generalizing the Clock Theorem of the Formal Knot Theory. The appendix (Section \ref{sec:appendix}) derives the most general vertex weights and indicates a $2$-variable polynomial that we will explore in a subsequent paper.

%\section{Knotoids and their states}\label{sec:state}
\section{Linkoids and universes}\label{sec:prem}

We begin with necessary definitions.

\begin{definition}\normalfont 
A \textit{linkoid diagram} $L$ in a closed, connected, oriented surface of genus $g$  is an immersion of a number of unit intervals and circles in  the surface. The immersion $L$ has a finite number of double points, each of which is endowed with an under or over crossing information and called a \textit{crossing} of $L$. The points corresponding to $L(0)$ and  $L(1)$ are all distinct from each other and any crossing of $L$, and appear as \textit{endpoints} of long (arc) components of $L$. An endpoint  of a component that corresponds to $L(0)$ is called a \textit{tail} and to $L(1)$ is called a \textit{head} of $L$.
Each long component of $L$ is considered to be oriented from its tail to its head. 

A \textit{knotoid diagram} is a linkoid diagram that consists of only one long component and a \textit{knotoid multi-diagram} is a linkoid diagram that consists of a number of circular components and one long component. 
 \end{definition}
 Note that if every endpoint of every long component of a linkoid diagram $L$ is assumed to lie at the boundary of a $2$-disk then $L$ is a tangle. Moreover, a linkoid diagram consisting of only circular components is clearly a link. 

Throughout the paper, we assume the ambient manifolds that links and linkoids lie in as closed, connected, orientable surfaces.

\begin{definition}\normalfont
Two linkoid diagrams in a surface are \textit{equivalent} if they can be transformed to one other by a sequence of three Reidemeister moves RI, RII, RIII that take place away from the endpoints of diagrams and isotopy of the surface.  It is forbidden to pull/push an endpoint across an arc. A \textit{linkoid} is an equivalence class of linkoid diagrams given by the relation induced by Reidemeister moves. 
\end{definition}

\begin{definition}\normalfont
The \textit{universe} $U_L$ of a linkoid or link diagram $L$ in a surface is the graph that is obtained by replacing each crossing  of $L$ with a vertex. Each endpoint of $L$, if there is any, is also considered to be a vertex. Each semi-arc of $L$ is regarded as an edge of $U_L$, and a region of $L$ is a connected component of $\Sigma_g - U_{L}$.  That is, a region of $L$ corresponds to a face of $U_L$. 
 A region that is enclosed with vertices and edges is a \textit{bounded} region of $L$, otherwise it is called the \textit{exterior} region of $L$.

\end{definition}

 \begin{definition}\normalfont
 If each region of a link or linkoid diagram $L$  is homeomorphic to a disk then we say $L$ is a \textit{tight embedding} (or tightly embedded) in $\Sigma_g$. Note that every link or linkoid diagram in $S^2$  is a tight embedding. \end{definition}

\begin{definition}\normalfont
A linkoid or link diagram in a surface is called \textit{split} if its universe is disconnected.

\end{definition}
\begin{definition}\normalfont
A linkoid or link diagram in a surface is called \textit{connected} if it is not split. Equivalently, a linkoid or link diagram is \textit{connected} if its underlying universe is connected.
\end{definition}

\section{Admissible diagrams}\label{sec:decorated}
\subsection{Results from Euler's formula}\label{sec:Euler}

 We begin by characterizing the relation between the number of faces and the number of four-valent vertices of a link or linkoid universe that is tightly embedded in a surface.

\begin{proposition}\label{prop:euler}

Let $U$ be the universe of a connected link or linkoid diagram with $n$ crossings that is tightly embedded in a closed, connected and orientable surface $\Sigma_g$. We have the following.

\begin{enumerate}

\item If $U$ is the universe of a link diagram then $ f - n = 2 - 2g$, where $f$ is the number of faces of $U$. 

\item If $U$ is the universe of a linkoid diagram, then $f - n = 2 - 2g - m$, where $f$ is the number of faces of $U$, $m$ is the number of long components of the linkoid diagram. 
\end{enumerate}

\end{proposition}

\begin{proof}
It is an easy observation that a connected $4$-regular graph with $n$ vertices has $2n$ edges connecting the vertices.  A connected linkoid universe $U$ is a connected graph with $n$ four-valent vertices and $2m$ pendant vertices where $2m$ is the number endpoints of the knotoid components of the linkoid. This gives us the following equality.
$$4n = 2e - 2m,$$ where $e$ is the number of edges of the universe $U_L$.

By Euler's formula, we know that for a connected graph that is tightly embedded in $\Sigma_g$, $g \geq 0$,
\begin{center}
$v - e + f = 2 - 2g$,

\end{center}
where $v$ is the number of vertices, $e$ is the number of edges, $f$ is the number of faces of the graph.

By using the two equalities above, we can conclude the following. 
If $U$ is the universe of a link diagram with $n$ crossings then $$n - 2n + f = 2-2g$$ and so $$f -n =2-2g.$$
If $U$ is the universe of a linkoid diagram with $n$ crossings and $m$ knotoid components then $$(n +2m ) -(2n +m) +f = 2-2g. $$ 
So, $$f- n~ =~2-2g -m.$$

.

\end{proof}
Proposition \ref{prop:euler} can be generalized to disconnected link or linkoid diagrams in $S^2$ or $\mathbb{R}^2$. 

\begin{proposition}\label{prop:eulerii}
Let $L$ be  a disconnected link or linkoid diagram in $S^2$ or $\mathbb{R}^2$ with $k>1$ disconnected components. 
Then we have the following.
\begin{enumerate}
\item If $L$ is a link diagram then $f - n=k+1$, where $f$ is the number of regions and $n$ is the number of crossings in $L$.
\item If $L$ is a linkoid diagram with $m$ knotoid components then $f - n = k -m+1$.

\end{enumerate}

\end{proposition}

\begin{proof}

The generalized Euler formula for planar graphs is given as $$v - e + f = 1 +k,$$ where $k$ is the number of disconnected components of the graph. Let $U$ is the universe of a link diagram in $S^2$ (or $\mathbb{R}^2$) with $k$ disconnected components. We already know that $e = 2n$ for $U$, where $n$ is the total number of crossings in $L$. Then, $f-n = 1 +k$ follows directly by substituting the formula. Note that when $L$ is a knot diagram in $S^2$, $k=1$ and so we obtain the equality $f- n = 2$, given in Proposition \ref{prop:euler}.

 Let $L$ be a linkoid diagram with $n$ crossings,  $m$ knotoid components and $k$ disjoint components. Substituting $v= n+ 2m$, $e= 2n +m$ into the formula, we reach to the required equality $f-n = k - m +1$. 

\end{proof}
\begin{remark}\normalfont
We remind the reader that any tight embedding of a graph in a surface of genus  $g >0$ forces the graph to be connected \cite{Lando}. We will be working with links and linkoid diagrams that are tight embeddings except in Section \ref{sec:more}. 
\end{remark}

\begin{definition}\normalfont

A tightly embedded link or linkoid diagram is called \textit{admissible}, if $f=n$,  where $f$ denotes the number of faces and $n$ denotes the number of vertices of its universe.
\end{definition}

The following is deduced directly from Propositions \ref{prop:euler} and \ref{prop:eulerii}.

 \begin{corollary}\label{cor:admissablegen}
\begin{enumerate}[i.]

\item A link diagram that is tightly embedded in a closed, connected, oriented genus $g$ surface $\Sigma_{g}$ is admissible if and only if the link diagram is connected and  $g=1$. 

\item A  linkoid diagram $L$ is admissible if and only if it lies in $S^2$ and the total number of knotoid components in $L$ is one more than the number of its disjoint components. See Figure \ref{fig:comp} where we picture an admissible linkoid diagram  in $S^2$ with two disjoint components and three knotoid components. 

\end{enumerate}

\end{corollary}

   \begin{figure}[H]
\centering
\includegraphics[width=.3\textwidth]{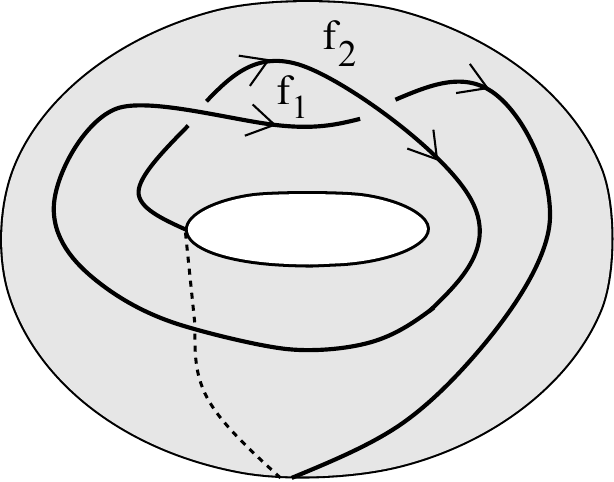}
\caption{An admissible knot diagram in torus with two crossings and regions.}
\label{fig:tor}
\end{figure}

   \begin{figure}[H]
\centering
\includegraphics[width=.35\textwidth]{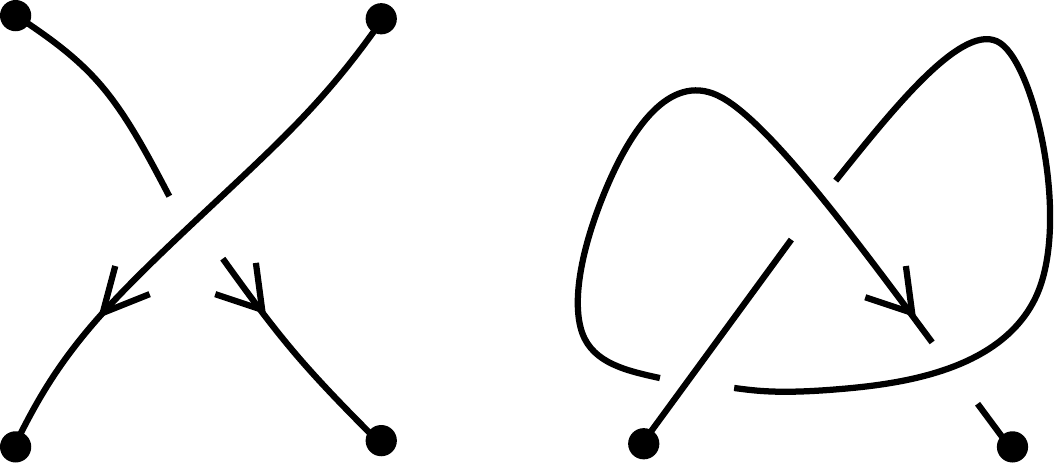}
\caption{An admissible disconnected linkoid diagram with two disjoint components, consisting of three knotoids.}
\label{fig:comp}
\end{figure}

%\subsection{Obtaining admissible linkoid diagrams in $S^2$ via isotopy}\label{sec:admissible}
An RII move that takes place between two disjoint components of a split linkoid diagram with a number of disjoint components clearly increases or decreases the number of disjoint components of a split linkoid diagram in $S^2$ by one. This makes it possible that a sequence of RII moves turns a non-admissible split linkoid diagram into an admissible linkoid diagram. Here we study the case of split linkoid diagrams consisting of only two disjoint knotoid components. By Corollary \ref{cor:admissablegen} ii., it is true that one can always obtain an admissible linkoid diagram from a split linkoid diagram that consists of only two knotoid components by connecting the components by an RII move. We illlustrate examples of this in Figures \ref{fig:splitintoadmissible}  and \ref{fig:eqv} where disjoint components of  split linkoid diagrams in $S^2$ are connected by RII moves.     %admissible linkoid diagrams obtained in this way are not necessarily equivalent to each other through admissible diagrams. That is, one may need to demolish connectivity to transform one admissible diagram. In Figure \ref{fig:noneqv}, we illustrate a pair of admissible linkoid diagrams that are obtained from the same split linkoid diagram by a RII move. 

%We believe that there is no sequence of Reidemeister moves and spherical isotopy transforming one to another by preserving the connectivity. We leave it as a conjecture that any equivalence 

%(!) In fact, we have the following proposition.

   \begin{figure}[H]
\centering
\includegraphics[width=.4\textwidth]{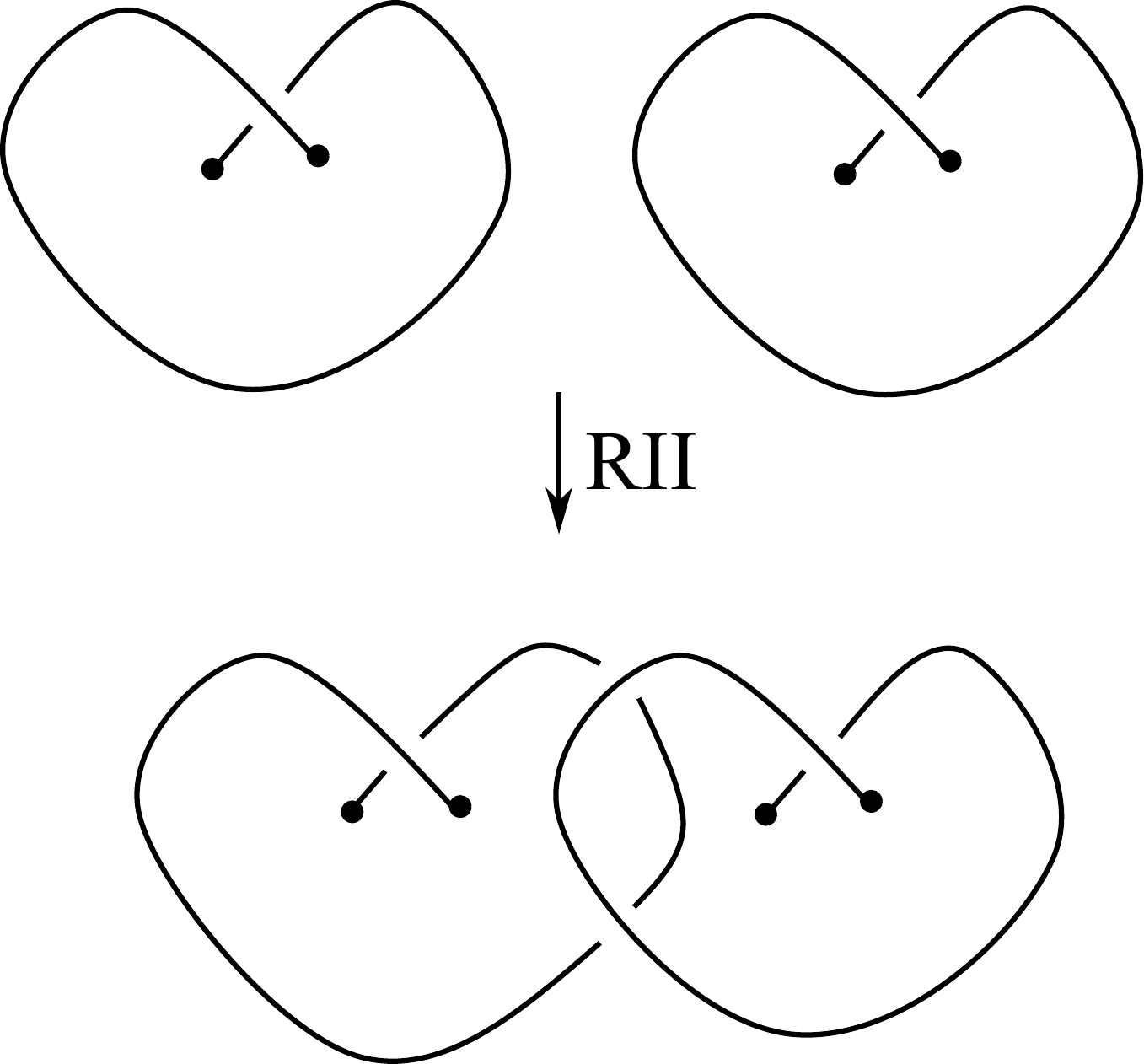}
\caption{A split linkoid transformed into an admissable diagram.}
\label{fig:splitintoadmissible}
\end{figure}

   \begin{figure}[H]
\centering
\includegraphics[width=.65\textwidth]{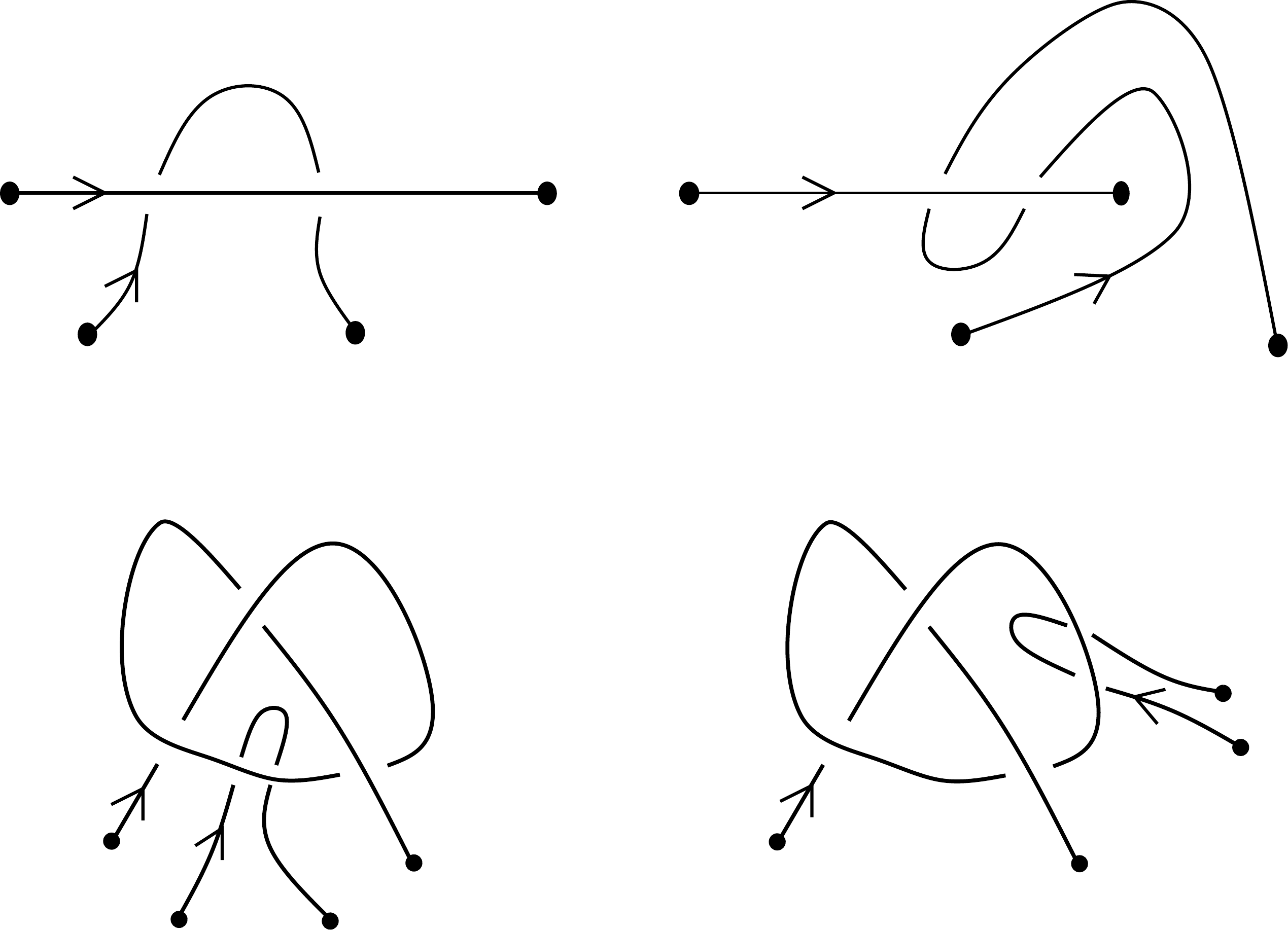}
\caption{Two pairs of admissible linkoid diagrams on top and bottom obtained by RII moves.}
\label{fig:eqv}
\end{figure}

  \begin{figure}[H]
\centering
\includegraphics[width=1\textwidth]{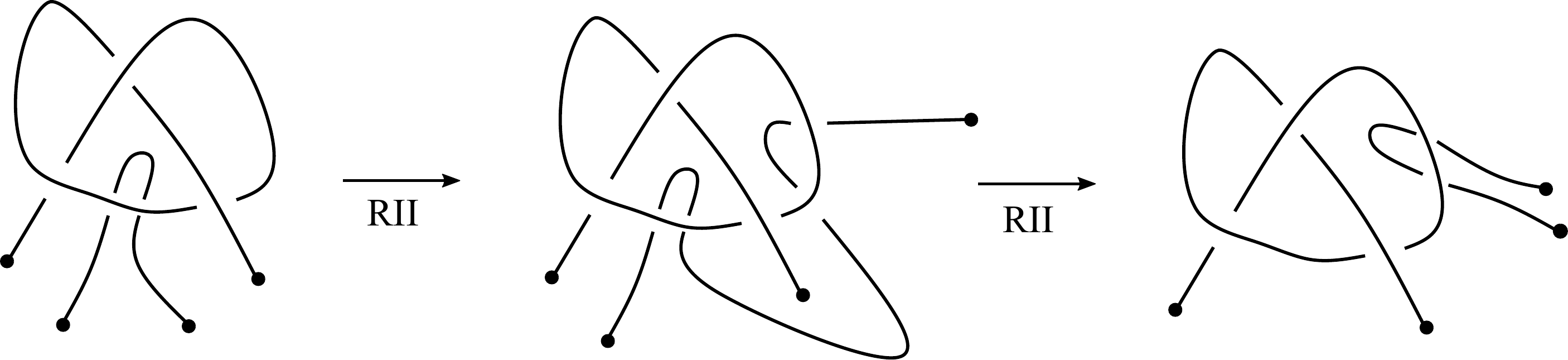}
\caption{RII moves exchanging the connecting regions.}
\label{fig:noneqv}
\end{figure}

\begin{proposition}\label{prop:split}
Let $L$ be a split linkoid diagram in $S^2$ that consists of only two knotoid components, denoted by $\alpha$ and $\beta$. Then, any two connected linkoid diagrams obtained from $L$ by an RII move which pushes the component $\beta$ under (or over) $\alpha$  are equivalent to each other through connected diagrams.
\end{proposition}

\begin{proof}
%Assume first that at most one of the endpoints of both $\alpha$ and $\beta$ lie in the exterior region. 

Consider any two possible ways to connect the components  $\alpha$ and $\beta$ by applying an RII move that pushes $\beta$ under (or over) $\alpha$. It is clear that the resulting connected diagrams can be taken to each other by a sequence of planar isotopy moves and two RII moves exchanging the connecting regions. During this transformation, connectivity is preserved.

\end{proof}

\subsection{Starred links and linkoids}\label{sec:starred}
In this section, we introduce a decoration for link and linkoid diagrams lying in a closed, connected, orientable surface. The purpose for decorating a diagram will be to obtain the equality for the number of four-valent vertices and the number of regions in the corresponding decorated universe. This equality will induce a square region-vertex incidence matrix whose permanent gives Alexander type polynomials.

 Let $L$ be a non-admissible link or linkoid diagram in a surface of genus $g$. Let the equality $f = n + k$ hold, for some $k \in \mathbb{N}_{> 0}$, where $f$ is the number of regions and $n$ is the number of crossings of $L$. We choose a collection of $k$ regions of $L$ and decorate each of them with a star. Regions endowed with a star are considered to be dismissed from $L$ so that the equality $f=n$ holds for the resulting starred diagram. In a similar manner, if the number of crossings is larger than the number of regions of $L$, that is, if $n= f + k$ holds, for some $k \in \mathbb{N}_{> 0}$, then we decorate a collection of $k$ crossings of $L$ with stars. A vertex that receives a star is considered to be dismissed from the universe of $L$ so that in the resulting starred diagram, the number of regions is equal to the number of crossings.

\begin{definition}\normalfont
A \textit{starred} link or linkoid diagram in a surface is a link or linkoid diagram that is endowed with stars either on a number of its regions or crossings so that it satisfies the equality $f = n$. We consider a link or linkoid diagram in a surface which already satisfies $f=n$ without any star addition as a starred diagram with zero stars. Figures \ref{fig:tor} and \ref{fig:comp} set examples for such diagrams.

 %A starred link or linkoid diagram by $L$.
\end{definition}

\begin{definition}\normalfont
Two starred link or linkoid diagrams in a surface  are \textit{star equivalent} if they are related to each other by Reidemeister moves that take place away from the starred regions or vertices and isotopy of the surface. The addition/deletion of a starred vertex or pulling a strand across a starred crossing or region is not allowed. In this way the number of starred regions or crossings is preserved. See Figure \ref{fig:starredmoves} for a list of moves that are not allowed for decorated diagrams. 
\end{definition}

\begin{figure}[H]
\centering
\includegraphics[width=.3\textwidth]{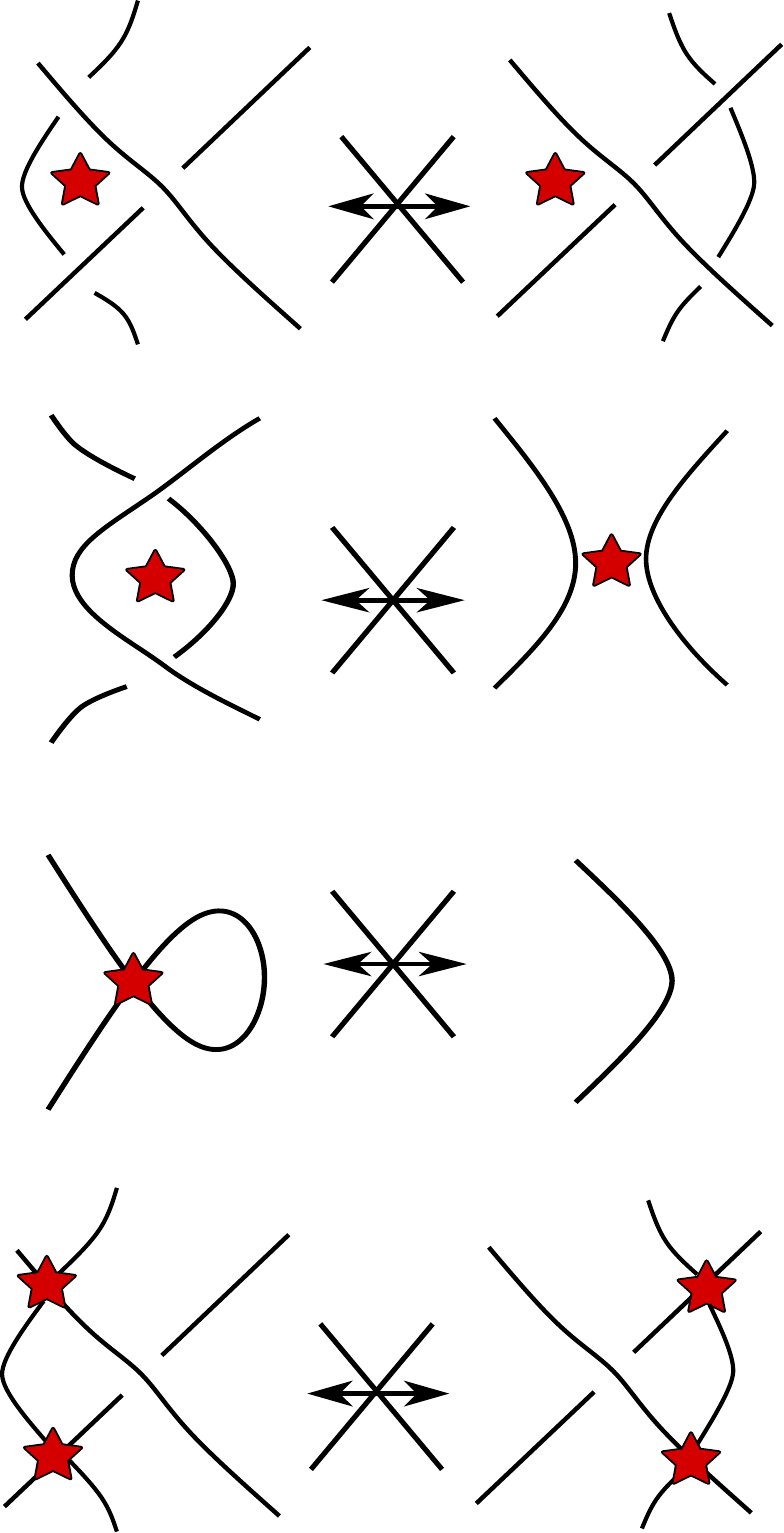}
\caption{Moves that are not allowed for starred diagrams.}
\label{fig:starredmoves}
\end{figure}
\begin{definition}\normalfont
A \textit{starred link} or \textit{linkoid} in a surface is an equivalence class of star equivalent link or linkoid diagrams.
\end{definition}

It is clear that the starred linkoid diagrams in Figure \ref{fig:noneqv} are not star equivalent , since the number of regions endowed with a star is not equal.
\begin{figure}[H]
\centering
\includegraphics[width=.6\textwidth]{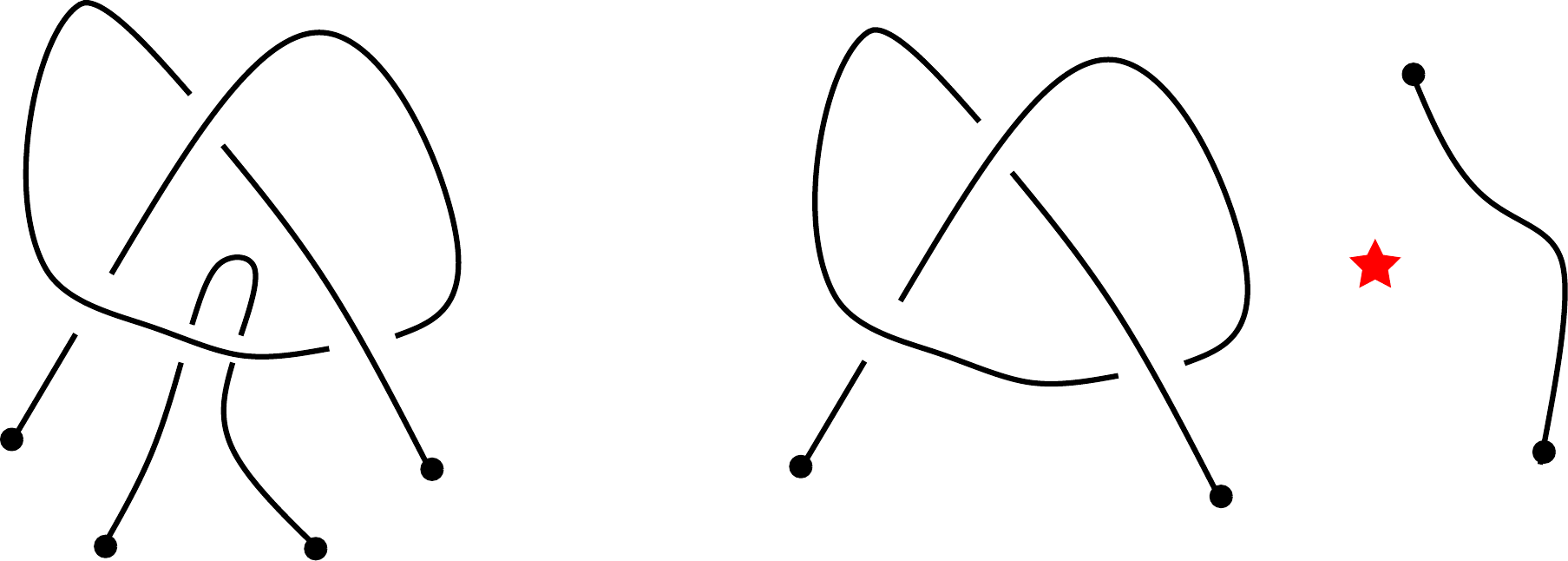}
\caption{Two nonequivalent starred linkoid diagrams.}
\label{fig:noneqv}
\end{figure}

Note here that a starred knotoid in $S^2$ can also be considered as a generalized knotoid. Generalized knotoids were introduced in \cite{Adams}.

%In Figure \ref{fig:Figure 5} we see a starred trefoil diagram with two stars placed at a pair of adjacent regions. 

%\begin{proposition}
%Let $K$ be a knot diagram in $S^2$. Any two starred diagrams that are obtained from $K$ by endowing a pair of adjacent regions with stars are star equivalent. 

%\end{proposition}

%\begin{proof}
%\end{proof}

%\begin{proposition}

 %$L$ is equivalent to a connected and admissible universe that is obtained by a flat R-II move. Moreover, any two connected admissible diagrams obtained by a flat R-II move from $L$ are equivalent to each other if $L$ is a universe whose components have at most one of its endpoints in the region where the R-II move takes place.

%Let $L$ be linkoid or link diagram that is either split in $S^2$ or not admissible in $T^2$. $L$ can be made into a connected or admissible diagram by a R-II move and two connected or admissible diagrams obtained in this way from $L$ are equivalent to each other.

%\end{proposition}

 %one of which contain the \textit{exterior region} where $\infty \in S^2$ lies. 

 %\begin{definition}\normalfont

%A knotoid is called \textit{knot-type} if the endpoints lie in the same region in at least one of its representations. In the standard representation of a knot-type knotoid, the endpoints lie in the exterior region. A knotoid is called \textit{proper} if in all of its representations the endpoints lie in distinct regions.

%\end{definition}

\section{Mock Alexander polynomials for starred links and linkoids} \label{sec:state-sum}

%\subsection{States }\label{sec:states}

\subsection{A state-sum polynomial for starred links and linkoids}\label{sec:map}
%We first introduce states of connected link or linkoid diagrams in $S^2$.   
In this section we introduce states and a state-sum polynomal of starred link or linkoid diagrams. 

%\begin{definition}\normalfont

%Let $L$ be an admissible link or linkoid diagram in a surface. We mark each (face) region of its universe at exactly one of the crossings incident to the region, as shown in Figure \ref{fig:st}. A choice of placement of markers at each of vertices of the universe of $L$ is called a \textit{state} of $L$. 
%\end{definition}

\begin{definition}\normalfont

%$k > 0$ at regions if $f = n +k$ and at its crossings if $n = f+ k$.  
Let  $L$ be an oriented starred link or linkoid diagram in a surface.  A \textit{state} of $L$ is obtained by placing a marker at every face of the universe of $L$ that does not admit a star,  placed at exactly one of the crossings that is incident to the region that also does not admit a star. We illustrate a state of a starred knotoid diagram in $S^2$, with black state markers in  Figure \ref{fig:st}. The reader can find the collection of all states of the starred knotoid diagram in Figure \ref{fig:labeltwo}.% If the starred diagram $L$ admits stars at its crossings, a \textit{state} of $L$ is obtained by endowing each region with a marker at a crossing incident to the region but without a star.

 If $L$ is a starred link or linkoid diagram without any stars on its regions or crossings, then a state of $L$ is obtained by endowing each face with a marker at exactly one of the crossings incident to the region that receives a marker.

%The universe of the starred link or linkoid diagram obtained from $L$, with a possible placement of markers at each of its face, is called a \textit{state} of $L^*$.

    \end{definition}
   
   %  as described above such that a region or crossing that receives a star is left without a state marker and each of the remaining regions and crossings receives a marker. 
   
   \begin{figure}[H]
\centering
\includegraphics[width=.3\textwidth]{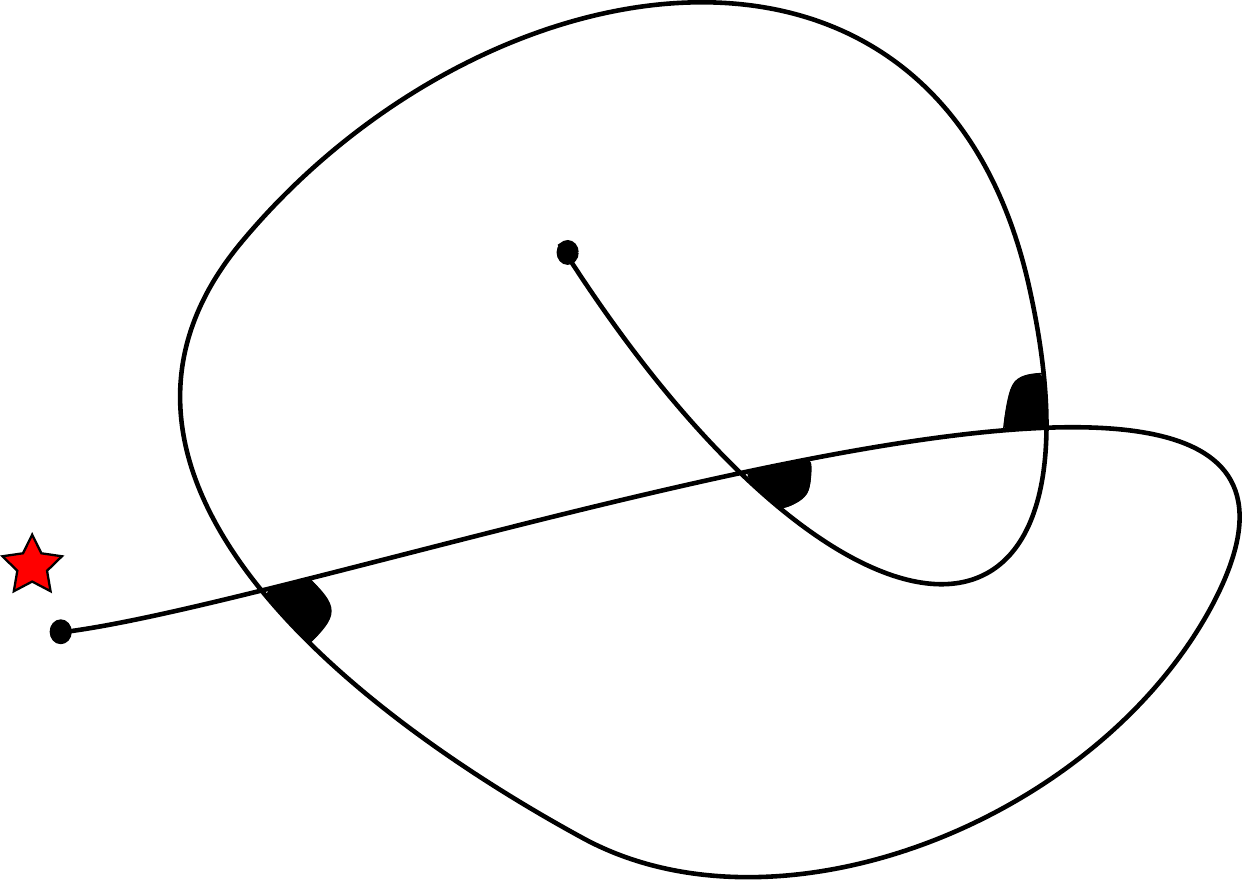}
\caption{A state of a starred knotoid diagram.}
\label{fig:st}
\end{figure}
The following observations can be derived easily from Proposition \ref{prop:euler}.

 %A \textit{state marker} is an indicator of one one of the four local regions at a crossing of $L$. See the state markers in Figure \ref{fig:st} given in black triangular forms.  

 \begin{enumerate}
 \item Let $L$ be a connected link diagram in $S^2$. Since $f=n+2$, we decorate the universe of $L$ with two stars at any of its two regions, then endow each of the remaining regions with exactly one state marker at an incident crossing to obtain a state of $L$.
 \item If $L$ is a connected linkoid diagram in $S^2$ with two knotoid components, then its universe is admissible. A state of $U$ is obtained directly by endowing $U$ with exactly one state marker at each of its bounded regions, placed at a crossing incident to the region.
 \item If $L$ is connected linkoid diagram in $S^2$ with $m \geq 3$ knotoid components, then we have $f - n  \leq -1$. In this case, endowing a collection of $f-n$ crossings with stars turns $L$ into a starred linkoid diagram.  A state of the resulting starred inkoid diagram is obtained by placing a state marker at each of its regions next to a crossing that is free of stars and incident to the region.  %Note that, $U_L$ is not decorated The exterior region of $U_L$ does not receive any state marker or star in any state of $L$.

 %The collection of all states of $L$ is obtained with respect to the same choice of the starred regions.

%\item If $L$ is a link or linkoid diagram in a surface $\Sigma_g$ $g > 0$ with a non-admissible universe, a number of stars are assigned to either its regions or vertices to turn it into an admissible universe. The choice of 

%\item If $L$ is admissible in $\mathbb{T}^2$, then endow each region with a state marker at an incident crossing. 
%\end{enumerate}
% We illustrate an example in Figure \ref{fig:st}, and for the complete list of states of the illustrated universe, see Figure \ref{fig:clock}. %If a linkoid/multi-knotoid diagram is split then it can be made into a connected diagram by a R-II move so that one can obtain states on the resulting connected universe.
 
 \end{enumerate}

%\end{definition}

 We consider the following labels given in Figure \ref{fig:labelgeneral}, at local regions incident to a crossing of $L$. %$\alpha, \alpha^{'},\beta, \beta^{'}, \gamma, \gamma^{'}, \tau, \tau^{'}$, 

\begin{figure}[H]
\centering
\includegraphics[width=.35\textwidth]{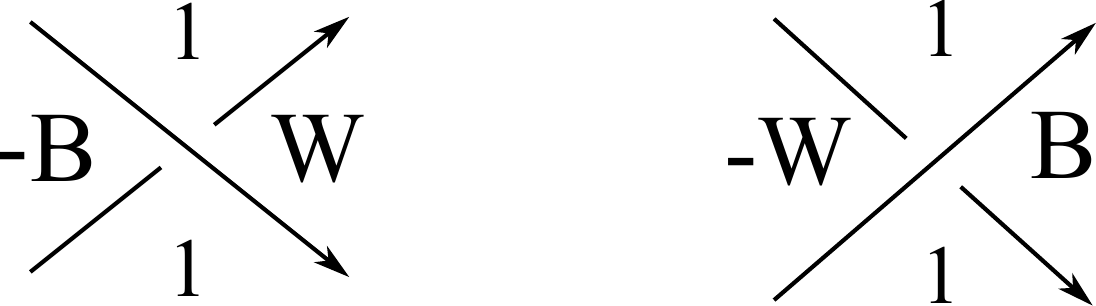}
\caption{Labels at local regions incident a positive and negative crossing.}
\label{fig:labelgeneral}
\end{figure}

%\begin{figure}[H]
%\centering
%\includegraphics[width=.5\textwidth]{label.pdf}
%\caption{Labels at a positive and a negative crossing}
%\label{fig:label}
%\end{figure}
\begin{definition}\normalfont
%Let $L$ be a link or linkoid diagram and $L^*$ be a starred diagram of $L$.
Let $L$ be an oriented starred link or linkoid diagram and $S$ denote a state of $L$.
 The \textit{weight} of  the state $S$, denoted by  $<L~ |~ S>$  is defined to be the product of labels at the crossings that receive a marker in $S$.% that are occupied by the state markers.
\end{definition}

\begin{definition}\normalfont
Let $L$ be an oriented starred link or linkoid diagram in a surface. We define a state-sum polynomial for $L$ denoted by $\nabla_{L}$ as follows.
$$\nabla_{L} (W, B) = \sum_{S \in \mathcal{S} } < L ~|~ S >,$$
where $\mathcal{S}$ denotes the set of all states of $L$. We call $\nabla_{L}$ the \textit{potential} of $L$.
\end{definition}

%%%%%%%%%%%%%%%%%%%%%%%
\begin{note}\normalfont
We will show below that the potential of a starred link or linkoid diagram $L$ is invariant with respect to the star equivalence with the condition $W= B^{-1}$. The reason that we first  present the potential of $L$ is that it is of graph theoretic interest on its own. This approach is supported by a conjecture we state in Conjecture \ref{conj:conj1}.% about the symmetry placement.
\end{note}
%%%%%%%%%%%%%%%%%%%%%%

\begin{example}\label{ex:one}\normalfont
In Figure \ref{fig:labeltwo}, we depict a starred knotoid diagram in $S^2$, named as $K_1$ and all of its states with their weights below.
It can be easily verified that the potential of  $K_1$ is  $W^2 - WB +(W-B) +1$.  %The contributions of its states to the function, respectively, are $s^2$, $-1$, $1$, $s$, $-s^{-1}$. Therefore, $\nabla_{L}(s)= s^2 + s - s^{-1}$. 

\begin{figure}[H]
\centering
\includegraphics[width=.8\textwidth]{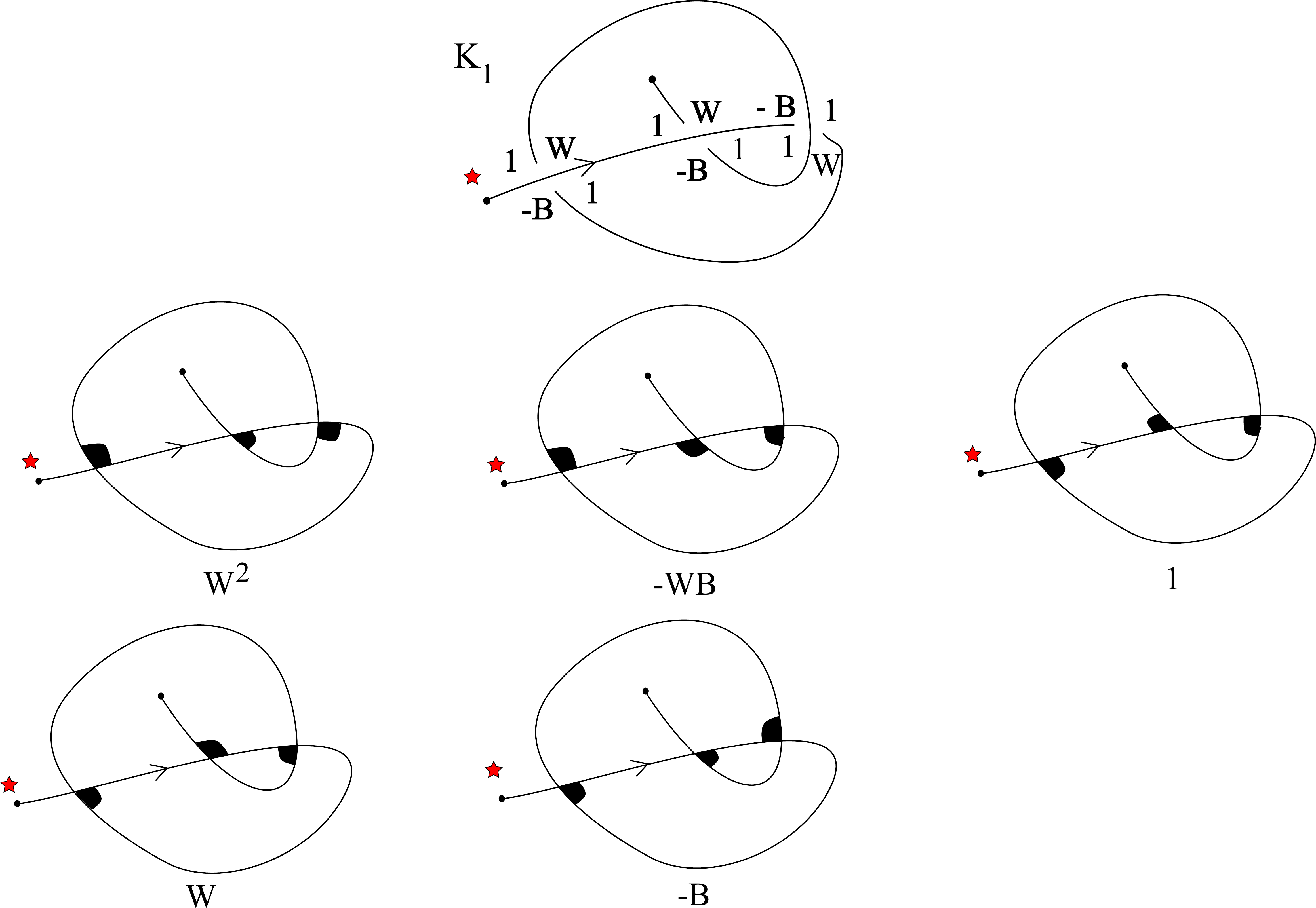}
\caption{The starred knotoid diagram $K_1$ in $S^2$ and its states.}
\label{fig:labeltwo}
\end{figure}

 %One can obtain another admissible universe of $K$ by placing the star in the region that is incident to the head of $K$.
 In Figure \ref{fig:allstates2}, we depict another starred knotoid diagram $K_2$ in $S^2$ (which is the same with $K_1$ except the starred region), with its states and state weights. We find that  the potential of $K_2$ is $B^2 -WB +(W-B)+1$. 
\begin{figure}[H]
\centering
\includegraphics[width=.8\textwidth]{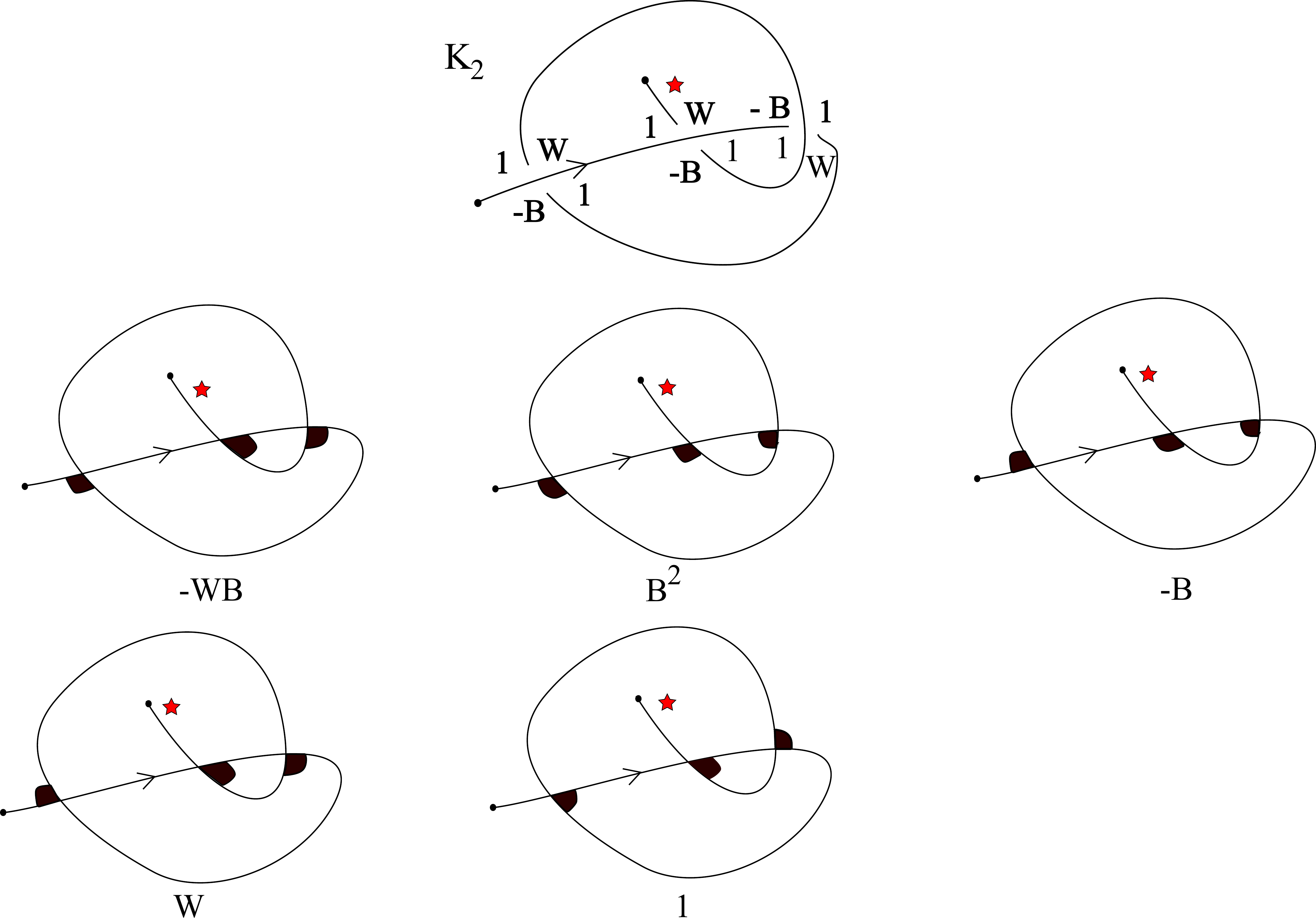}
\caption{The starred knotoid diagram $K_2$ in $S^2$ and its states.}
\label{fig:allstates2}
\end{figure}
\end{example}

Notice that the potential of $K_2$ , can be obtained by replacing $W$ by $-B$ and $B$ by $-W$ in the potential of $K_1$ given in Figure \ref{fig:labeltwo}. 
We have observed this symmetry in many other examples of pairs of starred knotoid diagrams in $S^2$ which differ from each other only for the starred regions as discussed above. We conjecture the following.

\begin{conjecture}\label{conj:conj1}
Let $K$ be a knotoid diagram in $S^2$, and $K_1$, $K_2$  be two starred knotoid diagrams obtained from $K$ by placing a star in the regions of $K$ that are incident to the tail and the head of $K$, respectively.
Then, the following holds.
$$\nabla_{K_{1}}(W, B) = \nabla_{K_{2}} (-B, -W).$$

\end{conjecture}

At the time of writing this paper, a specialization of this conjecture given here as Conjecture \ref{conj:ii}, has been proven for the case when $W= B^{-1}$ \cite{Woutpaper}.

% We point out here that if Conjecture \ref{conj:con1} can be proven, we have the following immediate corollary.

%If we can solve this conjecture, it will imply the following.

%\textit{\bf A Corollary of the conjecture:} If $K$ is a knot-type knotoid diagram  then $\nabla_{K} (W,B) =  \nabla_{K} (-B, -W)$. 

\subsubsection{A matrix formulation for the potential}\label{sec:permanent}

%It is clear that any knotoid universe with $n$ vertices divides the $2$-sphere $S^2 = \mathbb{R}^2 \cup \infty$ into $n+1$ regions, one of which is exterior in the sense that it contains the infinity point.  Note that any knotoid diagram is isotopic to a knotoid diagram with its tail located at the exterior region which we assume to be starred. 

Let $L$ be an oriented link or linkoid diagram with $n$ crossings, that is tightly embedded in a closed, connected, orientable surface of genus $g$. We enumerate each region and  crossing of $L$  with respect to a chosen starting point and the orientation on $L$. We consider the local  region labeling at crossings of $L$, that we introduce in Figure \ref{fig:labelgeneral}.

$L$ endowed with the local labels at its crossings can be represented by a matrix $M_L = [M_{ij}]$ whose rows and columns correspond crossings and regions of $L$, respectively, and consist of the corresponding label. Precisely, the ${ij}^{th}$ entry of $M_L$ is the sum of  labels on the $j^{th}$ region received at the $i^{th}$ crossing. If the $j^{th}$ region is not incident to the $i^{th}$ region then the corresponding entry is $0$. We call this matrix \textit{potential} matrix of $L$.

%\begin{example}\normalfont
The potential matrix of the knotoid diagram $K$ is presented in Figure \ref{fig:combex}. %In this figure,  $f_i$, for every $i = 0,1,2,3$ denotes a region of the diagram and crossings of $K$ are enumerated with respect to the orientation on $K$, starting from the tail of $K$.
\begin{figure}[H]
\centering
\includegraphics[width=.7\textwidth]{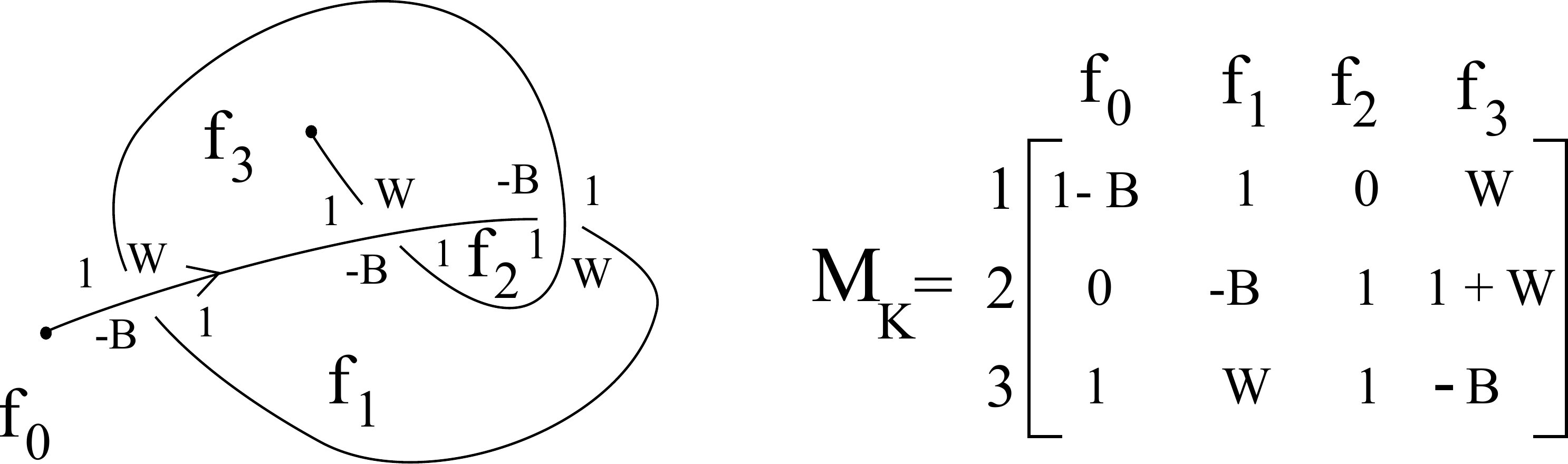}
\caption{The potential matrix of $K$}
\label{fig:combex}
\end{figure}
%The potential matrix of $K$ is:
%$$M_{K}= \begin{bmatrix}
%1-s^{-1}&  1&       0&      s\\
%0&      -s^{-1}&    1&    1+s\\†
%1&      s&     1&      -s^{-1}

%\end{bmatrix}
%$$

\begin{definition}\normalfont

Let $L$ be a starred link or linkoid diagram with $n \geq 1$ crossings, obtained from a link or linkoid diagram  in a surface by endowing a number of either regions or crossings of the link or linkoid diagram with stars. The \textit{potential matrix} of $L$ is the square matrix that is obtained by deleting each column or row from the potential matrix of the link or linkoid diagram which  corresponds to a region or a crossing, respectively that receive a star. We denote the potential matrix of $L$ by $M_L$. 
\end{definition}

% The following proposition shows that this equality holds for any knotoid diagram.
%\end{example}
%We have the following proposition relating the permanent of the potential matrix and its Mock Alexander polynomial. In fact, for the link case, the observation was made in \cite{FKT} by the second author. We generalize it to linkoids.
\begin{definition}\normalfont
Given an $n \times n$ matrix $M_{ij}$. The\textit{ permanent} of $M$, $Perm(M)$ is given by
\begin{center}
$Perm(M) = \sum_{\sigma \in S_{n}} \Pi^{n}_{i=1} M_{i\sigma (i)}$,
\end{center}
where the sum runs over the symmetric group $S_n$. 
\end{definition}

\begin{proposition}\label{prop:perm}
Let $L$ be a starred link or linkoid diagram in a surface, with $n \geq 1$ crossings and regions.
The potential of $L$, $\nabla_{L}(W, B)$ is equal to the permanent of the potential matrix $M_{L}$ of $L$. %That is, $\nabla_{L}(W)=  Perm (M_{L})$.
\end{proposition}

\begin{proof}
Given an $n \times n$ matrix $M_{ij}$. In the permanent expansion of $M_{ij}$, for every $\sigma \in S_n$, the components of the product $\Pi^{n}_{i=1} M_{i\sigma (i)}$ corresponds to a single row choice by each column since $\sigma$ is a bijection. For the potential matrix of $L$, this translates to a single choice of a vertex by each region of $L$. If $L$ is a starred link diagram, then this correspondence implies that each product in the permanent sum are in fact a weight of a state that appears in the mock Alexander polynomial of $L$.

Suppose now, $L$ is a starred linkoid diagram with a region incident to an endpoint of a knotoid component of $L$ and the crossing adjacent to the region do not admit a star. Then, the region incident to an endpoint receives two labels and the corresponding entry in the potential matrix is the sum of these labels. These labels contribute in state weights of two states of $L$ that differ from each other at only the vertex (crossing) that receives the labels. Then, the sum of these state weights is equal to the permanent product involving entry corresponding to the vertex and the region incident to it receiving two labels.

By the arguments above, we see that each nonzero summand of the permanent of $M_{L}$ is equal to a state weight or the sum of two state weights of $L$. Therefore, the permanent of $M_{L}$ is equal to the potential of $L$.

\end{proof}

\begin{example}\normalfont
 Let $K$ be the knotoid diagram given in Figure \ref{fig:combex}. Consider the starred diagram $K_{*}$ obtained from $K$ by endowing the region $f_0$ with a star. One can verify by direct calculation that the permanent of the potential matrix $M_{K_{*}}$ is $W^2 - WB +(W-B) +1$, and this coincides with the potential of $K_1$ calculated in Example \ref{ex:one}.

\end{example}

\subsection{An invariant of starred links and linkoids  induced by the potential}\label{sec:invariance}
Here, we investigate the conditions for the potential of a starred link or linkoid diagram to induce an invariant of starred links and linkoids.

\begin{theorem} \label{thm:invariance}
The potential is an invariant of starred links and linkoids that lie in a surface of genus $g \geq 0$, when $WB = 1$.

\end{theorem}

\begin{proof}
Let $L$ be a starred link or linkoid diagram that lies in a surface. 
An RI move adds a crossing and a new region incident to the crossing to $L$ or deletes a crossing and one of the regions that is incident to the crossing from $L$. Let us assume the first case. 
 \begin{figure}[H]
\centering
\includegraphics[width=.35\textwidth]{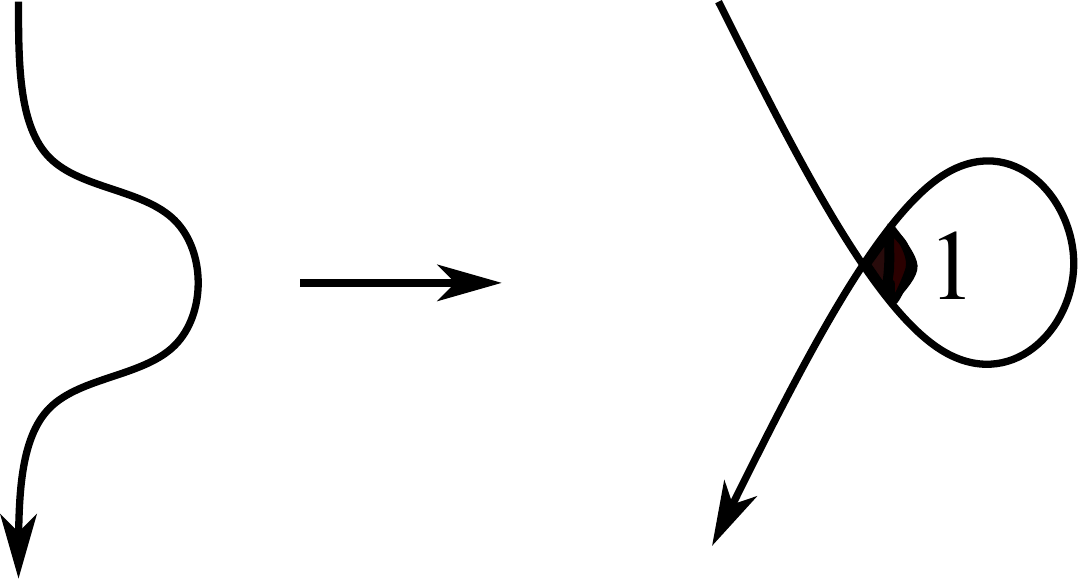}
\caption{Invariance under an RI move.}
\label{fig:invariance1}
\end{figure}
The  region added after the move receives a state marker at the new crossing, as shown in Figure \ref{fig:invariance1}, and this state marker appears in any state of the new diagram without any affect on the state markers of the former diagram.  This implies that the states of the former and latter  diagrams are in one-to-one correspondence with each other. We can see easily that the contribution of the state marker at the new crossing  to the state weight is trivial regardless of the type of the crossing created. Therefore, the state-sum polynomial is not affected under this move. If a type I move deletes a crossing from the diagram then a region that is adjacent only to the deleted crossing is removed and since the weight of the state marker at the crossing is trivial by the directions of the strands at the crossing  , the potential is not affected.

An RII adds or deletes two regions locally to a starred link or linkoid diagram. In Figure \ref{fig:nostar}, we examine two cases of an RII move where the strands in the move sites are not adjacent to any endpoints (if the diagram is a linkoid), and no  region incident to strands is endowed with a star. The possible local placements of the state markers at the crossings in the RII move are shown in Figure \ref{fig:nostar}. Here each tensor product sign indicates that the regions with the tensor products receive the state markers at crossings that are not involved with the move. We deduce that the assumption $WB =1$ gives the invariance under the first case of an RII move where the strands have the same direction. For the other case, where the strands are directed oppositely, the invariance is provided directly. Note here that the regions $T$ and $U$ may be the same region in the corresponding link or linkoid diagram. The local state configurations would be the same for this case as well. There if $WB= 1$, the invariance is provided.

 \begin{figure}[H]
\centering
\includegraphics[width=1\textwidth]{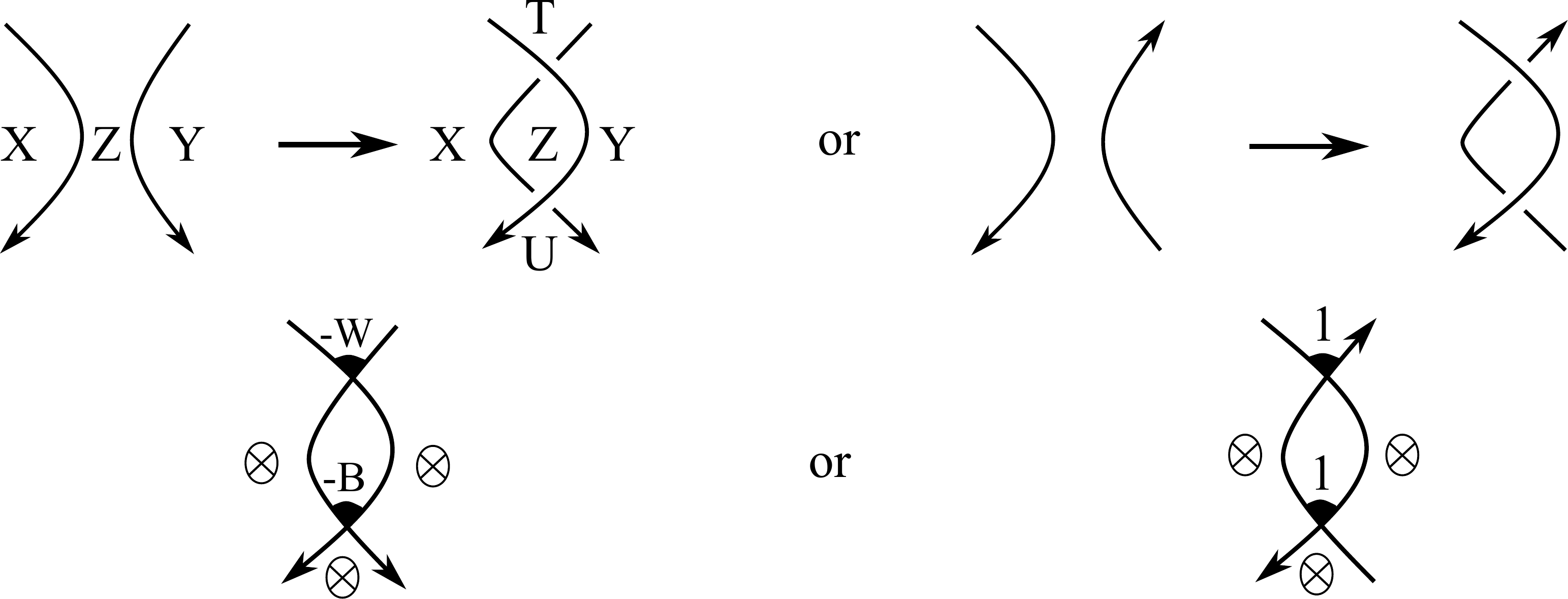}
\caption{Invariance under an RII move.}
\label{fig:nostar}
\end{figure}
Some other cases of an RII move applied on a starred link or linkoid diagram where the neighboring regions in an RII move site are endowed with stars, and the local state configurations corresponding to these cases, are depicted in Figure \ref{fig:starredR2} and Figure \ref{fig:starredendpoints}. The invariance of the potential under the assumption $WB =1$ can be verified easily by the figures for all of these cases. Here we examine in detail the behavior of the potential under the RII moves, depicted in Figure \ref{fig:starredendpoints},  where the strands in the move site are incident to endpoints of a starred linkoid diagram, and the verification for other cases of an RII move are left for the reader.

 \begin{figure}[H]
\centering
\includegraphics[width=1.1\textwidth]{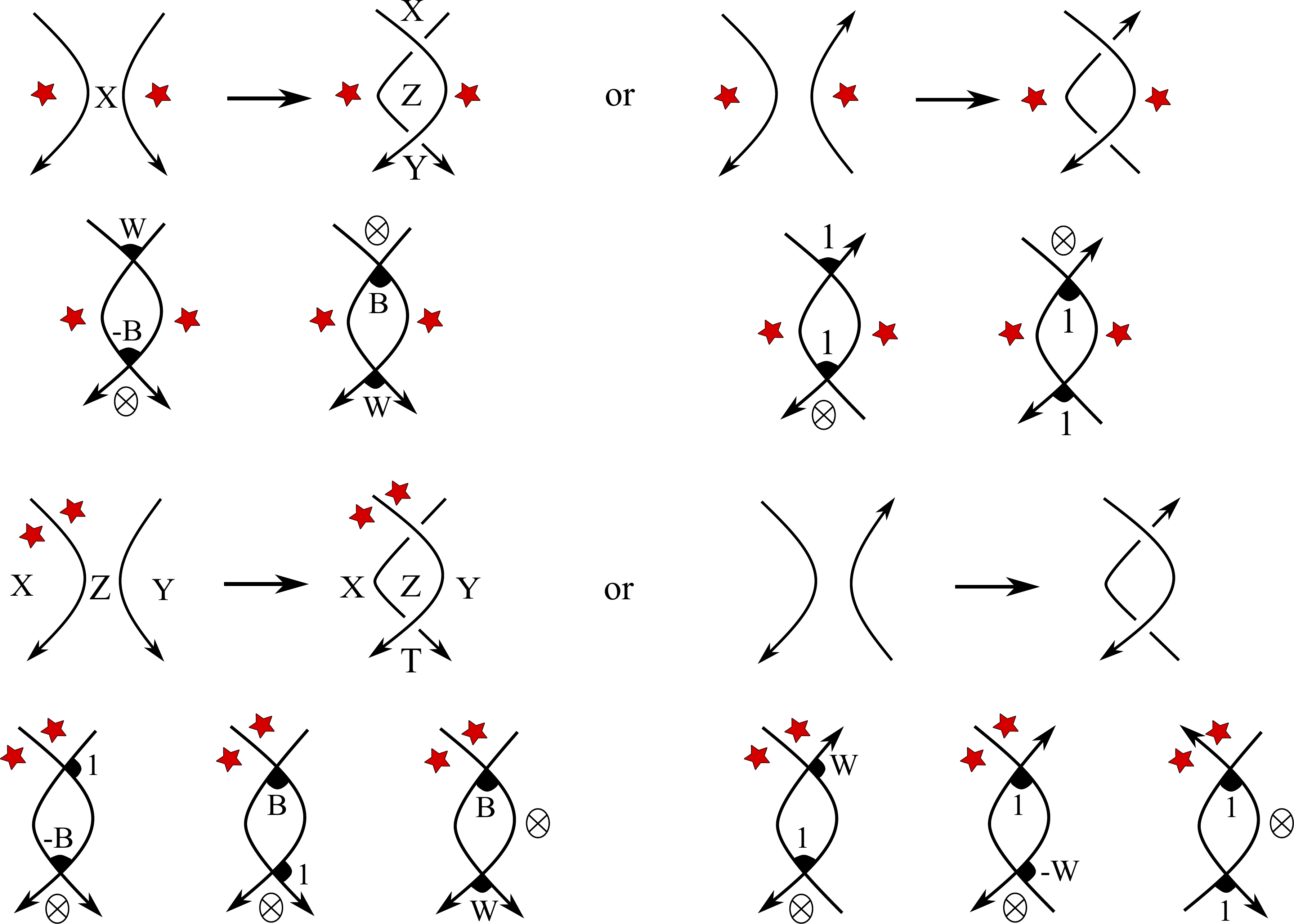}
\caption{RII moves where the neighboring regions receive stars.}
\label{fig:starredR2}
\end{figure}

 \begin{figure}[H]
\centering
\includegraphics[width=1\textwidth]{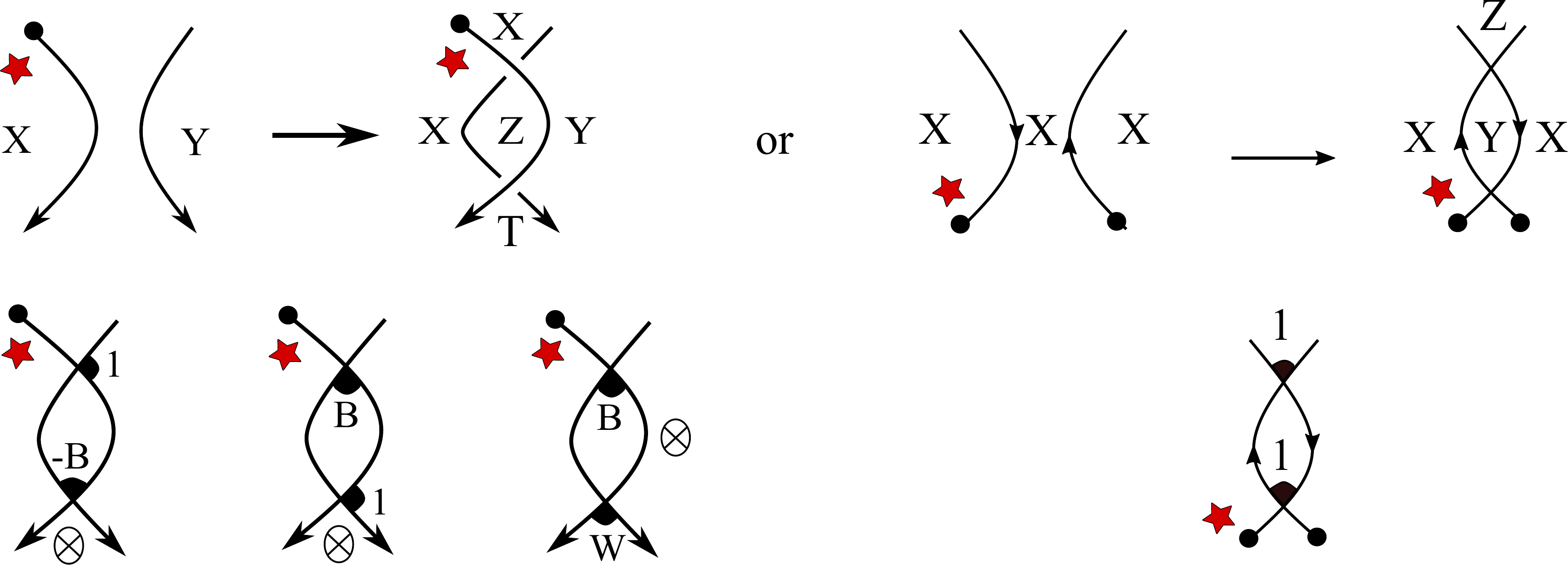}
\caption{RII moves where strands are incident to endpoints and neighboring regions receive stars.}
\label{fig:starredendpoints}
\end{figure}

One of the strands in the RII move depicted on the top row of Figure \ref{fig:starredendpoints},  is assumed to be incident to the tail of the knotoid component, and the adjacent region that is labeled as $X$, is endowed with a star. Notice here that region $X$ lies on both sides of the tail arc.  The move adds two regions (labeled as $T$ and $Z$) to the local portion of the diagram appearing after the move. We assume that the regions $X$ and $T$ are not the same region.  There are in total three possible local placements of the state markers at the crossings in the RII move as in the previous cases. We picture these placements at the bottom row of Figure \ref{fig:starredendpoints}. 

Notice that for every state $s$ that involves the first (leftmost) configuration, there exists a unique state that involves the second (middle) configuration and coincides with $s$ except  for the state markers at the crossings in the move site. Then,  it can be also be verified by the figure that for a state that involves the first configuration and  contributes to the potential with $-BP$, where $P$ denotes the product of weights of state markers outside the move site, there is a state that involves the second configuration and contributes with $BP$.  Thus, the total contribution of states involving the first and the second configurations to the potential is zero.  In the third local state configuration, the regions added by the move, namely  the regions $Z$ and $T$ receive state markers and the state markers at the crossings outside the move site are not affected. Then, the contribution of a state that involves this configuration is $WBP$, where $P$ holds for the product of state weights outside the move site in the state.  Therefore, the potential is invariant under the type II move if $WB=1$.  For the second case, it is not hard to see that the contributions of the first two local state configurations cancel each other, and the invariance under this move follows directly, as the last local state contributes with trivially.

%The states of the diagrams occurring before and after the move takes place, are in one-to-one correspondence. 

The regions labeled by $X$ and $T$ may be the same region only if the strands in the move belong to different components of the linkoid diagram. If $X$ and $T$ are the same region, then the region $T$ is also the starred region and does not receive any state marker.  So, the first two local state configurations at the crossings added by the move, are the only possible local configurations for states of the resulting diagram. where the region $Y$ gets its marker at the crossings added by the move in these configurations. It can be verified from the figure that the contributions of these two local states multiply the product of weights of the remaining crossings of the diagram with $-B$ and $B$, respectively in every possible state. This means that the total sum of contributions, and so the potential of the starred linkoid diagram is zero. 
% \begin{figure}[H]
%\centering
%\includegraphics[width=.5\textwidth]{reid2-2.pdf}
%\caption{Invariance under a type II move}
%\label{fig:invariance2}
%\end{figure}
 
In the case depicted on the right hand side of Figure  \ref{fig:starredendpoints}, the strands in the move site are incident to the endpoints of a knotoid diagram. The endpoints apparently lie on the same region that is endowed with the star. Here the RII move adds two new regions, named as $Y, Z$ to the knotoid diagram since the diagram is connected.  There is only one possible  configuration of state markers at the move site, as shown in the figure. The product of state weights in this local configuration is trivial ($1$) and the states of the knotoid diagrams before and after the move are in one-to-one correspondence since the state markers at the regions before the move are all outside of the move site and not affected by the move. Therefore, the potential of the knotoid diagram remains the same under the move.
% \begin{figure}[H]
%\centering
%\includegraphics[width=.5\textwidth]{RIIcase2.pdf}
%\caption{Invariance under a type II move}
%\label{fig:RIIcase2}
%\end{figure}

 \begin{figure}[H]
\centering
\includegraphics[width=.5\textwidth]{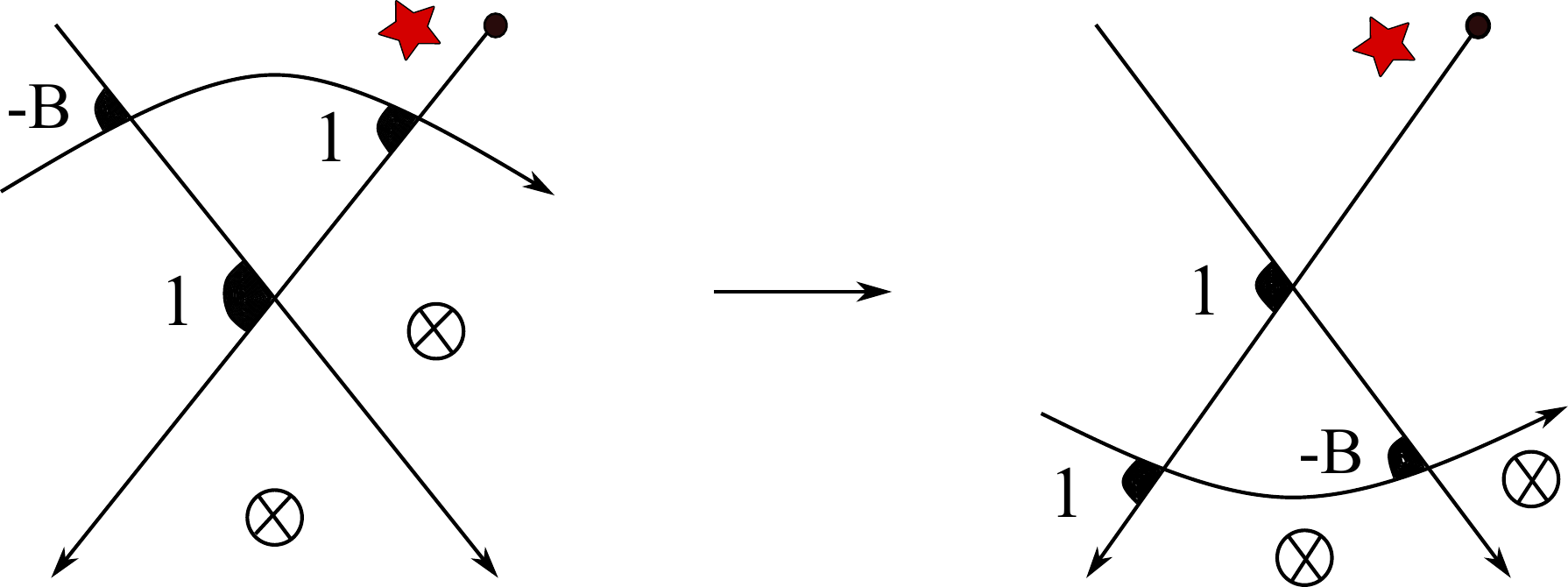}
\caption{An RIII move.}
\label{fig:invariance3-1}
\end{figure}
There are various cases for an RIII move as well. We first study the invariance in a case depicted in Figure \ref{fig:invariance3-1}, where one of the strands appearing in the move site is adjacent to the tail of the diagram and the region incident to the tail is endowed with a star. It is not hard to observe that the regions of the diagram before the move takes place are in one-to-one correspondence with the regions of the diagram obtained after the move takes place, and there is only one possible placement of the state markers at the crossings that are in the move site before and after the move takes place. The move also has no effect on the state markers at the regions that lie outside of the move site. Without loss of generality, we can assume all flat crossings shown in the move site are positive crossings in the corresponding knotoid diagram. It can be verified by the figure that the product of weights of the crossings lying in the move sites before and after the move takes place, are both $- B$.
Thus, the potential is preserved under this move.

Another case of an RIII move is given in Figure \ref{fig:invariance3-2}, where none of the regions adjacent to the crossings of the move site is starred and the strand adjacent to the tail is involved in the move. In addition to this, three of the local regions in the move site are occupied by tensor product signs. In this configuration, there are three possible placements of the state markers at the crossings in the move site before the move takes place. Assuming that all the crossings in the move site are positive crossings, the sum of the product of weights of the state markers in the local portions of the three states is $W^2 - W^2B - W$, as can be  verified by the figure. After the move, we see that there are five possible placements of the state markers at the crossings of the move site. Summing up the product of weights of the markers in the move site among all these local states, we find that the contribution is $W^2 -B- WB +WB^2 +1$. Therefore, the potential remains unchanged under this move if $WB = 1$.  

A variation of the case depicted in Figure \ref{fig:invariance3-2} where the region incident to the endpoint is not endowed with a star, is given in Figure \ref{fig:invariance3-3}. The invariance when $WB= 1$ can be verified directly by the figure.  Figures \ref{fig:invariance3-2} and \ref{fig:invariance3-3} take care of all possible types of an RIII move where an endpoint is involved with the move site and the star of the diagram is far away from the endpoint.

The reader is directed to \cite{FKT} for the details on possible cases of an RIII move that take place on a link diagram endowed with stars on a pair of adjacent regions. The verification of invariance can also be made similarly as in \cite{FKT}  for an RIII move that takes place on a link diagram that is endowed with stars on a pair of non-adjacent regions.
\begin{figure}[H]
\centering
\includegraphics[width=.8\textwidth]{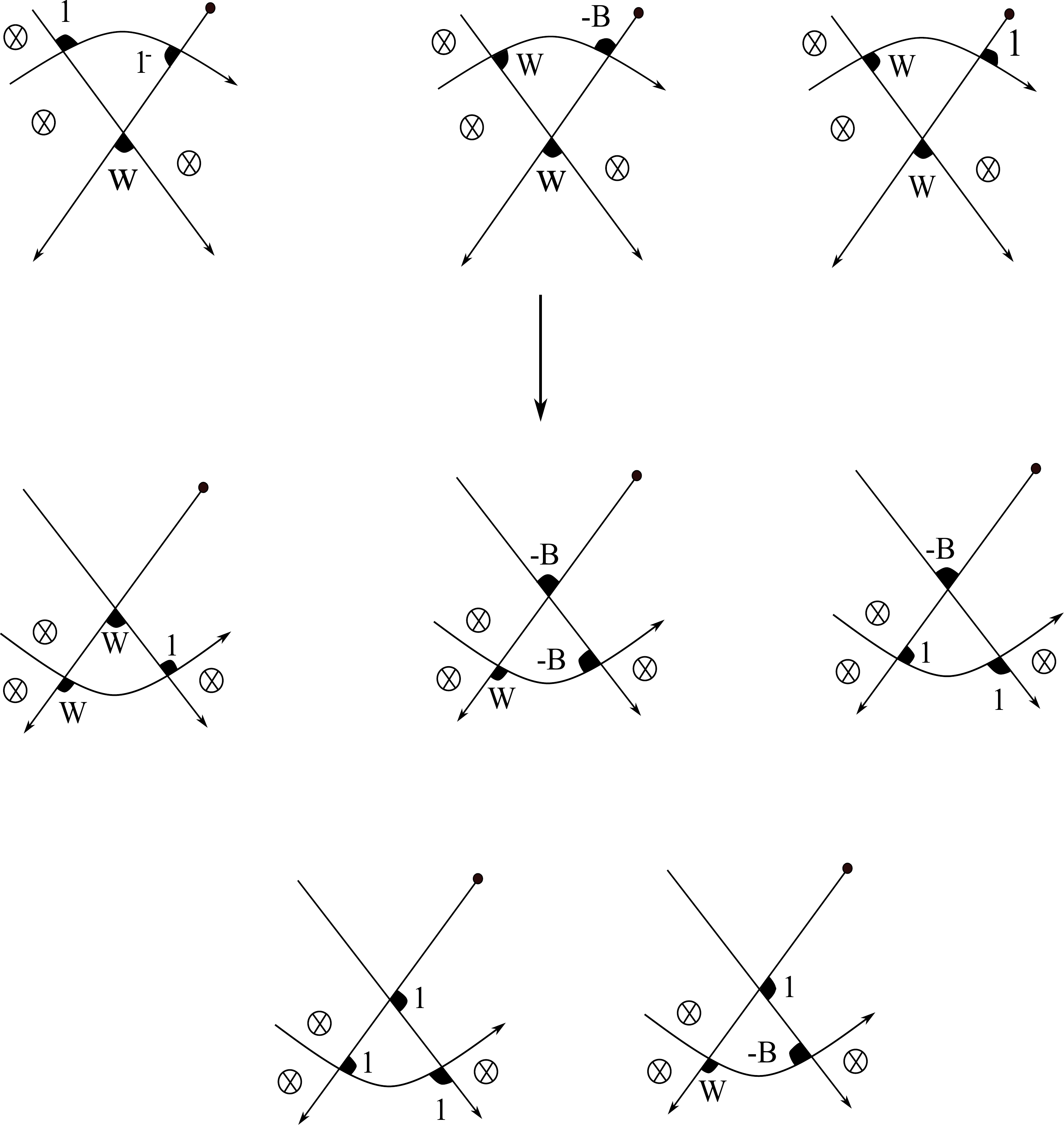}
\caption{An RIII move.}
\label{fig:invariance3-2}
\end{figure}

\begin{figure}[H]
\centering
\includegraphics[width=1.1\textwidth]{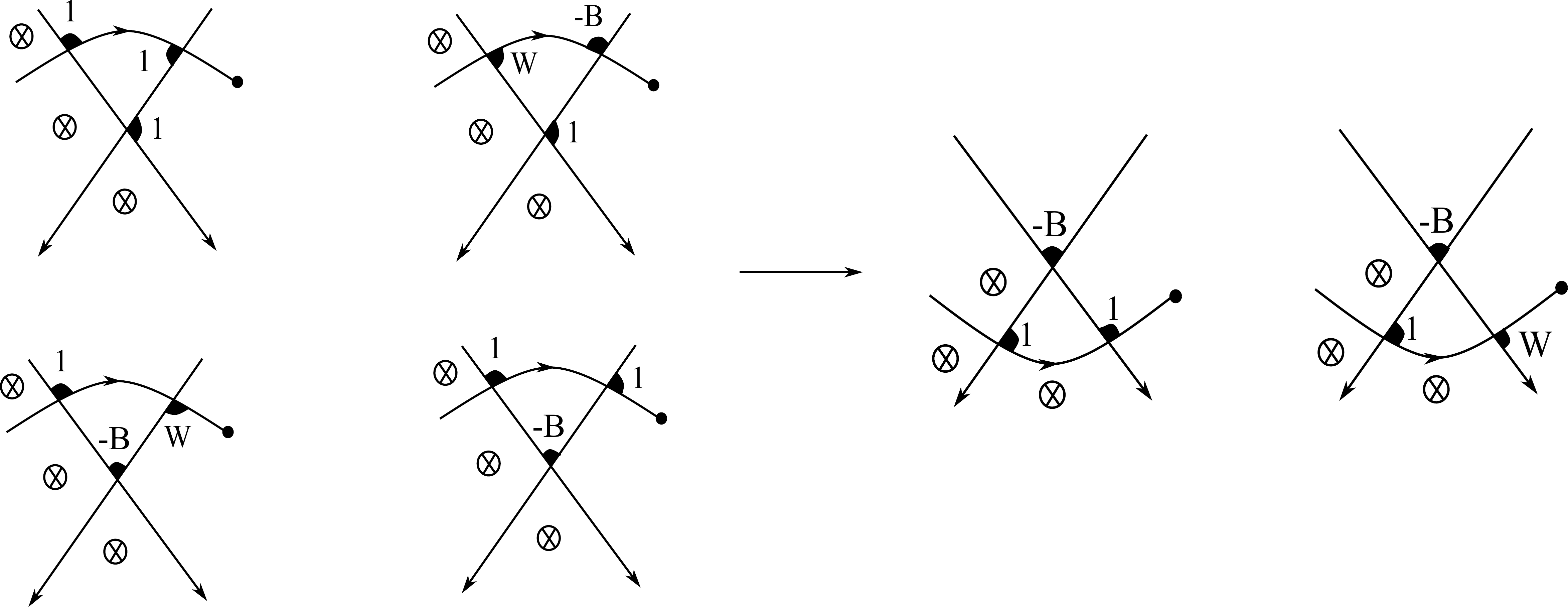}
\caption{An RIII move}
\label{fig:invariance3-3}
\end{figure}
%The invariance of the potential for the other possible configurations for type II and type III Reidemeister moves can be checked similarly as above.  

\end{proof}

\begin{theorem}\cite{FKT}
Let $L$ be an oriented link diagram in $S^2$. The Alexander-Conway polynomial of $L$ is equal the potential of a starred link diagram obtained by endowing a pair of adjacent regions of $L$ with stars, under the assumption $W~=~B^{-1}$ on the region labels.

\end{theorem}

\begin{proof}We discuss the relation between the Alexander-Conway polynomial of an oriented link and the state-sum polynomial briefly in Section \ref{sec:introduction}.
See \cite{FKT} for details.

\end{proof}

%We generalize Alexander's construction to any starred link or linkoid universe, and define Alexander-type invariants.

%\begin{proposition}
%Let $K_n$ denote the starred twist knot obtained by twisting the trivial knot $n$ times in the positive direction  and then linking the ends together. Suppose that the stars are...

%\end{proposition}

%\begin{example}

%\end{example}
%\begin{figure}[H]
%\centering
%\includegraphics[width=1\textwidth]{F22.eps}
%\caption{A knot diagram in a torus}
%\label{fig:}
%\end{figure}

 \begin{definition}\normalfont
 
 Let $L$ be an oriented starred link or linkoid diagram.  The potential of $L$ is called the \textit{mock Alexander polynomial} of $L$ when $W = B^{-1}$. 

\end{definition}

\begin{corollary}
 The mock Alexander polynomial is an invariant of oriented starred links or linkoids in a surface of genus $g$. 
\end{corollary}
\begin{proof}
This follows from Theorem \ref{thm:invariance}.
\end{proof}
It follows from Proposition \ref{prop:perm} that the mock Alexander polynomial of $L$, 
$\nabla_{L} (W)$ is the permanent of  $M_{L}$, where $M_{L}$ is the potential matrix of $L$ with entries $0, 1, W^{\pm 1}$ or $1 \pm W^{\pm 1}$. In the following examples, $\nabla_{L}(W)$ is calculated as the permanent of the potential of $L$.

\begin{examples}\normalfont

\begin{enumerate}

\item Consider the horizontal trivial oriented 2-tangle $T_0$. We add $n \geq 1$ positive twists to $T_0$.  Let $K_n$ denote the numerator closure of the resulting tangle, endowed with stars at the interior regions added by the closure. It is clear that we obtain an oriented starred link diagram. See Figure \ref{fig:k_n} for the starred link diagram obtained in this way.
 
It can be verified from Figure \ref{fig:k_n} that there are exactly two states of $K_n$ for every $n \geq 1$, and the contributions of these states to $\nabla_{K_{n}}(W)$ are $W^n$ and $(-1)^{n} W^{-n}$, respectively.  Therefore, 

$$\nabla_{K_{n}}(W) = W^{n} + (-1)^{n} W^{-n,}$$
 for every $n \geq 1$.

By direct  calculation, one can also see that the following equality holds for every $n \geq 2$.% Notice that this equality is a special case of the skein relation that will be discussed in Section \ref{sec:skein}.

$$\nabla_{K_{n+1}} (W) = \nabla_{K_{n-1} } + (W - W^{-1}) \nabla_{K_n}. $$

\begin{figure}[H]
\centering
\includegraphics[width=.85\textwidth]{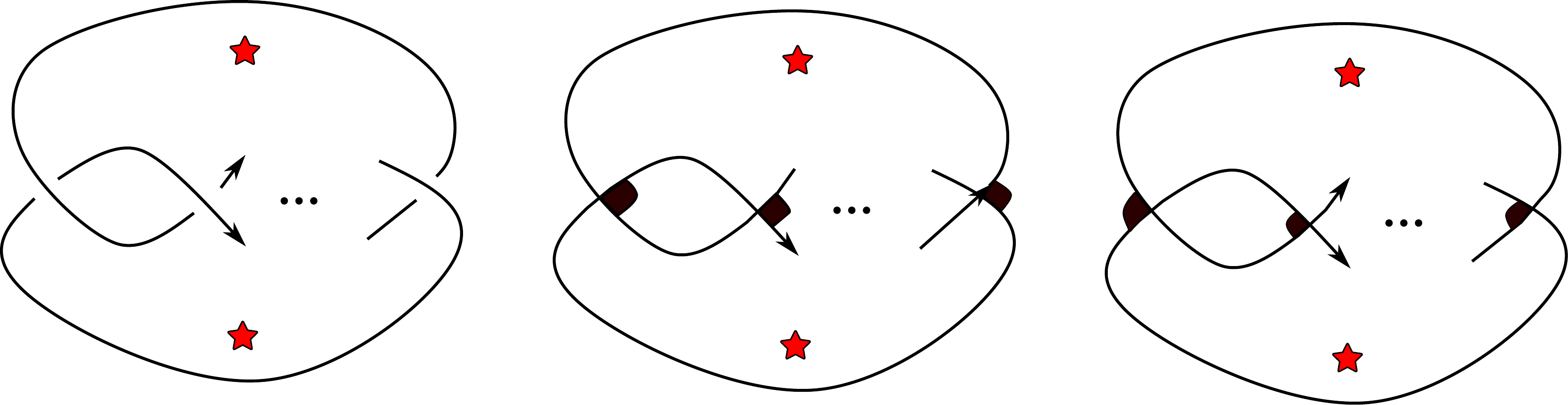}
\caption{ Starred link diagram $K_n$ with nonadjacent stars, for some $n \geq 1.$}
\label{fig:k_n}
\end{figure}

Let $K$ be the starred trefoil knot with its two non-adjacent regions starred, as shown in Figure  \ref{fig:stardiagonal}. Notice that $K$ is star equivalent to $K_3$ discussed above. We verify  that the permanent of $M_K$ coincides with the mock Alexander polynomial calculation of $K_3$.

\begin{figure}[H]
\centering
\includegraphics[width=.5\textwidth]{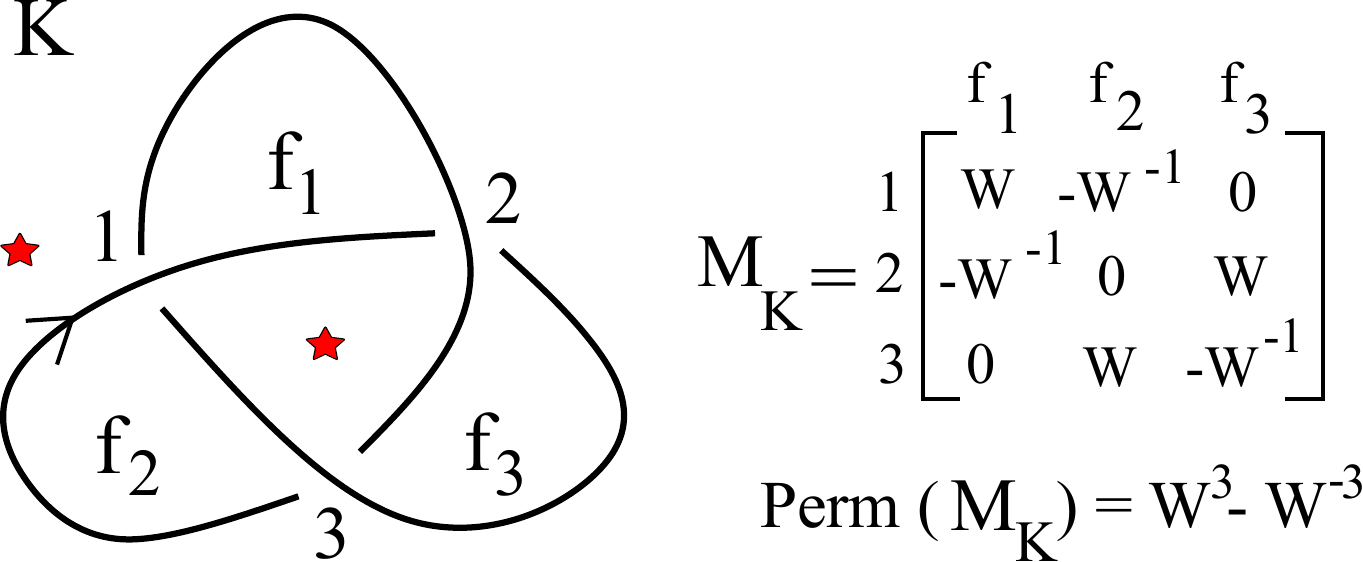}
\caption{$K_3$ is a starred trefoil diagram.}
\label{fig:stardiagonal}
\end{figure}
\begin{figure}[H]
\centering
\includegraphics[width=.5\textwidth]{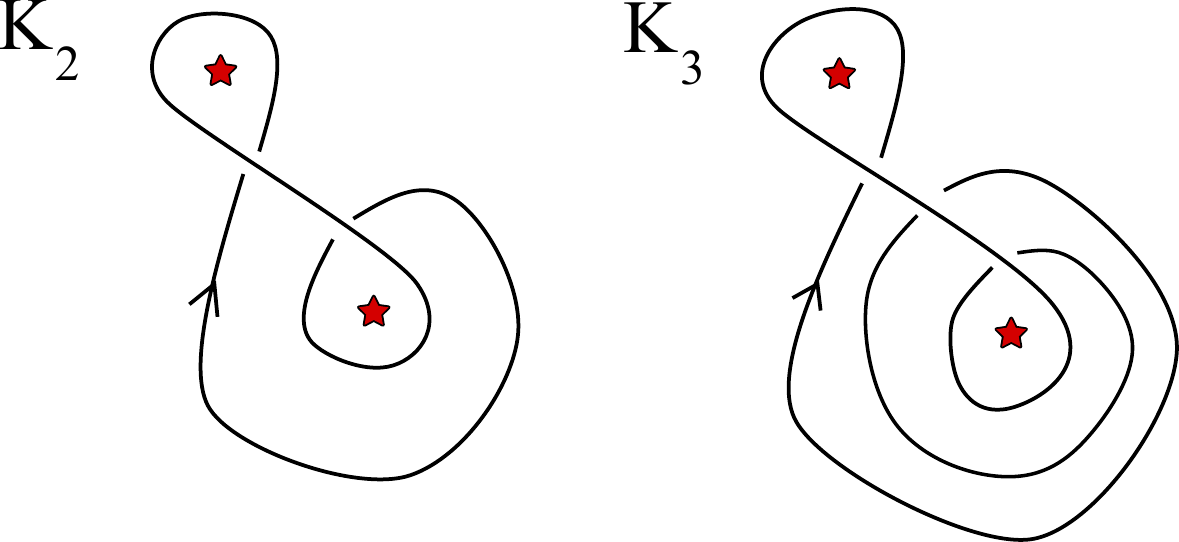}
\caption{Starred knot diagrams in spiral form.}
\label{fig:spiral}
\end{figure}

\item 
Let $K_n$ denote the spiral knot diagram with $n$ crossings with its stars placed in twisted regions as illustrated in Figure \ref{fig:spiral} for $n=2,3$. The potential matrix of $K_n$, $n \geq 2$ is given as
 \[ M_{K_{n}}= \begin{bmatrix}
W - W^{-1} & 1& 0 & \dots& & & 0\\
1 &W-W^{-1}& 1&0 &\dots& & \\
0&1& W-W^{-1}&1&0& \dots& \\
\vdots&0&1&W-W^{-1}&1&0&\dots \\
0&\dots&  &   & 0 &1&W-W^{-1}\\

\end{bmatrix}.
\]

%\[ $M_{K_3} =$
%\begin{bmatrix}
%W - W^{-1} & 1 & 0\\
%1 & W-W^{-1} & 1\\
%0& 1& W-W^{-1}
%\end{bmatrix}
%\]

By direct calculation utilizing the permanent formula, we find
$$\nabla_{K_{2}}(W) = 1 + (W- W^{-1}) ^{2}$$ and  
$$\nabla_{K_{3}}(W) = 2(W - W^{-1}) + (W - W^{-1})^3.$$ 
Assuming that $\nabla_{K_{-1}} (W) = 0$,  $\nabla_{K_{0}} (W) = 1$ and $\nabla_{K_{1}} (W) = W - W^{-1}$, we have the following recursive relation for the mock Alexander polynomial of $K_n$, for any $n\geq 1$,
$$\nabla_{K_{n+1}} (W) = (W - W^{-1}) \nabla_{K_{n}} (W) + \nabla_{K_{n-1}} (W),$$ 
that follows directly from the form of the permanent matrix for the invariants.

\end{enumerate}

\end{examples}

\begin{theorem}\label{thm:connect}
 
Let $L$ be a split linkoid diagram in $S^2$ that consists of only two disjoint knotoid components.
% each of which has at most one of its endpoints at the exterior region. 

Let $L_1$ and $L_2$ be two admissible linkoid diagrams obtained from $L$ by an RII move that connects the disjoint components by pushing the components over or under each other. Then,
$$\nabla_{L_1} (W) = \nabla_{L_2} (W).$$

\end{theorem}
\begin{proof}

The argument follows from Proposition \ref{prop:split} in Section \ref{sec:decorated} that tells any two admissible linkoid diagrams obtained by connecting $L_1$ and $L_2$ are equivalent to each other through connected diagrams.

\end{proof}
\begin{note}\normalfont
By Theorem \ref{thm:connect}, we assume that the mock Alexander polynomial of a split linkoid diagram $L$ in $S^2$ consisting of only two knotoid components is the mock Alexander polynomial of any connected  diagram obtained by connecting the components of $L$ by an RII move.
\end{note}

Let $K$ be any knotoid diagram in $S^2$ with $n$ crossings. It is apparent that $K$ gives rise to $n+1$ different starred knotoid diagrams.  One can also fix the placement of the star so that it is placed at the region adjacent to the tail of the knotoid diagram. Then, the region endowed with the star is uniquely determined for any knotoid diagram and the mock Alexander polynomial of the resulting starred diagram is a polynomial assigned to the knotoid diagram.

\begin{definition}\normalfont
Let $K$ be a knotoid diagram in $S^2$. The mock Alexander polynomial of $K$ denoted by$\nabla^{\sharp}_K$, is defined to be the mock Alexander polynomial of the starred diagram obtained from $K$ by endowing the region adjacent to the tail with a star.

\end{definition}

\begin{example}\label{ex:simpleknotoid}\normalfont
We depict a simple knotoid diagram with two crossings in Figure \ref{fig:simpleknotoid}. \begin{figure}[H]
\centering
\includegraphics[width=.3\textwidth]{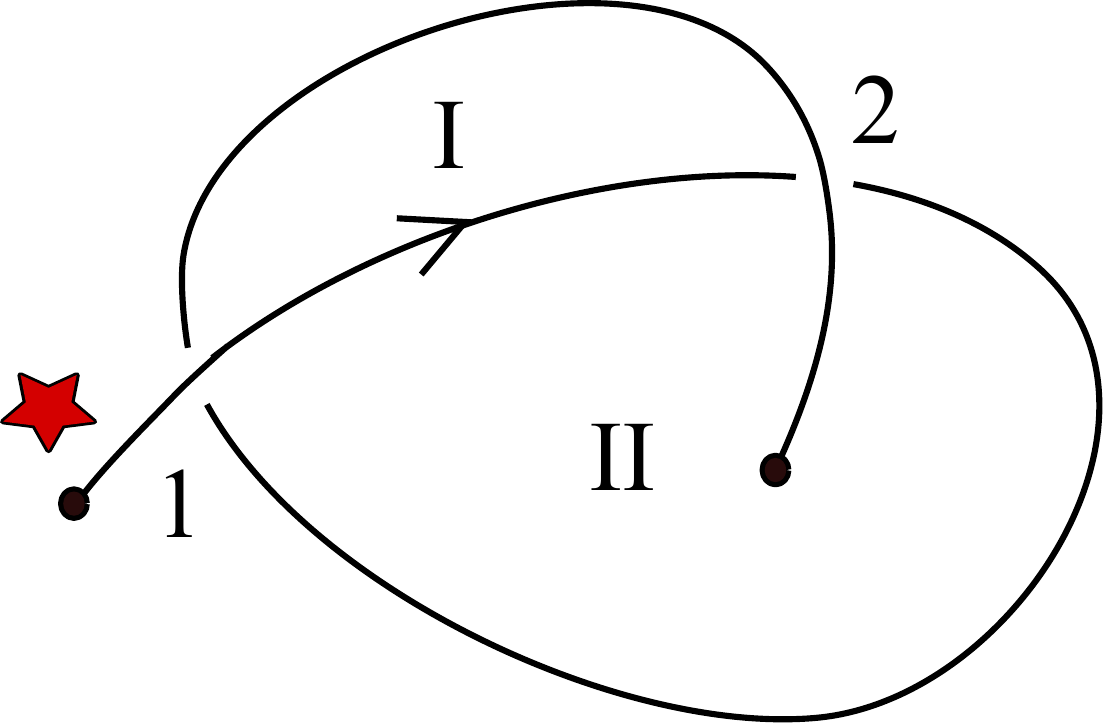}
\caption{A knotoid diagram with two crossings.}
\label{fig:simpleknotoid}
\end{figure}
It is not hard to verify that the mock Alexander polynomial of the knotoid diagram $K$, $\nabla_{K}^{\sharp} (W)$ is the permanent of $\begin{bmatrix}  W & 1 \\-W^{-1} & W +1\\
\end{bmatrix}$, 
which is equal to \\$W^2+W-W^{-1}$.

\end{example}

\begin{proposition}
The mock Alexander polynomial $\nabla^{\sharp}_K$ is an invariant of knotoids in $S^2$. 

\end{proposition}

\begin{proof}

It follows from Theorem \ref{thm:invariance}.

\end{proof}

%\begin{proposition}\label{prop:connection}\normalfont
%Let $L$ be a split linkoid diagram in $S^2$ that consists of only two knotoid components, denoted by $\alpha$ and $\beta$. %Assume that $\alpha$ and $\beta$ have at most one of their endpoints lying in the exterior region. 

%Then, any two connected linkoid diagrams obtained from $L$ by a  Reidemeister II move (which pushes the component $%\beta$ either under or over $\alpha$)  have the same Mock Alexander polynomial.
%\end{proposition}

%\begin{proof}
%This follows directly from Proposition \ref{prop:split} of Section \ref{sec:admissible}. 

%\end{proof}

 \subsubsection{\bf Reversal and mirror behavior of mock Alexander polynomials}\label{sec:symmetry}
\begin{proposition}\label{prop:rev}
Let $L$ be an oriented starred link or linkoid diagram in a surface and $\overline{L}$ be the starred diagram obtained from $L$ by reversing the orientation on each of its components, without changing the placement of the star of $L$. We have the following relation between the mock Alexander polynomials of $L$ and $\overline{L}$.
$$\nabla_{\overline{L}} (W) = \nabla _{L}(-W^{-1}).$$

\end{proposition}

\begin{proof}
The local labels assigned to local regions at a positive and a negative crossing of $L$ and the corresponding crossings of $\overline{L}$, are shown in Figure \ref{fig:reverse}. \begin{figure}[H]
\centering
\includegraphics[width=.35\textwidth]{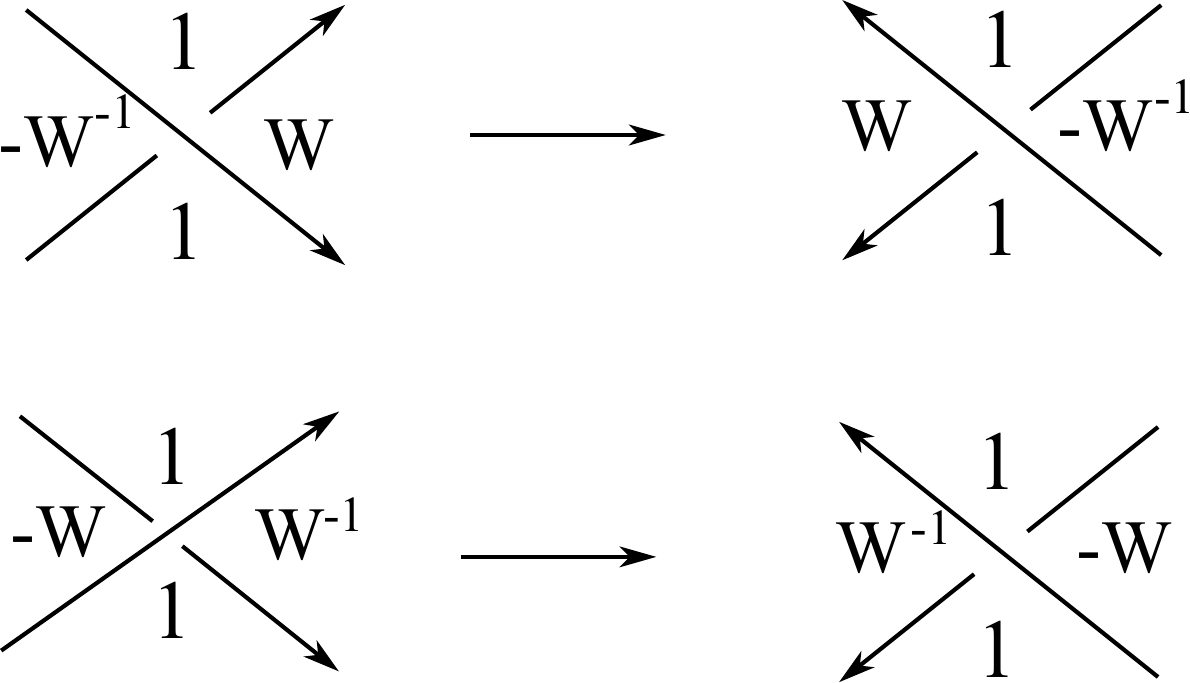}
\caption{The effect of reversing the orientation on region labels.}
\label{fig:reverse}
\end{figure}

\end{proof}

\begin{proposition}
Let $L$ be an oriented starred link or linkoid diagram in a surface. The mirror image of $L$, $L^*$ is obtained by only switching every crossing of $L$. We have the following relation between the mock Alexander polynomials of $L$ and $L^*$.

$$\nabla_{L^*}(W) =\nabla _{L} (W^{-1} ).$$

\end{proposition}

\begin{proof}
The change in local labels in the local regions adjacent to a crossing of $L$ by switching the crossing, is shown in Figure \ref{fig:mir}.
\begin{figure}[H]
\centering
\includegraphics[width=.35\textwidth]{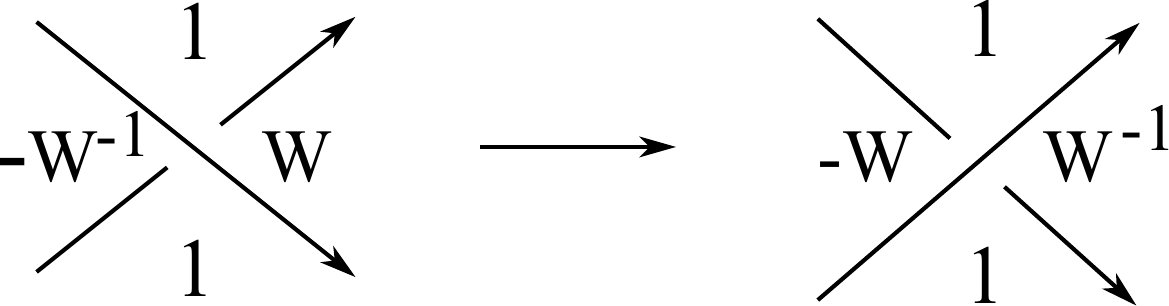}
\caption{The change of region labels under mirror symmetry.}
\label{fig:mir}
\end{figure}

\end{proof}

Here we present a topological specialization of Conjecture \ref{conj:conj1} for the mock Alexander polynomial of a knotoid in $S^2$.

\begin{conjecture}\label{conj:ii} \normalfont
Let $K$ be a knotoid diagram in $S^2$ with its tail lying in the exterior. Let $K_ {ext}$ denote the starred knotoid diagram with a star placed in the exterior region of $K$, and $K_{in}$ denote the starred knotoid diagram with a star placed in the region where the head of $K$ is located. Then, we have the following equality.\begin{center}

$\nabla_{K_{ext}} (W) = \nabla_{K_ {in}}  (-W^{-1})$.

\end{center}

\end{conjecture}
Note that  any knotoid diagram in $S^2$ can be represented by a diagram whose tail lies at the exterior region \cite{Turaev}. So, without loss of generality, we consider such representations of spherical knotoids. The proof  of Conjecture \ref{conj:ii} will appear in a joint paper by Wout Moltmaker, Jo-Ellis Monahan and the authors of the present paper.

The following corollary follows from Proposition \ref{prop:rev} and Conjecture \ref{conj:ii}. 

\begin{corollary}
Let $K$ be a knotoid diagram in $S^2$ and $\overline{K}$ be its reversion. Then,
$$\nabla^{\sharp}_{K}(W) = \nabla^{\sharp}_{\overline{K}} (W).$$

\end{corollary}

\subsection{Skein relations}\label{sec:skein}

\begin{definition}\normalfont
Let $L$ be a link or linkoid diagram in $S^2$. We call a crossing $v$ of $L$ \textit{separating} (or \textit{nugatory}) if at least one of the resolutions of $v$ results in a split diagram. Note that in the classical case, that is if $L$ is a classical link diagram, a nugatory crossing can be removed by isotopy. On the other hand, any crossing that is adjacent to an endpoint of a linkoid diagram is separating but not nugatory in the classical sense since it cannot be removed by an isotopy move. For example, the crossing of the linkoid diagram $L_+$ depicted in Figure \ref{fig:skeinhold}, is a separating  but not a nugatory crossing in the classical sense.
\end{definition}

\begin{theorem}\label{thm:skeinthm}
Let $L_{+}, L_{-}$   and $L_{0}$ denote three starred link diagrams in $S^2$  that differ from each other at exactly one crossing $v$, as shown in Figure \ref{fig:skein}.  Assume that all the regions that are incident to the crossing $v$ are mutually distinct regions in $L_{+}$ and $L_{-}$. Then, the skein relation 
$$\nabla_{L_{+}} (W)  - \nabla_{L_{-}} (W) = (W- W^{-1}) \nabla_{L_{0}} (W).$$

holds at $v$  if  
\begin{enumerate}[i.]
\item none of the regions that is incident to the crossing is starred or,
\item two  adjacent regions that are incident to the crossing are starred or,
\item both up and down regions that are incident to the crossing are starred. 
\end{enumerate}

\begin{figure}[H]
\centering
\includegraphics[width=.5\textwidth]{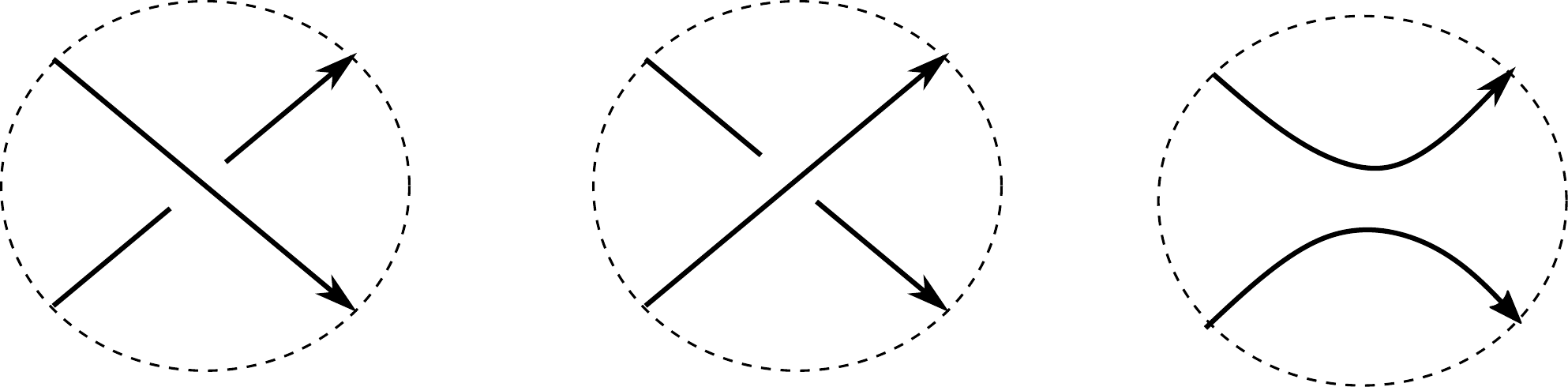}
\caption{ $L_{+}, L_{-}$  and $L_{0}$ from left to right.}
\label{fig:skein}
\end{figure}
\end{theorem}

\begin{proof}

We begin by making the observation that $L_{+}$ and $L_{-}$ admit the same universe. Let $U$ denote the universe of $L_{+}$, $L_{-}$,  and $U_{0}$ denote the universe of $L_{0}$.  Assume first that none of the regions incident to $v$ receives a star. This case is studied in \cite{FKT} for classical link diagrams. We extend the argument for this case of linkoids.

Since the regions incident to $v$ do not receive a star, the face of $U_{0}$ that is obtained by smoothing the crossing $v$ in $L_{+}$ in the oriented way, does not receive a star. Let us call this region $F$.
Let $\delta$ and $\delta_0$ be the sets of all states of $L_{+}$ (also of $L_{-}$) and $L_{0}$, respectively. Consider the partition of $\delta_0$ into two subsets, $\mathcal{L}$ and $\mathcal{R}$ where $\mathcal{L}$ denotes the collection of states where the state marker at the region $F$ lies at a crossing that lies on the left hand side of the former crossing $v$, and $\mathcal{R}$ denotes the collection of states where the state marker at the region $F$ lies at a crossing that is on the right hand side of $v$.
  
Consider now the partition of $\delta$ into the following four subsets: $\mathcal{L}_U$, where the state marker at $v$ is placed at its left,  $\mathcal{R}_U$ where the state marker at $v$ is placed at its right, $\mathcal{U}_U$ where the state marker at $v$ is placed up and $\mathcal{D}_U$ where the state marker is placed down.

There is a one-to-one correspondence  between $\mathcal{R}_U$, where the crossing $v$ receives its marker at its right and $\mathcal{L}$. Similarly, there is a one-to-one correspondence between $\mathcal{L}_U$ and $\mathcal{R}$. These correspondences result in the equalities in Figure \ref{fig:skein2} that induce the equalities in Figure \ref{fig:skeinpf} with the arguments above. The skein relation is obtained by subtracting the second equality in Figure \ref{fig:skeinpf} from the first one.

\begin{figure}[H]
\centering
\includegraphics[width=.5\textwidth]{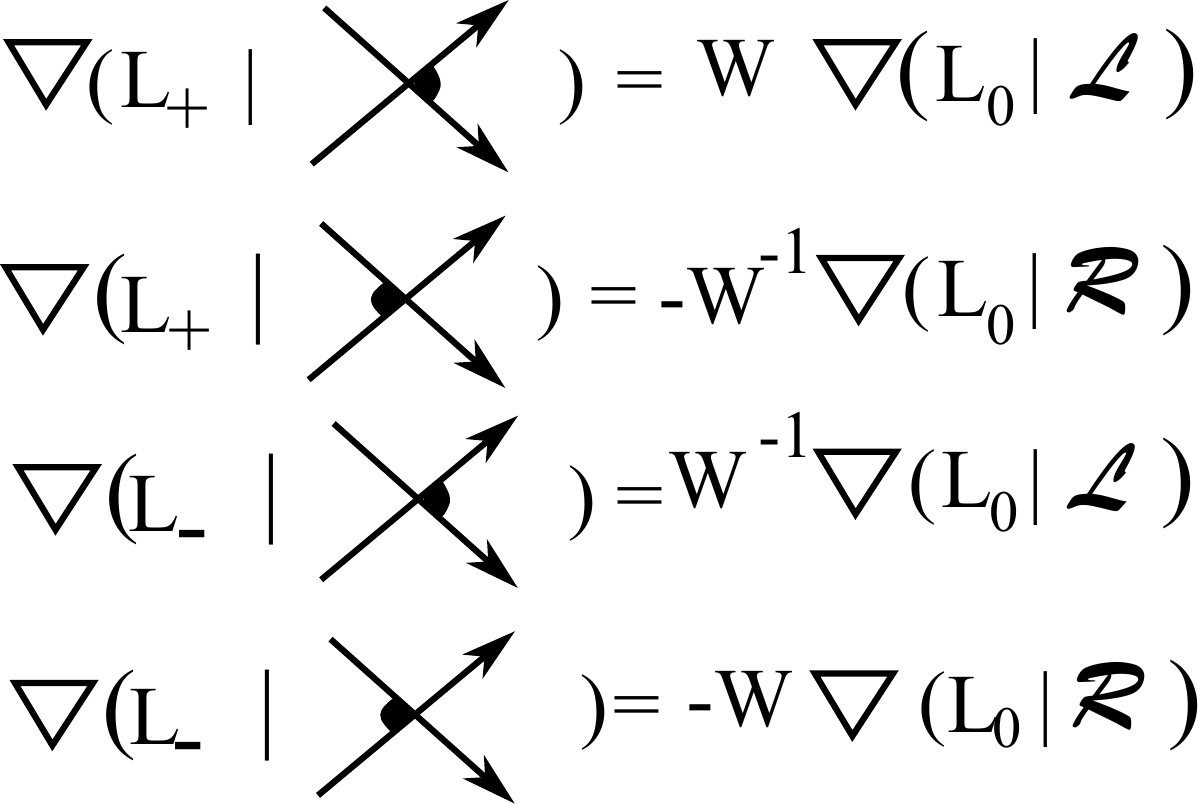}
\caption{The relations induced by different placement of markers at $v$.}
\label{fig:skein2}
\end{figure}

\begin{figure}[H]
\centering
\includegraphics[width=1\textwidth]{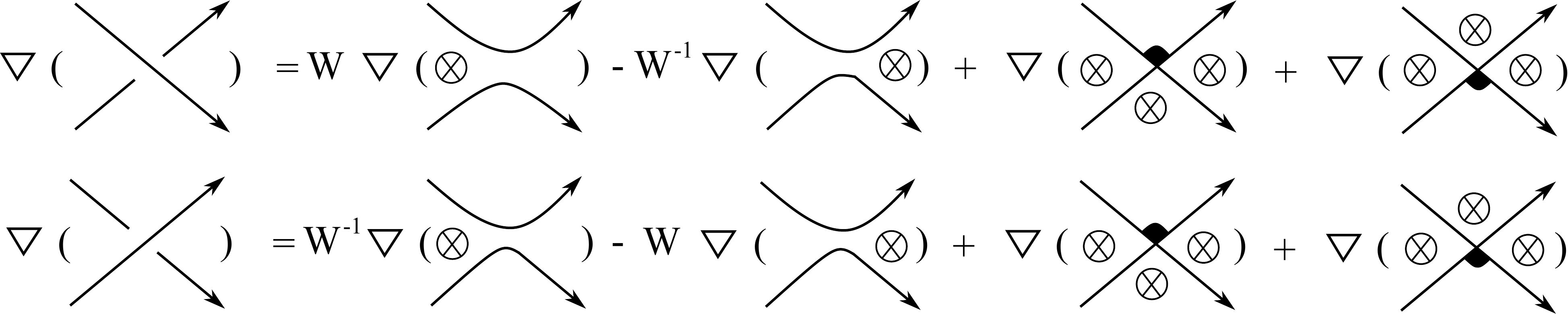}
\caption{State expansions at crossings far away from starred regions.}
\label{fig:skeinpf}
\end{figure}

Assume that \textit{ii.} holds. That is, any two of the regions incident to $v$ that are adjacent to each other, are starred. This implies the region $F$ of $L_{0}$ receives a star.  Without loss of generality, assume that the region that lies on left hand side of $v$ in $L_+$ (and so, in $L_{-}$) is starred. We observe the equalities given in Figure \ref{fig:skeinpf2} at $v$. By subtracting the second equality from the first one, we find the skein relation. (If the region on the right hand side of $v$ is starred then similar equalities hold as in Figure \ref{fig:skeinpf2}, with the difference in the coefficients. Precisely, in the first (top) equality we see $- W^{-1}$ instead of $W$, and $- W$ in the second (bottom) equality instead of $W^{-1}$. From this, we can deduce the skein relation.

\begin{figure}[H]
\centering
\includegraphics[width=.8\textwidth]{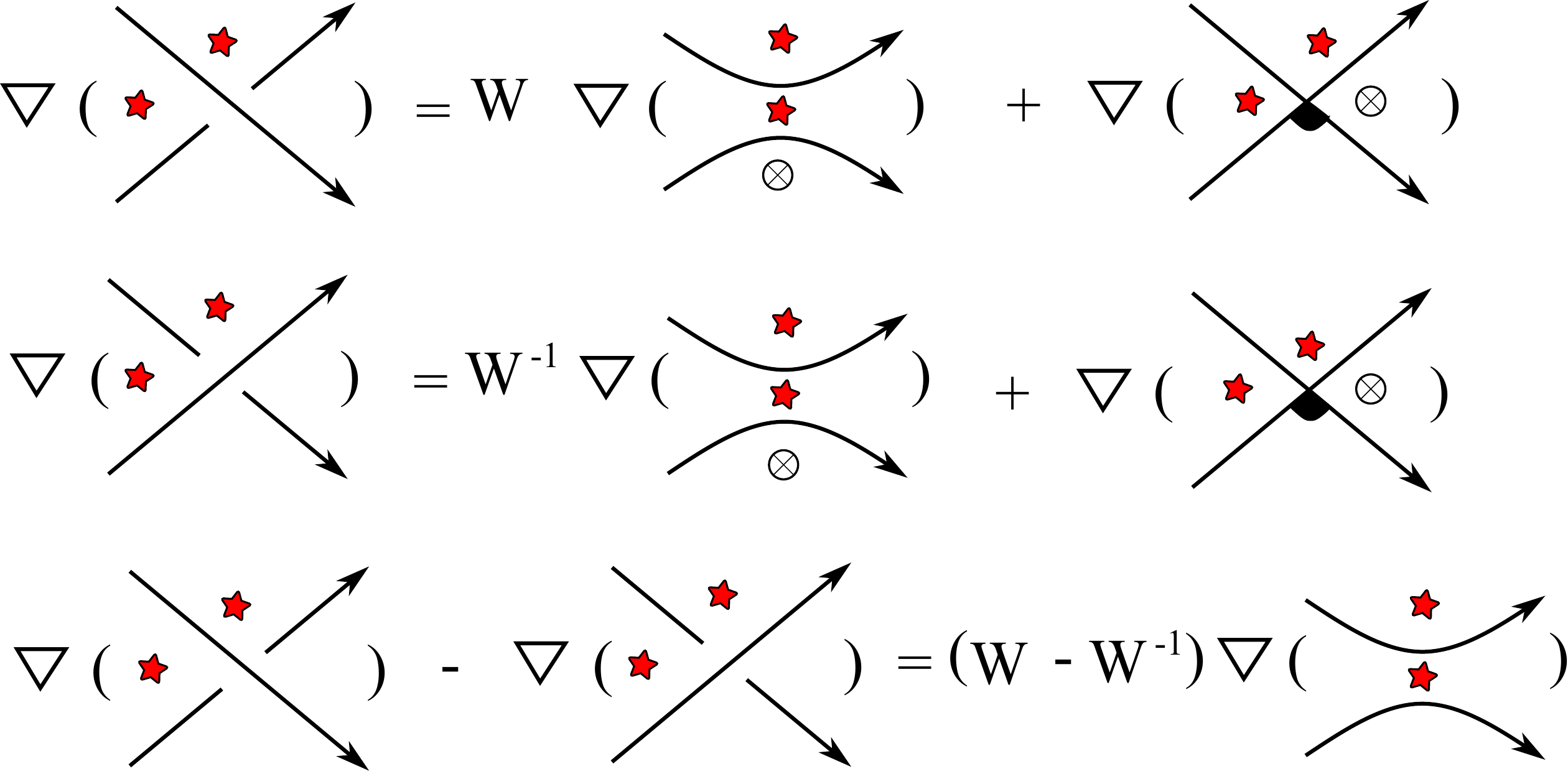}
\caption{State expansions at crossings incident to starred regions.}
\label{fig:skeinpf2}
\end{figure}

Now assume that \textit{iii.} holds for $v$. The observation made for \textit{i.} in Figure \ref{fig:skein2} directly applies to this case. Since there will be no state markers at top and bottom regions in any state of $L_{+}$ and $L_{-}$, in the skein state expansion at $v$ we only observe the first two terms that are given on the right hand side of the equalities in Figure \ref{fig:skeinpf} for both types of crossings. This results in the skein relation.

\end{proof}

%~~~SPIRAL EX~~~~~~~

\begin{proposition}
Let $L_{+}$ be a starred link diagram in $S^2$ and $v$ be the positive crossing of $L_{+}$ where $L_{+}$ differs both from $L_{-}$ and $L_0$, as shown in Figure \ref{fig:skein}.
Assume that the left hand side and the right hand side regions (not necessarily distinct regions) of  $v$ are starred. Then 
$$\nabla_{L_{+}}(W) = \nabla_{L_{-}}(W).$$
\end{proposition}

\begin{proof}

%If two of the neighboring regions that lie on the left and right of $v$ are starred then $\nabla_{L_{+}} (W)=\nabla _{L_{-}}(W)$ 
The collections of states of $L_{+}$ and $L_{-}$ can both be partitioned into two subsets of states, one which $v$ receives the state marker at the region that lies on top of $v$ and another which $v$ receives the state marker at the region that lies at the  bottom of $v$, and the weight of a state marker at the top or the bottom of $v$ is $1$. Then the mock Alexander polynomials of $L_{+}$ and $L_{-}$ are equal.
\end{proof}

\begin{proposition}\label{prop:split2}
Let $L_{+}$ be a starred link diagram in $S^2$ with stars lying on a pair of adjacent regions.
If the crossing $v$ is nugatory, then
$$\nabla_{L_{+}}(W) = \nabla_{L_{-}}(W).$$

 %and the starred regions in $L_+$ are adjacent.or the the regions on the left hand side and the right hand side are both starred.% the Mock Alexander polynomial of $L_+$ equals the Mock Alexander polynomial of $L_-$.

\end{proposition}

\begin{proof}
This case is discussed in \cite{FKT}.  We illustrate $L_{+}$ and $v$ in Figure \ref{fig:nugatory} where  $K_1$ and $K_2$ are the sublinks that can be obtained by separating $L_+$ at $v$.

\begin{figure}[H]
\centering
\includegraphics[width=.1\textwidth]{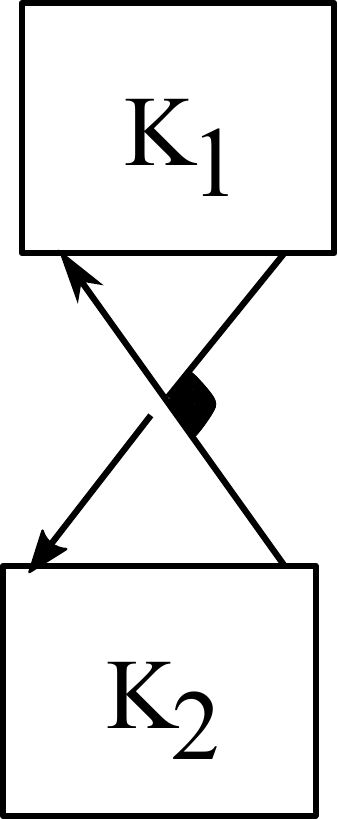}
\caption{An impossible state marker placement at a nugatory crossing}
\label{fig:nugatory}
\end{figure}

The main observation that is utilized for the proof is the following.
$L_{+}$ is assumed to be endowed with stars on a pair of adjacent regions. There is no state of $L_{+}$ that admits a state marker at $v$ lying on the left or the right  of $v$, so on the exterior region of $L_{+}$, as shown in the figure.  The reason for this is that the stars lie on adjacent regions and a state where $v$ receives a state market at its left (or at its right) forces the pair of stars to lie on a pair of interior regions of one of the sublinks that form $L_{+}$. This results in the corresponding sublink diagram, there are $3$ occupied regions (two of them with stars and one of them is the marked exterior region) and if it  has $n$ crossings, then it has $n-1$  available regions to be starred, which is contradictory to the definition of  state. Therefore, one of the stars is required to be placed on the left or right of the crossing $v$. 

From this observation, it follows that either the top or the bottom region that is incident to $v$ receives a state marker whose weight is $1$. Since this observation is also  true for $L_{-}$, we get the equality $\nabla_{L_{+}}(W) = \nabla_{L_{-}}(W).$

\end{proof}

\begin{remark} \normalfont
By Proposition \ref{prop:split2}, the Alexander polynomial of a split link is assumed to be zero. See \cite{FKT} for the discussion.
\end{remark}

\begin{proposition}\label{prop:separating1}
Let $L_{+} $ be a starred link diagram in $S^2$. Let $v$ be the crossing where $L_{+}, L_{-}, L_{0}$ differ from each other. Assume that the crossing $v$ is the unique nugatory crossing of $L_{+}$ and the stars lie on a pair of non-adjacent regions in $L_+$.

\begin{enumerate}[i.]
\item If the stars lie on two interior regions that belong to only one of the sublinks of $L_{+}$, that is, the stars are on either  $K_1$ or $K_2$ of Figure \ref{fig:nugatory}, or if one of the stars lie on the exterior region, then $\nabla_{L_{+}} (W) = \nabla_{L_{-}}(W)$.
\item If the stars lie on two interior regions, one star on $K_{1}$ and the other star on $K_{2}$, then 
$\nabla_{L_{+}} (W) - \nabla_{L_{-}} (W) = (W - W^{-1}) \nabla_{L_{0}}(W)$.

\end{enumerate}
\end{proposition}
\begin{proof}
\begin{enumerate}[i.]
\item Assume first that the stars are placed on two nonadjacent interior regions of $K_1$. In this case, in any state of $L_+$,  the exterior region of $K_1$ needs to get a state marker in one of the neighboring crossings of $K_1$. This forces the placement of the state marker at the nugatory crossing $v$ is either up or down the crossing $v$. Thus, a state marker at the nugatory crossing contributes to $\nabla_{L_{+}}$ trivially. This is also true for $L_{-}$. Since the states of $L_{+}$ and $L_{-}$ coincide, the equality follows. The equality is also true for the latter case: If the exterior region is occupied with a star, then the state marker placement at $v$ is either up or down the crossing. The rest follows similarly. 

\item In this case, a state marker at the nugatory crossing $v$ can be placed on any side of $v$ (up, down, left or right). If the state marker lies at the right of $v$ in a state of $L_{+}$ then its weight is $W$. Let $P_1$ be the sum of the product of weights of state markers in states where the state marker at $v$ is at the right of $v$ then the total contribution of all such states to $\nabla_{L_{+}}$ is $WP_{1}$. It is clear that the states where the state marker at $v$ is at the right of $v$ coincide with the states where the state marker at $v$ is at the left of $v$, except for the state marker at $v$. Then, the total contribution of all states where the state marker at $v$ is at the left of $v$, to $\nabla_{L_{+}}$ is $-W^{-1}P_{1}$. Let $P_2$ and $P_3$ denote the sums of the product of weights of state markers in states where the state marker at $v$ is at the top and bottom of $v$, respectively. The weight of a state marker at the top or bottom of $v$ is $1$. We find that $\nabla_{L_{+}} = (W -W^{-1}) P_{1}  + P_{2} + P_{3} $.

States of $L_{+}$ clearly coincide with states of $L_{-}$. Similar arguments hold for $L_{-}$ and we find that $\nabla_{L_{-}} = (W^{-1} - W) P_{1} +P_{2} + P_{3}$. Then $\nabla_{L_{+}} (W) - \nabla_{L_{-}}  (W) = 2(W - W^{-1}) P_{1}$.

On the other hand, $L_{0}$ is a split link diagram. Let $\tilde{L_{0}}$ denote the connected link diagram that is obtained from $L_{0}$ by an RII move. As shown in Figure \ref{fig:RIIconnecting}, there are in total six possible placements of state markers at the crossings in the RII move site when none of the regions adjacent to the strands in the move site is endowed with a star. The states I, III, IV, VI of  $\tilde{L_{0}}$  depicted in Figure \ref{fig:RIIconnecting} coincide with the states of $L_{+}$ (also with $L_{-}$) where the state marker at $v$ is at the top or the bottom of $v$. Then, the total contributions of these states is $ W P_{2} - W P_{2} + W P_{3} - W P_{3} = 0$. The states II and V contribute as $2 P_{1}$, where $P_1$ is the sum of the product of weights of state markers in states of $L_{+}$ where the state marker at $v$ is at the right  (or left) of $v$. Then,  
$\nabla_{L_0} (W) = \tilde{L_{0}} (W) = 2 P_{1}$. Therefore,  $\nabla_{L_{+}} (W) - \nabla_{L_{-}} (W) = (W -W^{-1}) \nabla_{L_0} (W)$. 
\end{enumerate}

\end{proof}

\begin{figure}[H]
\centering
\includegraphics[width=.5\textwidth]{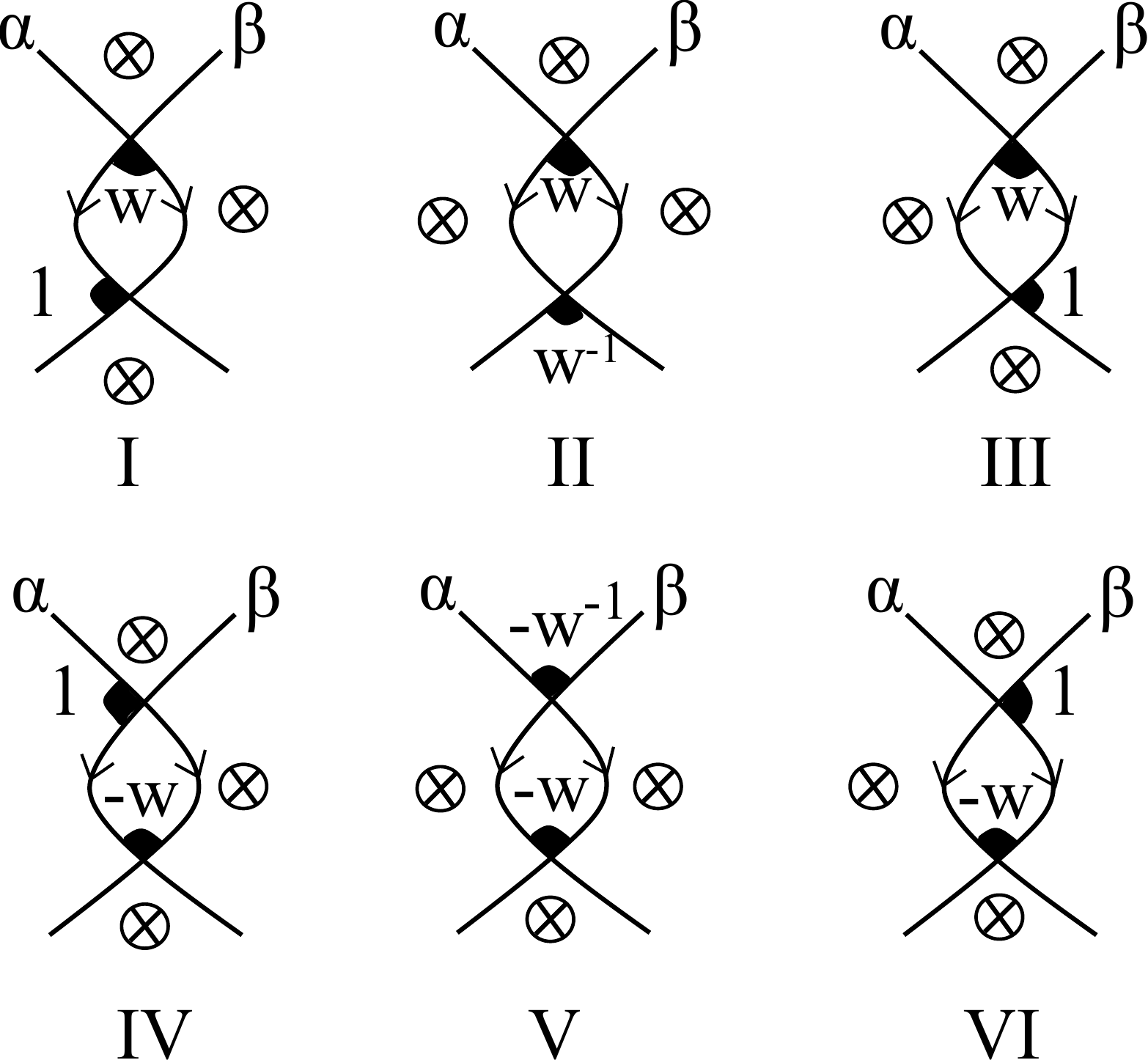}
\caption{The local states in an RII move site connecting disjoint knots $\alpha$, $\beta$.}
\label{fig:RIIconnecting}
\end{figure}

%\begin{remark}
%If $L^{*}_+$ contains more than one separating crossing, Proposition \cite{prop:separating} does not necessarily hold. See an example for this in Figure \ref{fig:separatingmore}.

%\end{remark}

\begin{theorem}
Let $L_{+}, L_{-}$   and $L_{0}$ denote three starred linkoid diagrams in $S^2$ which differ from each other at a unique crossing $v$, as shown in Figure \ref{fig:skein}. Then the skein relation
$$\nabla_{L_{+}} (W) - \nabla_{L_{-}} (W) = (W - W^{-1}) \nabla_{L_{0}}(W)$$
holds  if $v$ is not a separating crossing.

\end{theorem}

\begin{proof}

Let $v$ be a crossing in $L_+$ that is not starred and and not separating. Then either four of the regions that are adjacent to $v$ are mutually distinct or two of them are the same regions if $v$ is incident to an endpoint. But either of these cases, similar equalities that appear in the proof of Theorem \ref{thm:skeinthm} hold at $v$. Note that only one of the regions that are adjacent to $v$ may be starred since $L_+$ is a linkoid diagram. The details of the proof are left to the reader.

\end{proof}

\begin{remark}\normalfont
A linkoid diagram may have a separating crossing but still satisfies the skein relation at such crossing.
See the linkoid diagram given in Figure \ref{fig:skeinhold} with one and only separating crossing. Notice that $L_0$ is a split linkoid diagram with two components and is not an admissible diagram. Transforming $L_0$ to an admissible diagram $\tilde{L}_{0}$ by an RII move is possible as in Figure \ref{fig:skeinhold}. Then, direct computation shows that 
$\nabla_{\tilde{L}_{0}}(W) = 2$. Therefore, $\nabla_{L_{0}} (W)= 2$. One can verify easily that $$\nabla_{L_{+}}(W) - \nabla_{L_{-}} (W)= 2(W - W^{-1}) = (W - W^{-1}) \nabla_{L_{0}}(W).$$
\begin{figure}[H]
\centering
\includegraphics[width=.5\textwidth]{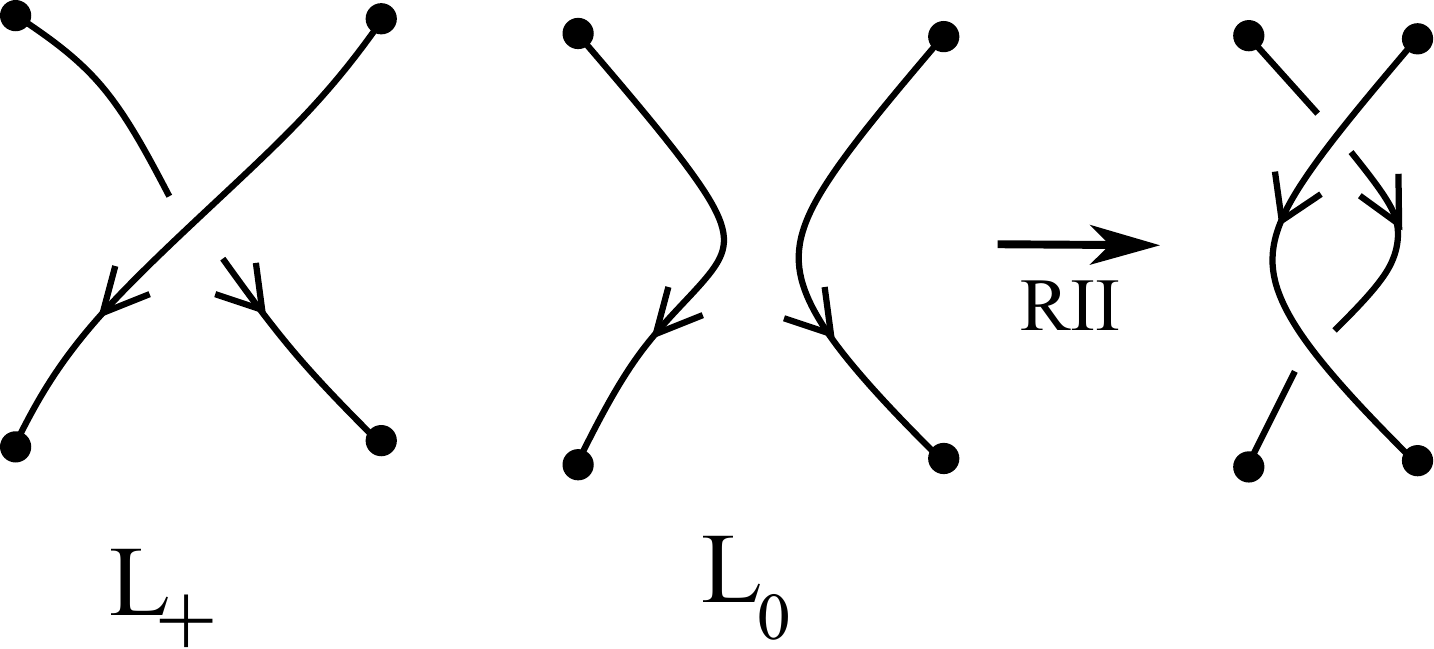}
\caption{A linkoid diagram with a separable crossing and its resolution.}
\label{fig:skeinhold}
\end{figure}

\end{remark}

\section{More on starred link diagrams}\label{sec:more}
\subsection{Admissible link diagrams in higher genus surfaces}
Let $L$ be a a connected link diagram that is tightly embedded in some surface of genus $g$. We know that the equality $f= 2 - 2g +n$ holds for $L$, where $f$ denotes the number of regions and $n$ denotes the number of crossings of $L$. In Section \ref{sec:decorated}, we discuss decorating regions or crossings with stars to obtain the equality $f = n$ for the regions and crossings without stars. See Figure \ref{fig:tori} for two examples, where we illustrate knot diagrams in a torus and double torus.
 \begin{figure}[H]
\centering
\includegraphics[width=1\textwidth]{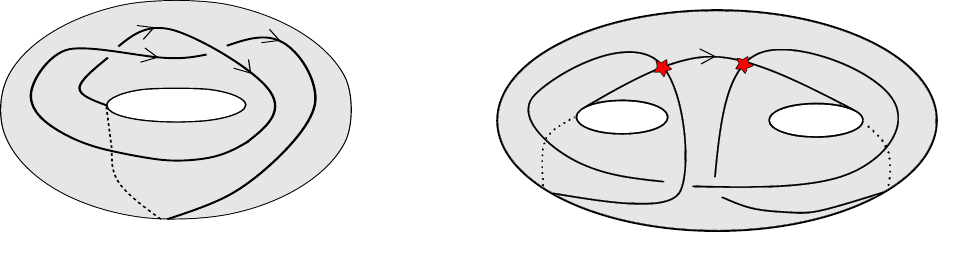}
\caption{Admissible knots in a torus and a double torus}
\label{fig:tori}
\end{figure}

The number of regions of $L$ can be also reduced by two by adding handles to $S^2$ that connect either a triple of regions or two distinct pairs of regions of $L$. The resulting link diagrams are admissible link diagrams in genus surface as it is described below.  We first discuss adding handles to $S^2$ that connect a triple of regions of $L$.

Let $\tau$ denote $S^2 - \bigcup^{3}_{i=1} D_i$, where $D_i$ is an open disk in $S^2$.  We call $\tau$ a \textit{trident surface}. Choose any three distinct regions in $L$ and delete an open disk from each of these regions.  A trident surface is added to $S^2$ by gluing its boundary components to the boundaries of the deleted disks, as illustrated in Figure \ref{fig:trident2}. As a result, $L$ is regarded as a link diagram in a genus two surface. It is not hard to observe that the chosen triple of regions of $L$  are connected by the trident surface and so represents  the same region of $L$ (named as $T$ in the figure) that lies in the genus two surface.  This implies the total number of regions in $L$ that lies in the resulting genus two surface is equal the number of its crossings. That is, $L$ becomes admissible in the genus two surface.

\begin{figure}[H]
\centering
\includegraphics[width=.75\textwidth]{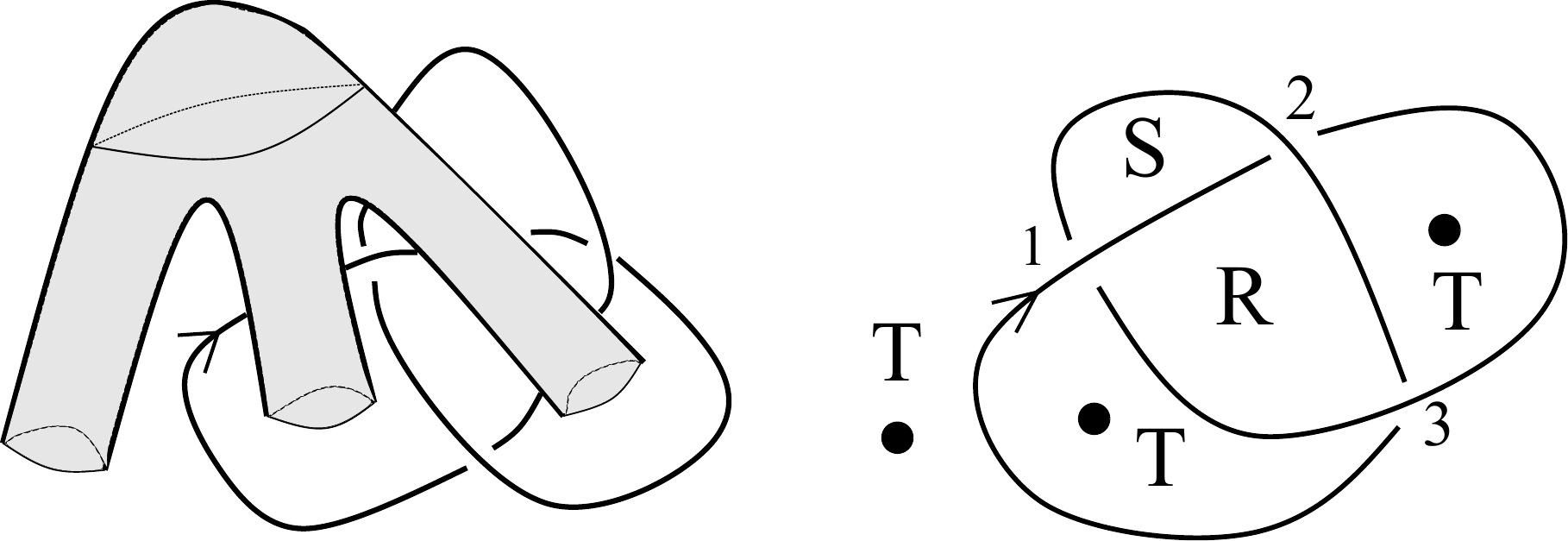}
\caption{Attaching handles onto a triple of regions of the trefoil diagram}
\label{fig:trident2}
\end{figure}

One another way to obtain an admissible link diagram in a genus two surface is to choose two distinct pairs of regions of $L$ in $S^2$ (so in total four distinct regions), delete out an open disk from each of the regions and connect these regions pair-wisely by attaching two disjoint handles through the boundaries of the disks.  The resulting link diagram is clearly an admissible link diagram in a genus two surface. 

\begin{figure}[H]
\centering
\includegraphics[width=.4\textwidth]{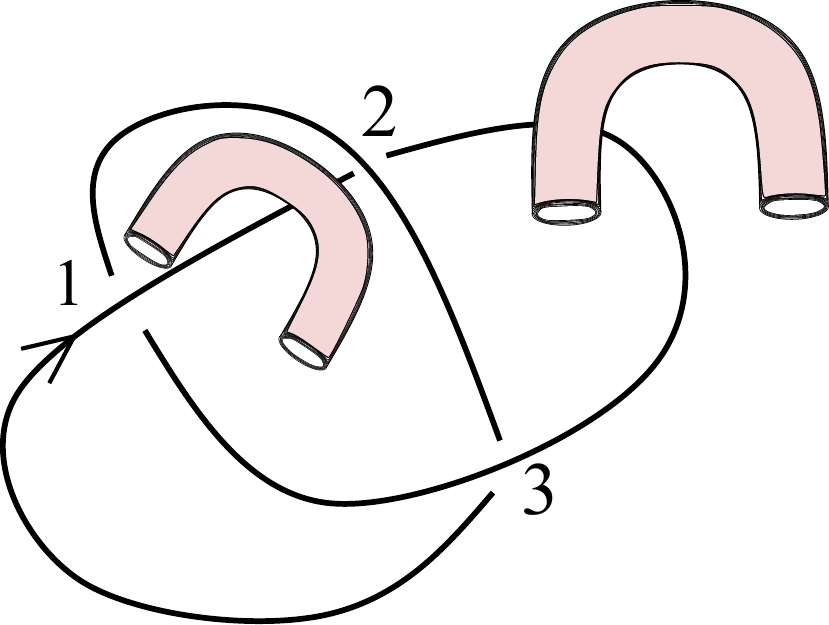}
\caption{Attaching disjoint handles on pairs of distinct regions of the trefoil diagram}
\label{fig:handle}
\end{figure}

These constructions of admissible link diagrams in genus two surfaces bring us to the question of classification of admissible link diagrams in higher genus surfaces. We answer to the question by the following theorem but we shall first give a necessary definition.

\begin{definition}\normalfont
Let $L$ be a connected link diagram that is admissible  in a surface of genus $g$.
Assume that all extraneous handles have been removed from the surface. This means that all faces with one boundary component are disks, and that any face with multiple boundary components consists in handles attached to disks in such a way
 that removal of any handle will disconnect the region.We call such a surface a \textit{minimal surface} for the link diagram $L$.

\end{definition}

\begin{theorem}
If $L$ is an admissible link diagram that lies in a surface $\Sigma_g$ of genus $g>0$ and $\Sigma_g$ is minimal for $L$, then $g = 1$ or $g= 2$. If $g= 2$, then the embedding can be obtained from a link diagram $L$ in $S^2$ by adding some handles in either the form of three disk regions made into one region (attaching a trident surface) or two distinct pairs of disk regions made into two regions by connecting pairs by two disjoint handles.

\end{theorem}

\begin{proof}
Assume $L$ lies in its minimal genus surface in $\Sigma_g$ and that we can reduce the genus by cutting handles. This will increase $f$ by one for each handle we cut. We then obtain $L$ in $S_h$ where $h < g$ where it is a tight diagram with $f'<f
$ faces and $f’>n$. By Euler's formula, we know that $f’ - v = 2 - 2h$. Since $f’ - v > 0$, we have $2 - 2h > 0$ and so $h < 1$, whence $h = 0$. Thus $L$ in its minimal surface $\Sigma_g$ is obtained from a diagram $L$ in the 2-sphere $S^2$ by adding handles. We know that in $S^2$,  $f’ - v = 2$ and so we can only add two 1-handles. This yields the genus two case.  If $L$ is admissible in its minimal surface $\Sigma_g$, but there is no reduction by cutting handles, then all regions of $L$ are disks (that is, $L$ is tightly embedded in $\Sigma_g$ and we have $f - v = 0 = 2 - 2g$ implies $g =1$.

\end{proof}

\subsection{Two variants of the mock Alexander Polynomial}

\begin{definition}\normalfont
Let $L$ be a classical link diagram in $S^2$ and $\tilde{L}$ be the corresponding link diagram in a genus two surface that is obtained by adding a trident surface to $S^2$ that connects a triple of regions in $L$. The \textit{trident polynomial} of $\tilde{L}$ , $T_{\tilde{L}}$ is defined to be the mock Alexander polynomial of $\tilde{L}$.

\end{definition}

\begin{corollary}[Equivalent definition]
The trident polynomial of $\tilde{L}$,  is the permanent of the potential matrix of $\tilde{L}$.

\end{corollary}

\begin{remark}\normalfont

The link diagram $L$ in $S^2$ can be considered as the planar representation of $\tilde{L}$ with the addition of nodes on the chosen triple of regions of $L$ that are connected by two handles to obtain $\tilde{L}$. Let $L_1$ and $L_2$ be two connected (classical) link diagrams in $S^2$ with nodes on triples of regions for each. $L_1$ and $L_2$ represent the same link in the genus two surface obtained by attaching two handles to $S^2$ if and only if one can be transformed to the other by a sequence of Reidemeister moves and spherical isotopy moves that take place far from the regions with nodes.

It is more convenient  to obtain the potential matrix of $\tilde{L}$ by using the planar representation, and  so to compute the trident polynomial of $\tilde{L}$.  The potential matrix of the trefoil diagram with columns determined by the regions $R,S,T$ in order, is given as follows.

$$
\begin{bmatrix}
1 & W & -W^{-1} +1\\
1 & -W^{-1} & W+1\\
1 &0 & W -W^{-1}+1\\
\end{bmatrix}
$$
The permanent calculation shows that $T_{\tilde{L}} (W) = 2(W^2 + W^{-2}) + 2(W -W^{-1}) -2$.
\end{remark}

%Let $L_1$ and $L_2$ be two connected link diagrams in $S^2$ and $\tilde{L_1}$ and $\tilde{L_2}$ be the corresponding link. diagrams that lie in genus two surface obtained by attaching handles to $S^2$ on triples of regions of $L_1$ and $L_2$. $\tilde{L_1}$ and $\tilde{L_2}$ are \textit{equivalent} if they are related to one another by a sequence of Reidemeister moves in the surface and isotopy of the surface.

%Two classical  link diagrams that lie in a genus two surface obtained by attaching two handles , are considered as \textit{equivalent} up to the isotopy relation induced by Reidemeister moves in $\Sigma_g$ and surface isotopy. \cite{}. 

%Isotopy moves acting on $\tilde{L}$  transform into Reidemeister moves and planar isotopy restricted by the nodes on the planar representation of $\tilde{L}$. By planar isotopy It is not allowed the move strands through the nodes.) 

\begin{theorem}

The trident polynomial is an invariant of classical connected link diagrams that is embedded in a genus two surface obtained by adding a trident surface connecting three regions of classical link diagrams in $S^2$.

\end{theorem}

\begin{proof}
This follows by the definition of the trident polynomial.

\end{proof}

In a similar manner, one can define the following variant of the Mock Alexander polynomial for admissible linkd diagrams in a genus two surface, due to the second handle attaching construction of an admissible diagram in a genus two surface. 

\begin{definition}\normalfont
Let $L$ be a classical link diagram in $S^2$ and $\tilde{L}$ be the corresponding link diagram in a genus two surface that is obtained by adding two disjoint handles to $S^2$ that connects two distinct pairs of regions of $L$.
The \textit{handle polynomial} of $\tilde{L}$ , $T_{\tilde{L}}$ is defined to be the Mock Alexander polynomial of $\tilde{L}$.
\end{definition}

\begin{corollary}[Equivalent definition]
The handle polynomial of $\tilde{L}$,  is the permanent of the potential matrix of $\tilde{L}$.

\end{corollary}

\begin{theorem}

The handle polynomial is an invariant of classical connected link diagrams that is embedded in a genus two surface obtained by adding a trident surface connecting three regions of classical link diagrams in $S^2$.

\end{theorem}

\begin{proof}
This follows by the definition of the handle polynomial.

\end{proof}

\begin{Remark}\normalfont
\begin{enumerate}
\item The link diagram that lies in the genus two surface obtained by attaching the trident to $S^2$ connecting three distinct regions of the diagram in $S^2$, is clearly not a tight embedding as the connected regions contain genus.
\item A similar construction applies to knotoid diagrams in $S^2$ as follows. Choose a pair of regions in a knotoid diagram in $S^2$ and delete a pair of open disks out from these regions. Then attach a handle to $S^2$ over the boundaries of the deleted disks. In this way, one obtains a knotoid diagram immersed in torus which is an admissible diagram since two distinct regions in $S^2$ turn out to be the same region in the torus. Then, the Mock Alexander polynomial is an available invariant for knotoids immersed in torus. We plan to work on this observations further in a subsequent paper.

\end{enumerate}
\end{Remark}

\subsubsection{Virtual closure of knotoids}
Let $K$ be a knotoid diagram in $S^2$. The virtual closure of $K$ is obtained by connecting two endpoints of $K$ with a simple arc that may intersect $K$. Each intersection of the connecting arc with $K$ is declared to be a virtual crossing so that the resulting closed diagram is a virtual knot diagram in $S^2$. The virtual closure of any knotoid diagram can be represented in a torus by attaching a handle to $S^2$ that holds the connecting arc realizing the closure. Figure \ref{fig:virtualclosure} depicts the virtual closure of a knotoid diagram and its torus representation.

 \begin{figure}[H]
\centering
\includegraphics[width=1\textwidth]{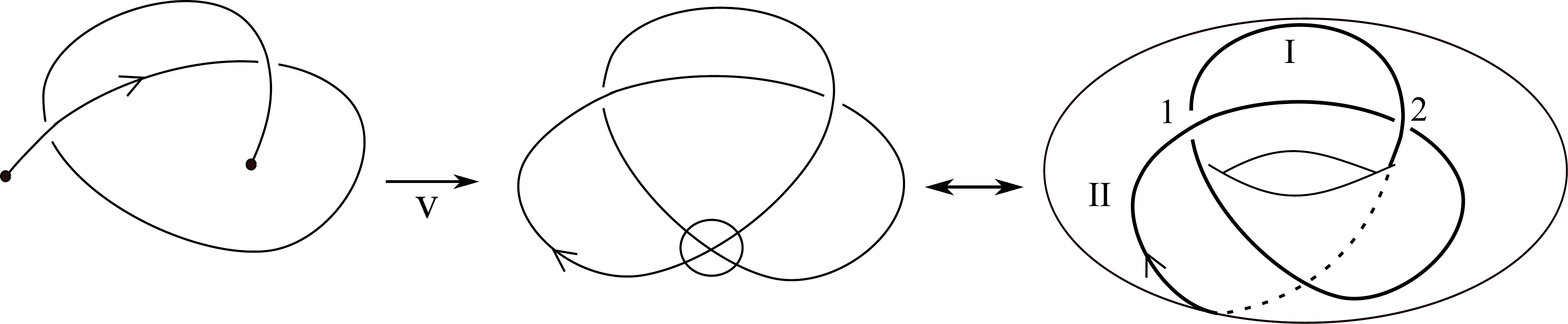}
\caption{The virtual closure of a knotoid}
\label{fig:virtualclosure}
\end{figure}

 The torus representation of the virtual closure of any knotoid diagram is an admissible diagram in the torus. The potential matrix of the resulting knot $v(K)$ in the torus, appearing in Figure \ref{fig:virtualclosure} is given as 
 $
\begin{bmatrix}
-W^{-1} & W+2\\
W& 2-W^{-1} \\
\end{bmatrix}.
$ By direct calculation of the permanent, we find that $\nabla_{v(K)} (W) = W^2 + W^{-2}+ 2( W-W^{-1}).$  Notice that in this case,  the mock Alexander polynomial of $v(K)$ is in fact, equal to $ \nabla_{K_{ext}}(W) + \nabla _{K_{int}} (W)$. See Example \ref{ex:simpleknotoid} where we calculate $\nabla^{\sharp}_{K}(W) = \nabla_{K_{ext}} (W)$
\begin{lemma}
Let $K$ be a knotoid diagram in $S^2$ whose tail lies at the exterior region and $v(K)$ denote the torus representation of its virtual closure. Then $$\nabla_{v(K)} (W)= \nabla _{K_{ext}} (W) + \nabla_{K_{int}}(W).$$

\end{lemma}

\begin{proof}
Let $M_{K_{int}}$ and $M_{K_{ext}}$ be the potential matrices of $K_{int}$ and $K_{ext}$ with size $n \times n$, respectively. $M_{K_{int}}$ and $M_{K_{ext}}$ coincide except for one unique column which represents the exterior and the interior regions of $K_{int}$ and $K_{ext}$, respectively. Without loss of generality, we can assume these columns are the $n^{th}$ columns of the matrices.  It is clear that the potential matrix of $M_{v(K)}$ consists of the same columns with $M_{K_{int}}$ and $M_{K_{ext}}$ except its  $n^{th}$ column is the sum of the $n^{th}$ columns of  $M_{K_{int}}$ and $M_{K_{ext}}$. From this, it follows that $\nabla_{v(K)} (W)= \nabla _{K_{ext}} (W) + \nabla_{K_{int}}(W)$.
\end{proof}
 
 From Conjecture \ref{conj:ii}, we obtain the following equality.
\begin{proposition} 
Let $K$ be a knotoid in $S^2$ and $v(K)$ its virtual closure.

$$\nabla_{v(K)} (W) = \nabla^{\sharp}_{K} (W) + \nabla^{\sharp} _{K} (-W^{-1}).$$

\end{proposition}

\begin{corollary}\label{cor:notclosure}
The potential matrix of the virtual knot in torus, depicted in Figure \ref{fig:notvirtualclosure} is,$$\begin{bmatrix}
W&-W^{-1} &2\\
W^{-1}&W& 2 \\
0&1&1+W-W^{-1}\\
\end{bmatrix}.
$$

It can be verified by direct permanent calculation that $\nabla_{K} (W) = 2W^2 + W- W ^{-1} + W^{-2} - W^{-3} -1$.  $\nabla_K$ cannot be written as the sum $\nabla_{K} (W) + \nabla _{K} (-W^{-1})$ for some knotoid $K$, as the symmetric term for $W^{-3}$, $-W^3$ is missing in the polynomial. Thus, the virtual knot depicted in Figure \ref{fig:notvirtualclosure} is not the virtual closure of a knotoid. 

\end{corollary}

 \begin{figure}[H]
\centering
\includegraphics[width=.3\textwidth]{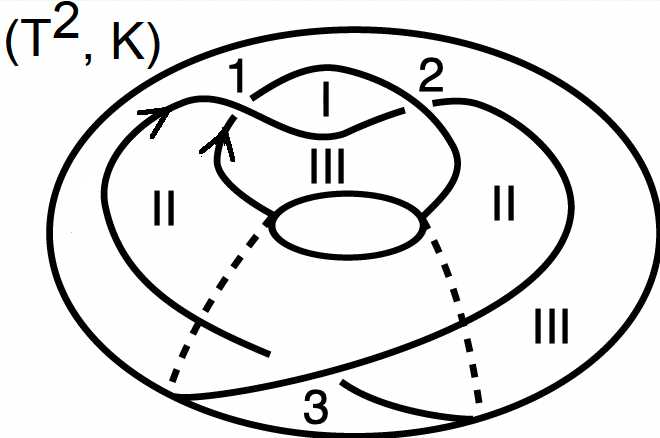}
\caption{A virtual knot in $T^2$ which is not the virtual closure of a knotoid}
\label{fig:notvirtualclosure}
\end{figure}

\begin{remark}
It was shown in \cite{GK2} that there are infinitely many virtual knots of virtual genus one obtained by virtualization and that do not lie in the image of the virtual closure map. This was shown by utilizing the surface bracket polynomial of virtual knots \cite{Dye}. As Corollary \ref{cor:notclosure} shows,  $\nabla_{v(K)}$ can be also used to determine whether a virtual knot of virtual genus one, is the virtual closure of a knotoid in $S^2$ or not.  

\end{remark}

The virtual closure map $v$ is  a well-defined map on the set of knotoids in $S^2$ to the set of virtual knots of virtual genus at most $1$. This implies that any virtual knot invariant gives rise to an invariant of knotoids in $S^2$ via the virtual closure map \cite{GK1}. Therefore, for every knotoid diagram $K$ in $S^2$, the mock Alexander polynomial of the torus representation of  the virtual closure of $K$ can be considered an invariant of $K$. 

\begin{definition}\normalfont
Let $K$ be a knotoid diagram in $S^2$, and $v(K)$ be the virtual closure of $K$, represented in a torus. The virtual Alexander polynomial of $K$, $\nabla_{K}^{v}$ is defined to be the mock Alexander polynomial of $v(K)$. For any knotoid $K$,  $\nabla_{K}^{v}$ is an invariant of $K$.

\end{definition}

\subsection{Starred links vs. links in handlebodies}
 In Section \ref{sec:decorated}, we have introduced starred link diagrams  in surfaces as link diagrams admitting stars on a number their regions or crossings so that the diagrams satisfy the equality $f=n$. In this section, we present the notion of a starred link diagram that lies in $\mathbb{R}^2$ in a more general sense.
 
 \begin{definition}\normalfont
   A \textit{region starred link diagram} in $\mathbb{R}^2$ is a classical link diagram with $n$ crossings such that $g$ of its regions, where $0 \leq g \leq n+2$, admit a number of stars.  A region starred link diagram is not necessarily an admissible diagram and some of its regions may get multiple stars.
 We assume that two region starred link diagrams are \textit{equivalent} if there is a sequence of star equivalence moves and planar isotopy moves transforming one to another.

A \textit{region starred link} in $\mathbb{R}^2$ is an equivalence class of starred link diagrams in $\mathbb{R}^2$ taken up to star equivalence.

 \end{definition}
 
Let $L$ be a link diagram in $\mathbb{R}^2$ with $g \geq 0$ of its regions are starred.  $L$ represents a unique link diagram embedded in genus $g$ handlebody by considering $L$ in $\mathbb{R}^2 \times I$, where $I$ is the unit interval $[0,1]$, and drilling holes in $\mathbb{R}^2 \times I$ through the starred regions. Conversely, a link diagram in a genus $g$ handlebody uniquely represents a starred link diagram in $\mathbb{R}^2$ by taking the projection of the link diagram onto  $\mathbb{R}^2$, and representing each genus as a star in a region. This correspondence induces a bijection between the set of  all region starred links in $\mathbb{R}^2$ and the set of links in a handlebody of genus $g$. This correspondence was initially utilized in \cite{Bat} for the construction of the Jones polynomial of links in handlebodies.

\section{Discussion}\label{sec:discussion}

The results of the present paper arise from generalizing the methods of \cite{FKT}. This paper is the first of a series of papers on these generalizations, which will include wider versions of the Clock Theorem of Formal Knot Theory  and other applications of these ideas.

All the invariants discussed in this paper are expressed by state summations that can be expressed as permanents of matrices associated with the diagram of a knot, link, linkoid or knotoid. In principle, permanents are of complexity higher than the polynomial time algorithms that compute determinants. At this time it remains an open problem to determine the complexity type of our permanent based invariants. This problem will be the subject of a paper subsequent to the present work. See \cite{Jaeger} where complexity results about the Jones polynomial are proved.

 Consider the starred knotoid diagram and its states depicted in Figure \ref{fig:labeltwo}. Given a state of $K$, each crossing can be resolved in the direction of the state. This gives an \textit{Eulerian trail} underlying $K$ that is a path from one endpoint to another, visiting all edges of $K$ exactly once.  In fact, there is a one-to-one correspondence between the states of a starred knotoid diagram and the Eulerian trails underlying the graph of the knotoid diagrams. Moreover, each Eulerian trail determines a unique spanning tree of a generalized Tait graph of $K$. In Figure \ref{fig:tree}, we illustrate all Eulerian trails and the corresponding spanning trees (in blue) of $K$.
  \begin{figure}[H]
\centering
\includegraphics[width=1\textwidth]{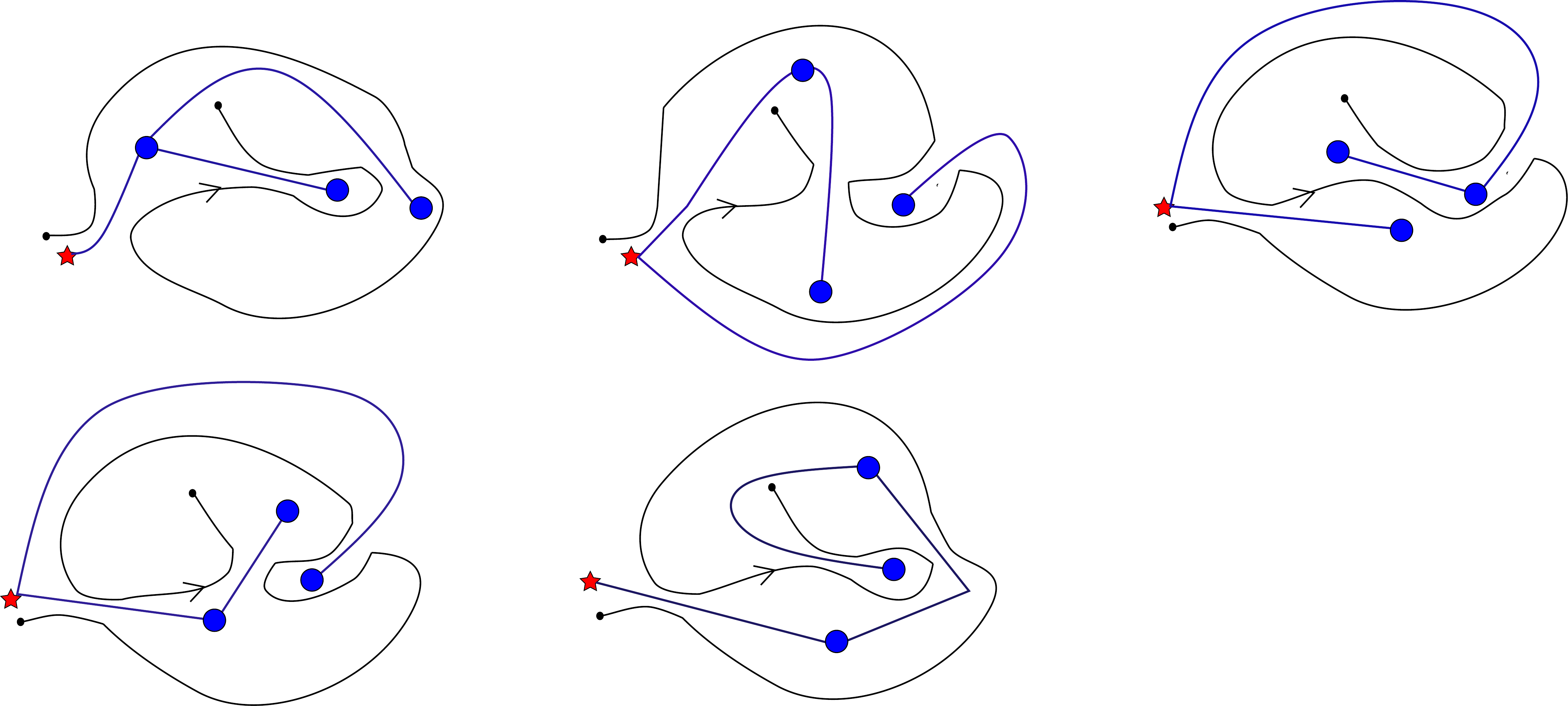}
\caption{Euler trails and spanning trees determined by states of $K$.}
\label{fig:tree}
\end{figure}
 There are versions of Heegaard Floer Link homology that are based on the correspondence between the Formal Knot Theory states and spanning trees of a link diagram. See \cite{Levine, Ozsvath, KrizKriz}.
We are in the process of using our generalizations of these states to study new versions of link homology for knotoids and their relatives as discussed in the present paper.

\section{Appendix}\label{sec:appendix}
Consider a general labeling at crossings of a starred link or linkoid diagram with commuting variables $U,D,L,R, U^{'}, D^{'}, L^{'}, R^{'}$, as given in Figure \ref{fig:apendix}, and a generalized state-sum with respect to such labeling. We obtain the following conditions on the labels, if we impose invariance under an RII move on the state-sum.
\begin{figure}[H]
\centering
\includegraphics[width=.45\textwidth]{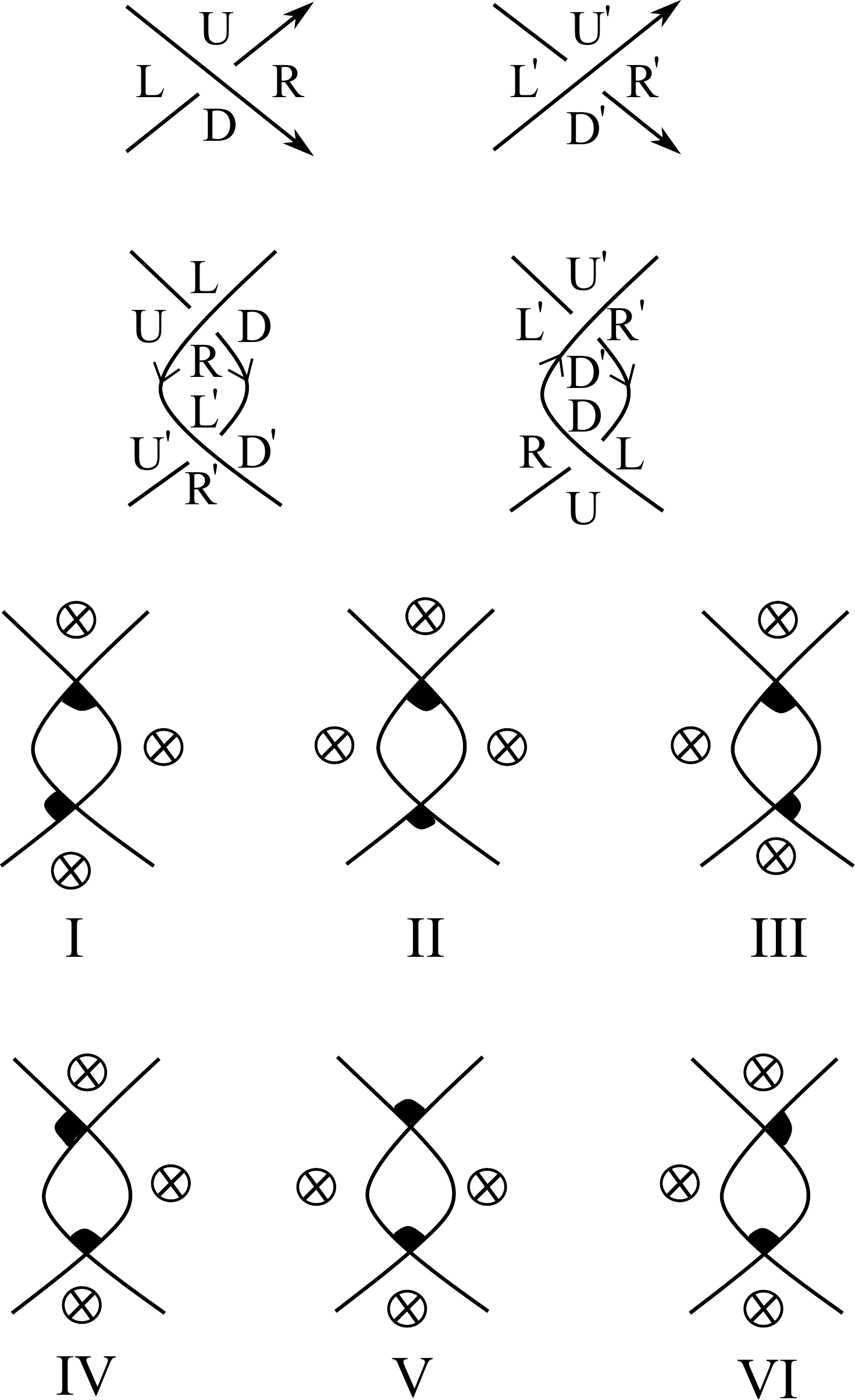}

\caption{General labeling at crossings and state configurations at an RII move site.}
\label{fig:apendix}
\end{figure}

If the RII move is of first type depicted on the left hand side of the second row in Figure \ref{fig:apendix}, then from states II and V  we have,\\
$$RR^{'} =1 = L L^{'}.$$
From the states I and IV, we have 
$$R U^{'} + U L^{`} = 0. $$
From state III and VI, we have
$$ RD^{`} + D L^{'} = 0.$$
If the RII move if the second type depicted on the right hand side of the second row in Figure \ref{fig:apendix}, then from states II and V, we have
$$D^{'} U =1 = U^{'} D.$$
From states I and IV, we have
$$RD^{'} + L^{'} D = 0.$$
From states III and VI, we have
$$D^{'} L + R^{'} D =0.$$

\begin{figure}[H]
\centering
\includegraphics[width=.35\textwidth]{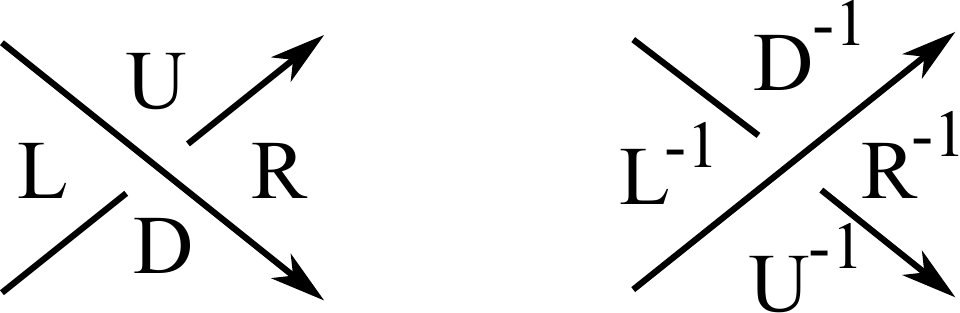}

\caption{Necessary labeling for invariance under an RII move.}
\label{fig:apendixx}
\end{figure}

By substituting the new labels into equations, we find
$$R = -UDL^{-1}$$ 
$$L= -UD R^{-1}.$$
Then by direct calculation, $$R= -(UD)^{-1} L^{-1}.$$
Therefore $(UD)^{2}= 1$. We assume $UD=1$, then the final labeling is as given in Figure \ref{fig:apendixxx}. This labeling is a generalization of the one we assumed for the potential summation throughout the paper with $R=W$.%, in the sense that $U,D$ are not necessarily equal to $1$.

\begin{figure}[H]
\centering
\includegraphics[width=.35\textwidth]{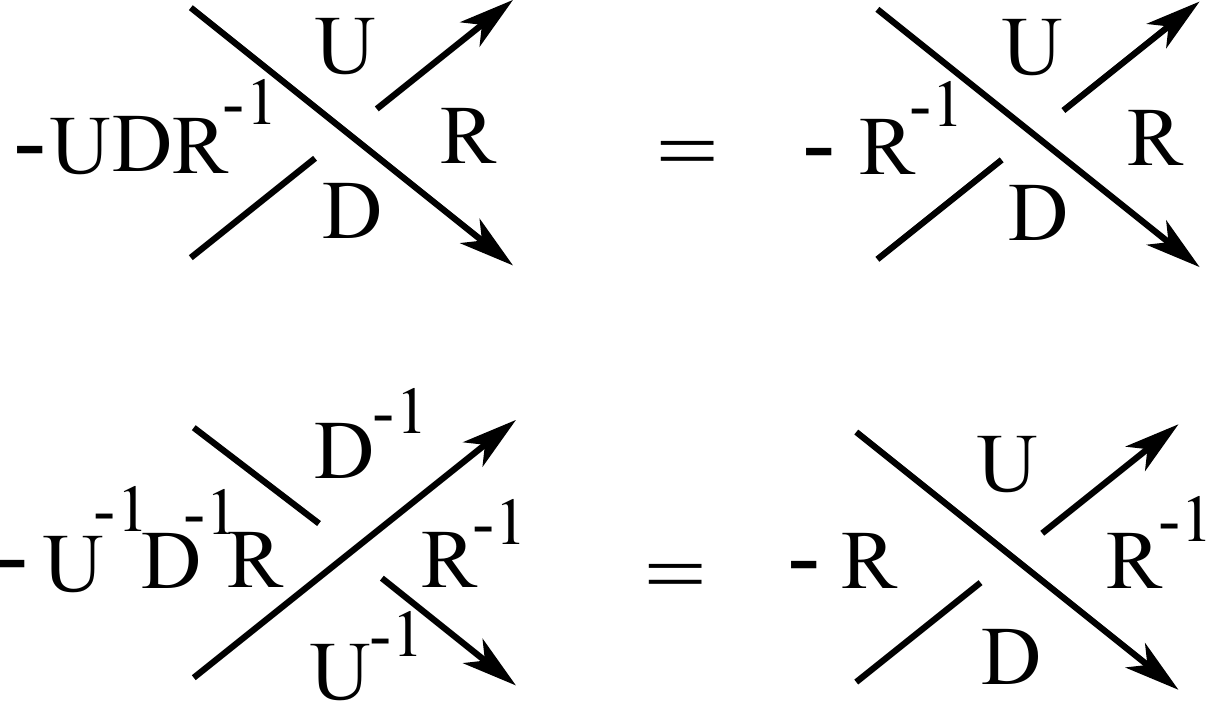}

\caption{Final labels at crossings.}
\label{fig:apendixxx}
\end{figure}

 Let $K$ be an oriented knot or knotoid diagram. Smoothing each crossing of $K$ in the oriented way results in a number of oriented circles (namely, Seifert circles) and a long segment if $K$ is a knotoid diagram. A Seifert circle that is oriented counterclockwise as a \textit{positive} circle and a circle that is oriented clockwise as \textit{negative} circle. 
 
 A generalization of the mock Alexander polynomial of $K$ with stars on a pair of its regions, can be defined as the state-sum over the labelings in Figure \ref{fig:apendixxx}. Let $\nabla_{K}(W,D)$ denote this state sum. Let $Deg(K) = p(K) - n(K)$, where $p(K)$ is the number of positive Seifert circles in an oriented smoothing of $K$ and $n(K)$ is the number of negative Seifert circles. Define $\tilde{\nabla}_{K}(W,D) = D^{-Deg(K)} \nabla_{K}(W,D)$. Then one can verify that $\tilde{\nabla}$ is an invariant of planar links and planar linkoids by using the same techniques  we used in Theorem \ref{thm:invariance}. We call $\tilde{\nabla}$ the \textit{planar potential of K}. We note here that the planar potential of a knot or link diagram with adjacent stars is a multiple of $\nabla_{K} (W)$ (the original potential function in variable $W$) by $D^{Deg(K)}$, but is more complex for other cases. The planar potential function will be discussed further in a subsequent paper. 
 
 Here we end with an example, given in Figure \ref{fig:last} that shows a knotoid $K$ and the knotoid $K!^{*}$ (obtained taking the reflection of $K$ through the vertical line in $\mathbb{R}^2$ and then taking the mirror image of the resulting knotoid by switching all the crossings) can be distinguished by the planar potential. More precisely, $Deg(K) = 0 = Deg (K!^{*})$ as the diagrams do not admit any Seifert circles and $\tilde{\nabla}_K (W, D ) =  W^2 + WD -W^{-1} D$ and $\tilde{\nabla}_{K!^{*}} (W, U)= W^2 + WU - W^{-1}U$, where $U =D^{-1}$.  Notice that $\nabla_{K} (W) = \nabla_{K!^{*}} (W)$ and also $K$ and $K!^{*}$ admit the same Jones polynomial. We have previously distinguished these two knotoids via an application of parity in \cite{GK2}. The method indicated here has general applicability due to the formula $\tilde{\nabla}_{K!^{*}} (W, D)=\tilde{\nabla}_K (W, U)$ holding for any knotoid $K.$\\
 
 \begin{figure}[H]
\centering
\includegraphics[width=.55\textwidth]{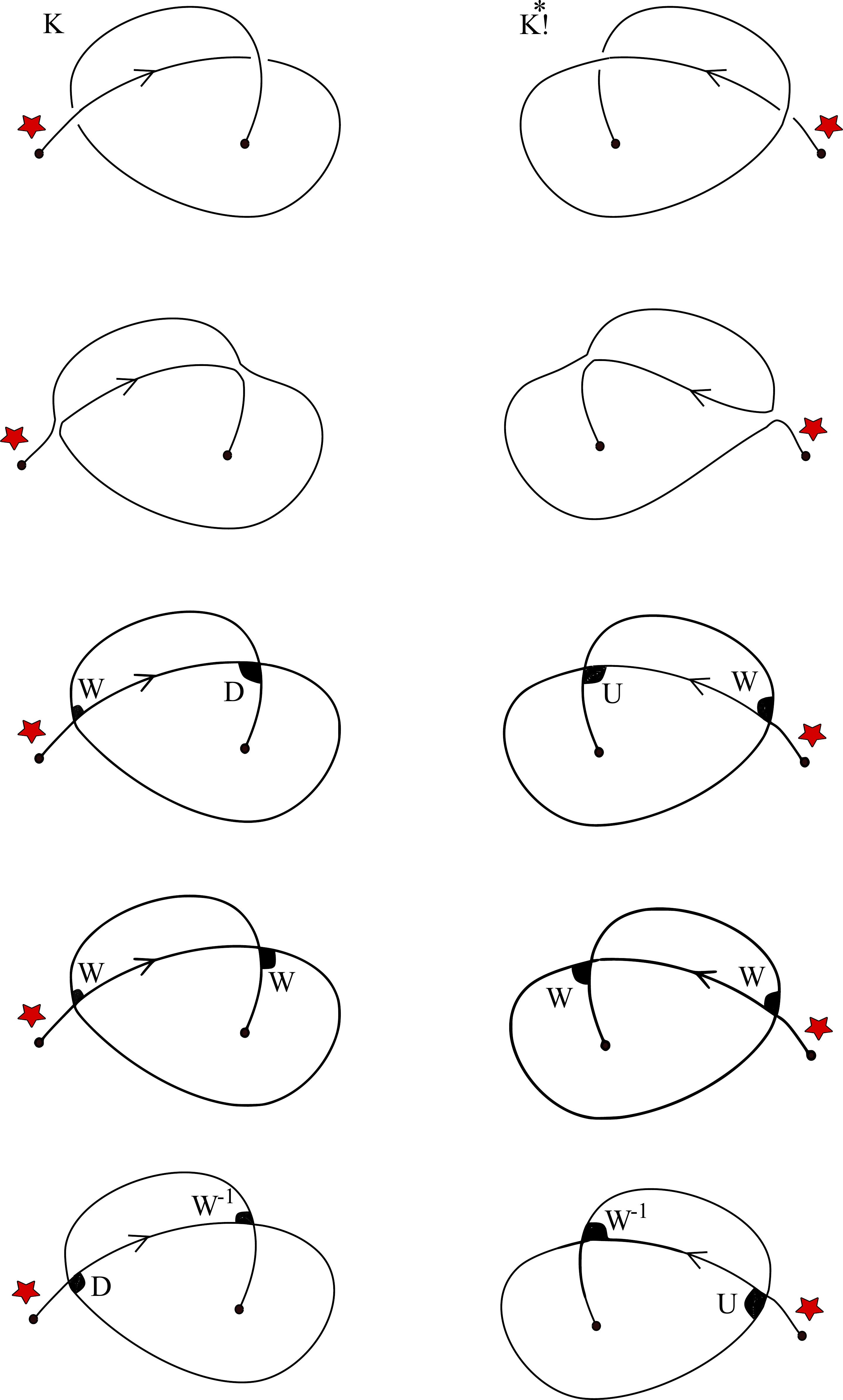}

\caption{A knotoid $K$ and its symmetric mirror image $K!^{*}$, and oriented smoothings of them}
\label{fig:last}
\end{figure}

\end{document}